\documentclass[11pt,reqno]{amsart}
\usepackage{amsmath, amsthm, amssymb,stmaryrd}
\usepackage{enumitem}

\usepackage{fullpage}
\usepackage[many]{tcolorbox}
\usepackage{xcolor}
\usepackage{wasysym}
\usepackage{url}
\usepackage{mathtools}
\usepackage{marvosym}
\usepackage[all,tips, 2cell]{xy}
\usepackage{amssymb}
\usepackage{mathrsfs}
\usepackage{array}
\usepackage{dsfont}
\usepackage{stackengine}
\newcommand\pesos{%
  \stackengine{-1.38ex}{P}{\stackengine{-1.2ex}{$-$}{$-$}{O}{c}{F}{F}{S}}{O}{c}{F}{T}{S}}

\usepackage{tikz}
\usetikzlibrary{cd}
\usetikzlibrary{calc}
\usetikzlibrary{decorations,decorations.pathreplacing,decorations.markings,decorations.pathmorphing}
\tikzset{squiggly/.style={decorate, decoration=snake}}

\tikzset{super thick/.style={line width=3pt}}
\tikzstyle{far>}=[decoration={markings, mark=at position 0.75 with {\arrow{>}}}, postaction={decorate}]
\tikzstyle{mid>}=[decoration={markings, mark=at position 0.55 with {\arrow{>}}}, postaction={decorate}]
\tikzstyle{mid<}=[decoration={markings, mark=at position 0.55 with {\arrow{<}}}, postaction={decorate}]
\tikzset{super thick/.style={line width=3pt}}
\tikzstyle{far>}=[decoration={markings, mark=at position 0.75 with {\arrow{>}}}, postaction={decorate}]
\tikzstyle{mid>}=[decoration={markings, mark=at position 0.55 with {\arrow{>}}}, postaction={decorate}]
\tikzstyle{mid<}=[decoration={markings, mark=at position 0.55 with {\arrow{<}}}, postaction={decorate}]
\tikzstyle{knot}=[preaction={super thick, white, draw}]
\tikzstyle{coupon}=[draw, very thick, rectangle, rounded corners=5pt]
\tikzset{Rightarrow/.style={double equal sign distance,>={Implies},->},
triplecd/.style={-,preaction={draw,Rightarrow}},
quadruplecd/.style={preaction={draw,Rightarrow,
shorten >=0pt
},
shorten >=1pt,
-,double,double
distance=0.2pt}}
\tikzset{
    tripleline/.style args={[#1] in [#2] in [#3]}{
        #1,preaction={preaction={draw,#3},draw,#2}
    }
}
\tikzstyle{triple}=[tripleline={[line width=.15mm,black] in
      [line width=.7mm,white] in
      [line width=1mm,black]}] 
\tikzset{
    quadrupleline/.style args={[#1] in [#2] in [#3] in [#4]}{
        #1,preaction={preaction={preaction={draw,#4},draw,#3}, draw,#2}
    }
}
\tikzstyle{quadruple}=[quadrupleline={[line width=.3mm,white] in
      [line width=.6mm,black] in
      [line width=1.2mm,white] in
      [line width=1.5mm,black]}]

\definecolor{violet}{RGB}{148,0,211}
\definecolor{DarkGreen}{RGB}{0,150,0}

\usepackage[pdftex,plainpages=false,hypertexnames=false,pdfpagelabels]{hyperref}
\definecolor{medium-blue}{rgb}{0,0,.8}
\hypersetup{colorlinks, linkcolor={purple}, citecolor={medium-blue}, urlcolor={medium-blue}}
\newcommand{\arxiv}[1]{\href{http://arxiv.org/abs/#1}{\tt arXiv:\nolinkurl{#1}}}
\newcommand{\arXiv}[1]{\href{http://arxiv.org/abs/#1}{\tt arXiv:\nolinkurl{#1}}}

\newcommand{\doi}[1]{\href{http://dx.doi.org/#1}{{\tt DOI:#1}}}

\DeclareMathOperator{\adj}{adj}

\DeclareMathOperator{\br}{br}

\DeclareMathOperator{\dom}{dom}
\DeclareMathOperator{\End}{End}
\DeclareMathOperator{\faith}{faith}
\DeclareMathOperator{\ff}{ff}

\DeclareMathOperator{\ob}{ob}
\DeclareMathOperator{\Ar}{ar}

\DeclareMathOperator{\loc}{loc}

\DeclareMathOperator{\gen}{gen}
\DeclareMathOperator{\Hom}{Hom}
\DeclareMathOperator{\id}{id}

\DeclareMathOperator{\Inv}{Inv}

\DeclareMathOperator{\ndeg}{ndeg}
\DeclareMathOperator{\pt}{pt}

\newcommand{\rev}{\mathsf{rev}}
\newcommand{\mop}{\mathsf{mop}} 
\renewcommand{\mp}{\mathsf{mp}}

\DeclareMathOperator{\sHom}{sHom}
\DeclareMathOperator{\sss}{ss}

\DeclareMathOperator{\str}{str}

\DeclareMathOperator{\MF2C}{\mathbf{MF2C}}

\DeclareMathOperator{\BMF1C}{\mathbf{BMF1C}}
\DeclareMathOperator{\Mor}{\mathbf{Mor}}
\DeclareMathOperator{\tKar}{\mathbf{2Kar}}
\DeclareMathOperator{\Cat}{\mathbf{Cat}}

\DeclareMathOperator{\Map}{\mathsf{Map}}

\DeclareMathOperator{\MF2C}{\mathbf{MF2C}}
\DeclareMathOperator{\Bimod}{\mathbf{Bimod}}
\DeclareMathOperator{\Glob}{\mathrm{Glob}}

\renewcommand{\Im}{\operatorname{Im}}
\newcommand{\Z}{\mathbb{Z}}

\DeclareMathOperator{\SH}{SH}

\newcommand\mhyphen{{\text{-}\!}}

\newcommand{\Alg}{\mathbf{Alg}}
\newcommand{\Aut}{\mathcal{A}ut}

\newcommand{\Mod}{\mathbf{Mod}}

\newcommand{\Spec}{\mathsf{Spec}}

\newcommand{\Vect}{\mathbf{Vect}}
\newcommand{\sVect}{\mathbf{sVect}}
\newcommand{\Rep}{\mathbf{Rep}}
\newcommand{\op}{\mathsf{op}}
\newcommand{\cV}{\mathcal{V}}

\newcommand{\Witt}{\mathcal{W}itt}

\def\semicolon{;}
\def\applytolist#1{
    \expandafter\def\csname multi#1\endcsname##1{
        \def\multiack{##1}\ifx\multiack\semicolon
            \def\next{\relax}
        \else
            \csname #1\endcsname{##1}
            \def\next{\csname multi#1\endcsname}
        \fi
        \next}
    \csname multi#1\endcsname}

\def\calc#1{\expandafter\def\csname c#1\endcsname{{\mathcal #1}}}
\applytolist{calc}QWERTYUIOPLKJHGFDSAZXCVBNM;
\def\bbc#1{\expandafter\def\csname bb#1\endcsname{{\mathbb #1}}}
\applytolist{bbc}QWERTYUIOPLKJHGFDSAZXCVBNM;
\def\bfc#1{\expandafter\def\csname bf#1\endcsname{{\mathbf #1}}}
\applytolist{bfc}QWERTYUIOPLKJHGFDSAZXCVBNM;
\def\sfc#1{\expandafter\def\csname s#1\endcsname{{\sf #1}}}
\applytolist{sfc}QWERTYUIOPLKJHGFDSAZXCVBNM;
\def\fc#1{\expandafter\def\csname f#1\endcsname{{\mathfrak #1}}}
\applytolist{fc}QWERTYUIOPLKJHGFDSAZXCVBNM;
\def\rmc#1{\expandafter\def\csname rm#1\endcsname{{\mathrm #1}}}
\applytolist{rmc}QWERTYUIOPLKJHGFDSAZXCVBNM;
\def\scc#1{\expandafter\def\csname sc#1\endcsname{{\mathscr #1}}}
\applytolist{scc}QWERTYUIOPLKJHGFDSAZXCVBNM;
\def\scc#1{\expandafter\def\csname sc#1\endcsname{{\mathscr #1}}}
\applytolist{scc}QWERTYUIOPLKJHGFDSAZXCVBNM;

\numberwithin{equation}{section}

\theoremstyle{plain}
\newtheorem{thm}[equation]{Theorem}
\newtheorem*{thm*}{Theorem}
\newtheorem{cor}[equation]{Corollary}
\newtheorem{lem}[equation]{Lemma}
\newtheorem{prop}[equation]{Proposition}

\newtheorem*{claim*}{Claim}

\newtheorem{thmalpha}{Theorem}

\newtheorem{coralpha}[thmalpha]{Corollary}
\theoremstyle{definition}
\newtheorem{defn}[equation]{Definition}

\newtheorem*{trick*}{Trick}

\newtheorem*{convention*}{Convention}
\newtheorem{conjecture}[equation]{Conjecture}

\newtheorem{ex}[equation]{Example}

\newtheorem{rem}[equation]{Remark}

\newtheorem{warn}[equation]{Warning}

\definecolor{theo}{HTML}{00B4CE}
\definecolor{dmitri}{HTML}{AF72B0}
\definecolor{peter}{HTML}{FF5733}
\newcommand{\thib}[1]{\textcolor{green!50!black}{[[Thib: #1]]}}

\newcommand{\dave}[1]{\textcolor{orange}{[[Dave: #1]]}}

\newcommand{\peter}[1]{\textcolor{peter}{[[Peter: #1]]}}

\title{The Classification of Fusion 2-Categories}
\date{\today}
\begin{document}
\author{Thibault D\'ecoppet}\address{\textnormal{\url{decoppet@math.harvard.edu}}, Harvard University}
\author{Peter Huston}
\address{\textnormal{\url{peter.e.huston@vanderbilt.edu}}, Vanderbilt University}
\author{Theo Johnson-Freyd}
\address{\textnormal{\url{theojf@pitp.ca}}, Perimeter Institute for Theoretical Physics and Dalhousie University}
\author{Dmitri Nikshych}
\address{\textnormal{\url{dmitri.nikshych@unh.edu}}, University of New Hampshire}
\author{David Penneys}
\address{\textnormal{\url{penneys.2@osu.edu}}, The Ohio State University}
\author{Julia Plavnik}
\address{\textnormal{\url{jplavnik@iu.edu}}, Indiana University Bloomington and Vrije Universiteit Brussel}
\author{David Reutter}
\address{\textnormal{\url{david.reutter@uni-hamburg.de}}, Universit\"at Hamburg}
\author{Matthew Yu}
\address{\textnormal{\url{yum@maths.ox.ac.uk}}, Oxford University}

\begin{abstract}
We classify (multi)fusion 2-categories in terms of braided fusion categories and group cohomological data. This classification is homotopy coherent --- we provide an equivalence between the 3-groupoid of (multi)fusion 2-categories up to monoidal equivalences and a certain 3-groupoid of commuting squares of $\mathrm{B}\mathbb{Z}/2$-equivariant spaces. Rank finiteness and Ocneanu rigidity for fusion 2-categories are immediate corollaries of our classification.
\end{abstract}
\maketitle

\tableofcontents

\newpage 

\section{Introduction}

Fusion 2-categories, first introduced in \cite{douglas2018fusion}, are finite semisimple  2-categories with a monoidal structure with duals for objects and simple monoidal unit. They categorify the notion of fusion 1-category, and have found relevance in the study of TQFTs both from the cobordism hypothesis \cite{Decoppet2022;centers, Decoppet2023;dualizable} and state-sum constructions \cite{douglas2018fusion}, braided fusion 1-category theory \cite{JFR}, condensed matter physics \cite{Inamura:2023qzl,Delcamp:2023kew}, as well as high energy physics \cite{BBFP,BBSNT,Decoppet:2022dnz,Bullimore:2024khm}. 

One point of physical interest in fusion higher category theory is that we now understand that many relevant physical theories in higher dimensions generally enjoy a symmetry that is given by a higher category where the objects, morphisms, and higher morphisms are topological operators that implement non-invertible or categorical symmetries. 
Most prominently, systems in (1+1)-spacetime dimensions, such as the critical Ising model, do enjoy a fusion 1-category of symmetries \cite{KWdefect,PhysRevLett.121.177203,MR3543452,MR4737369,2410.08884}. 
As for fusion 2-categories, they describe the fusion of 2-spacetime-dimensional surface operators, which are prevalent for many quantum field theories, especially those in higher dimensions.
Studying categorical symmetries has given new insights on how to study the phases of a quantum field theory \cite{Bhardwaj:2024qiv}, and, in certain cases, how to actually classify the theories \cite{JFY:2021tbq}. 
A classification of fusion 2-categories thus has many applications for physics and improving our understanding of the general structure of symmetries.

The main result of this article is that a fusion 2-category is completely described by the braided fusion 1-category of endomorphisms of its unit together with certain group theoretical and cohomological data. This is in sharp contrast to the situation with fusion 1-categories: While it is sometimes possibly to classify special examples and classes of fusion categories \cite{ostrik2003module,etingof2004classification,natale2018classification},
there are many ``exotic'' examples constructed from quantum groups at roots of unity \cite{MR0696688,MR936086,MR1090432,MR1660937} and examples arising from the theory of subfactors \cite{MR1686551,MR1832764,MR3167494,MR3166042,
MR4598730,
MR4565376}. In particular, while not every fusion 1-category is ``group-theoretical'' in the sense of \cite[Section 8.8]{ENO}, our main result can be interpreted as the statetement that every fusion 2-category is  ``braided fusion 1-category plus group'' theoretical.

\subsection{The Classification of Fusion 2-Categories}\label{subsection:DataProfits}

For any fusion 2-category $\mathfrak{C}$, there is a braided fusion 1-category $\Omega\mathfrak{C}$ of endomorphisms of the monoidal unit. 
Its \emph{symmetric center} (also called M\"uger center) $\mathcal{Z}_2(\Omega \mathfrak{C})$ is a symmetric fusion 1-category.  

\begin{defn}
    A fusion 2-category $\mathfrak{C}$ is \textit{bosonic} if $\mathcal{Z}_{2}(\Omega\mathfrak{C})$ is Tannakian, i.e.\ is equivalent to the symmetric monoidal category $\Rep(H)$ of finite-dimensional representations of a finite group $H$.
\end{defn}

\begin{thmalpha}[All bosons]
\label{thmalpha:AllBosons}
Bosonic fusion $2$-categories are parameterised by the following data:
\begin{itemize}
    \item A nondegenerate braided fusion $1$-category $\cA$;
    \item An inclusion of finite groups $\iota:H\hookrightarrow G$;
    \item A monoidal functor  $\rho:H\to\mathcal{A}ut^{\br}(\cA)$, where $\mathcal{A}ut^{\br}(\mathcal{A})$ denotes the $1$-groupoid of braided autoequivalences of $\cA$;
    \item A class $\pi\in \mathrm{H}^4(\mathrm{B}G,\mathbb{C}^\times)$; and
    \item A homotopy between the anomaly of $\rho$ and $\pi|_H$.
\end{itemize}
\end{thmalpha}

\begin{rem}\label{rem:vague}
The precise meaning of the words in Theorem~~\ref{thmalpha:AllBosons} and Theorem~\ref{thmalpha:EmergentFermions} below, and the sense in which this data ``parametrizes'' all fusion 2-categories is discussed in \S\ref{subsection:highlevelproof}.  For example, both $H$ and $G$ should really be considered connected $1$-groupoids instead of groups (they are the groupoids of fiber functors of $\cZ_{2}(\Omega \fC)$ and $\Omega \cZ(\fC)$, respectively; writing them as groups amounts to the unique, but up to non-unique isomorphism, choice of a  fiber functor). Accordingly,  the inclusion of groups $H\to G$ should really be considered a faithful functor between connected $1$-groupoids. In practice, this means that monoidal equivalences between fusion $2$-categories will induce isomorphisms $G \to G'$ so that after transporting the subgroup $H$ along this isomorphism to a subgroup of $G'$, it becomes \emph{conjugate} (but not necessarily equal) to the subgroup $H'$.  
\end{rem}
\noindent 
Such a list has appeared in the physics literature in the context of topological operators in (2+1)-dimensions \cite{Bullimore:2024khm}. It has also appeared in \cite{Bhardwaj:2024qiv} for the purpose of studying gapped $(2+1)$d phases, where the authors noticed that the category $\cA$ given in Theorem \ref{thmalpha:AllBosons} appears as non-minimal boundary conditions with $H$ global symmetry for the topological bulk theory in $(3+1)$d with gauge group $G$. The homotopy between the anomalies is also discussed therein.

A \emph{supergroup} $(H,z)$ is a group $H$ with a central element $z\in H$ of order $2$. There are analogous notions for higher categorical groups, introduced below.  
 We let $\Rep(H,z)$ denote the symmetric monoidal category with underlying monoidal category $\Rep(H)$ and braiding twisted by the action of $z$. 
Recall  from \cite{deligne2002} that any symmetric fusion category which is not Tannakian is \emph{super-Tannakian}, i.e.\ of the form $\Rep(H,z)$ for a finite supergroup $(H,z)$ with non-trivial $z$. The simplest example of such a super-Tannakian symmetric fusion 1-category is $\sVect$, the symmetric fusion 1-category of super-vector spaces, corresponding to the supergroup $(H,z) = (\mathbb{Z}/2, -1)$.

\begin{defn}
    A fusion 2-category $\mathfrak{C}$ is said to be \textit{fermionic} if $\mathcal{Z}_{2}(\Omega\mathfrak{C})$ is super-Tannakian.
\end{defn}

\begin{thmalpha}[Emergent Fermions]
\label{thmalpha:EmergentFermions}
Fermionic fusion 2-categories are parameterized by the following data:
\begin{itemize}
    \item An $\sVect$-nondegenerate braided fusion $1$-category $\cA$, i.e.\ a braided fusion 1-category $\cA$ equipped with a symmetric equivalence $\mathbf{sVect} \to \mathcal{Z}_{2}(\mathcal{A})$;
    \item An inclusion $\iota:(H,z)\hookrightarrow (G,z)$ of finite supergroups where $z$ is non-trivial;
    \item A monoidal functor of (higher) supergroups $\rho:(H,z)\to(\mathcal{A}ut^{\br}_{\sVect}(\cA),(-1)^f)$, where $\mathcal{A}ut^{\br}_{\sVect}(\cA)$ denotes the $2$-group of braided automorphisms $F$ of $\cA$ equipped with a monoidal natural isomorphism between the induced identifications of symmetric centers with $\sVect$, and ``superelement'' $(-1)^f$ (coherently) given by the identity functor equipped with the non-trivial natural isomorphism which acts by $(-1)$ on odd objects in $\sVect$;
    \item A class $\varpi$ in a certain torsor for the supercohomology group $\SH^{4}(\mathrm{B}(G,z))$; and
    \item A homotopy between the anomaly of $\rho$ and $\varpi|_{(H,z)}$.
\end{itemize}
\end{thmalpha}
\noindent 
While we have stated the cohomological data in the form of a torsor, this torsor can be trivialized by a choice of minimal nondegenerate extension for $\mathcal{A}$, which always exists thanks to the main result of \cite{JFR}.

As in Remark~\ref{rem:vague}, we give a more precise version of Theorem~\ref{thmalpha:EmergentFermions} below.

\begin{rem}
Allowing trivial $z$ in Theorem~\ref{thmalpha:EmergentFermions} (and leaving the rest of the data unchanged) in fact gives the classification in the general case, i.e.\ of arbitrary fusion $2$-categories. In particular, the $z=1$ case of Theorem~\ref{thmalpha:EmergentFermions} reduces (somewhat non-trivially) to Theorem~\ref{thmalpha:AllBosons}. See ~\S\ref{subsection:DelphicUnfold}.
\end{rem}
Given a fusion $2$-category $\mathfrak{C}$, we sketch how the data appearing in Theorem \ref{thmalpha:AllBosons} may be extracted, the fermionic case works analogously: 
\begin{itemize}
\item Let $\cZ(\mathfrak{C})$ denote the (2-categorical) Drinfeld center of $\mathfrak{C}$. We show  that the canonical symmetric monoidal functor $\Omega \cZ(\mathfrak{C}) \to \cZ_{2}(\Omega \mathfrak{C})$ is dominant. In particular, $\Omega \cZ(\mathfrak{C})$ is Tannakian if and only if $\cZ_{2}(\Omega \mathfrak{C})$ is and by \cite{deligne2002} the functor $\Omega \cZ(\mathfrak{C}) \to \cZ_{2}(\Omega \mathfrak{C})$ is equivalent to the restriction functor $\Rep(G) \to \Rep(H)$ along an inclusion of finite  groups $H \hookrightarrow G$. 
\item The braided fusion 1-category $\Omega \mathfrak{C}$ has symmetric center $\cZ_{2}(\Omega \mathfrak{C}) = \Rep(H)$ and hence is by the (de-)equivariantization theorem of \cite[Theorem 4.4]{DGNO2010braided} classified by the nondegenerate braided fusion category $\cA = \Omega \mathfrak{C} \boxtimes_{\cZ_{2}(\Omega \mathfrak{C})} \Vect$ together with a braided $H$-action (and $\Omega \mathfrak{C}$ can be recovered as the homotopy/categorical fixed points of that action on $\cA$).  Thus, the first three bullet points in Theorem~\ref{thmalpha:AllBosons} fully determine the braided fusion 1-category $\Omega \mathfrak{C}$ and the symmetric functor $\Omega \cZ(\mathfrak{C}) \to \cZ_{2}(\Omega \mathfrak{C})$. 
\item The last two bullet points in Theorem \ref{thmalpha:AllBosons} record more subtle cohomological information pertaining to the higher coherences of $\mathfrak{C}$. 
\end{itemize}

It is clarifying to understand where the most familiar examples of fusion 2-categories fit into the above classification theorems:

\begin{ex}
The case $H= G$ corresponds to the case of \emph{connected fusion 2-categories} $\mathfrak{C}$, i.e.\ those equivalent to $\Mod(\Omega \mathfrak{C})$, see \cite{douglas2018fusion}. Indeed, in this case the class $\pi$, resp. $\varpi$ is determined by $\rho$ through the homotopy and hence the last two bullet points are no data. The first three bullet points classify $\Omega \mathfrak{C}$ via de/equivariantization as outlined above. 
\end{ex}

\begin{ex}\label{ex:generalizedbosonicstronglyfusion}
     Let us consider the case when $H$ is the trivial group in Theorem \ref{thmalpha:AllBosons}. Then the only non-trivial data is a  nondegenerate braided fusion 1-category $\cA$, and a class $\pi$ in $H^4(G;\mathbb{C}^{\times})$. 
     The corresponding fusion 2-category is the 2-Deligne tensor product $\mathbf{2Vect}^{\pi}_{G} \boxtimes \mathbf{Mod}(\mathcal{A})$. 
\end{ex}

\begin{ex}
    The group-theoretical fusion 2-categories of \cite{DY;grouptheoretical} are recovered by taking $\mathcal{A}=\mathbf{Vect}$, with $H$ nontrivial, in Theorem \ref{thmalpha:AllBosons}. In particular, if $\cG$ is a categorical group, then we can consider the fusion 2-categories $\mathbf{2Rep}(\cG)$, of finite semisimple 2-representations of $\cG$ as introduced in \cite{Elgueta}, and $\mathbf{2Vect}_{\cG}$, of $\cG$-graded 2-vector spaces \cite{douglas2018fusion}. Let us write $H$ for the group of equivalence classes of objects of $\cG$, and $A$ for its abelian groups of morphisms, which inherit an action by $H$. We can therefore consider the semidirect product $\widehat{A}\rtimes H$ of the Pontrijagin dual of $A$ with $H$. Then the fusion 2-category $\mathbf{2Vect}_{\cG}$ corresponds to the inclusion $\widehat{A}\hookrightarrow \widehat{A}\rtimes H$, a class $\pi$ in $H^4(\widehat{A} \rtimes H;\mathbb{C}^{\times})$ derived from the Postnikov 3-cocycle for $\cG$ (see \cite[Equation 5]{DY;grouptheoretical} for an explicit formula at the level of cocycles), and a certain null-homotopy of the restriction $\pi|_{\widehat{A}}$ (in terms of the explicit cocycle-level formulas of \cite{DY;grouptheoretical}, the restriction is trivial on the nose and the null-homotopy is the one given by the trivial cochain).
     Likewise, the fusion 2-category $\mathbf{2Rep}(\cG)$ corresponds to the inclusion $H\hookrightarrow \widehat{A}\rtimes H$, with the same class $\pi$ and a certain null-homotopy of~$\pi|_H$.
\end{ex}

\subsection{Profits of the Classification}\label{section:profits}
The classification immediately leads to applications for the structure of fusion 2-categories. We obtain a rank finiteness result, whose decategorified version for arbitrary fusion 1-categories is currently not known, as well as an analogue of Ocneanu rigidity. In what follows, the \emph{rank} of a fusion 2-category $\mathfrak{C}$ is the product of the the number of simple objects of the braided fusion 1-category $\Omega\mathfrak{C}$ together with the number of connected components of $\mathfrak{C}$.

\begin{coralpha}\label{cor:rankfinite}{(Rank finiteness)}
Up to monoidal equivalence, there exists only finitely many fusion 2-categories of a given rank.
\end{coralpha}

\begin{coralpha}\label{cor:rigid}{(Ocneanu rigidity)}
Fusion 2-categories admit no non-trivial deformations.
\end{coralpha}

\begin{coralpha}\label{cor:components}
The connected components of a fusion 2-categories whose corresponding inclusion of finite (super)groups is $H\hookrightarrow G$ are in bijective correspondence with the double cosets $H\backslash G/H$. Moreover, this correspondence is compatible with the fusion rules.
\end{coralpha}
\noindent
These corollaries are made precise and proved in \S\ref{sec:proofsofprofits}.

\subsection{The Main Theorem: Relating Fusion 2-Categories and Delphic Squares}\label{subsection:highlevelproof}

In this section, we state our main theorem whose unpacking results in Theorems  \ref{thmalpha:AllBosons} and \ref{thmalpha:EmergentFermions}.

We will freely use the language of $\infty$-categories (also know as $(\infty,1)$-categories) as developed in~\cite{highertopos, higheralgebra}. In particular, below we will treat the terms ``space'' (thought of as a homotopy type) and ``$\infty$-groupoid'' synonymously, and similarly the terms ``$n$-truncated space'' (i.e space with vanishing homotopy groups in degrees $>n$) and ``$n$-groupoid.''

Consider the classifying space $\mathrm{B} \Z/2$ equipped with its group structure induced from the further delooping $\mathrm{B}^2 \Z/2:= K(\Z/2, 2)$. Equivalently, this `$2$-group' may be thought of as the monoidal $1$-groupoid with a single object, automorphisms $\Z/2$ and (unique) monoidal structure induced by the abelian-ness of $\Z/2$. It is a well-known observation that this $2$-group is equivalent to the monoidal groupoid $\Aut^{\br}(\sVect)$ of braided automorphisms of $\sVect$ and braided monoidal natural isomorphisms.

\begin{defn}\label{def:superspaces}
    A \textit{superspace} is a space equipped with an action by $\Aut^{\br}(\sVect) \simeq \mathrm{B} \Z/2$. A \emph{map of superspaces} is a $\mathrm{B}\Z/2$-equivariant map. 
\end{defn}

Given a group $G$, a $\mathrm{B}\Z/2$-action on the classifying space $\mathrm{B}G$ (equivalently, the $1$-groupoid with one object and automorphisms $G$) amounts precisely to a central element $z\in Z(G)$ of order at most $2$, i.e.\ a supergroup structure in the sense of \S\ref{subsection:DataProfits}.

The reason superspaces feature so prominently in this work stems from the following consequence of a theorem of Deligne~\cite{deligne2002}: every symmetric multifusion $1$-category is completely determined by its $1$-groupoid $\Spec(\cE)$ of linear symmetric monoidal functors $\cE\to \sVect$.  This groupoid $\Spec(\cE)$ admits a canonical action by $\Aut^{\br}(\sVect)$ by postcomposition, and hence is a supergroupoid. Moreover, $\cE$ is fusion if and only if there is a unique --- up to nonunique isomorphism --- such fiber functor, i.e.\ if $\Spec(\cE)$ is connected. In this case $\Spec(\cE)$ is of the form $\mathrm{B}G$ for some finite group $G$ and the superspace structure amounts to the familiar supergroup structure on $G$. 

We briefly recall two important examples of \emph{higher} superspaces: 
Firstly, there is the $2$-groupoid $\mathrm{BMF1C}^{\ndeg}(\sVect)$ of $\sVect$-nondegenerate braided fusion categories (that is braided fusion categories equipped with an equivalence between their symmetric centers and $\sVect$), and compatible braided equivalences between them.
Secondly, there is the $4$-groupoid $\Witt(\sVect)$ of $\sVect$-nondegenerate braided fusion categories and $\sVect$-Witt equivalences, etc., between them.
The set of isomorphism classes of objects in this latter $4$-groupoid is the super-Witt group introduced in~\cite{DNO2013}.
Both carry an $\Aut^{\br}(\sVect)$-action  given by precomposing the identification of the symmetric centers with $\sVect$, making both into higher super-groupoids. 
Moreover, there is an $\Aut^{\br}(\sVect)$-equivariant map $[-]: \mathrm{BMF1C}^{\ndeg}(\sVect) \to  \Witt(\sVect)$ sending a $\sVect$-nondegenerate braided fusion category to the corresponding object in $\Witt(\sVect)$.

We are now ready to state our main theorem:

\begin{thmalpha}\label{thm:Delphic}
If $\mathfrak{C}$ is a multifusion $2$-category,
the canonical symmetric monoidal functor $ \Omega \cZ(\fC) \to \cZ_{2}(\Omega \fC) $ is faithful and dominant and hence corresponds by Deligne's theorem to a functor of super-1-groupoids $\iota: \Spec(\cZ_{2}\Omega \fC) \to \Spec(\Omega \cZ(\fC))$ which is essentially surjective on objects and faithful\footnote{In homotopical terms, $\iota$ is $(-1)$-connected, i.e.\ surjective on $\pi_0$, and $0$-truncated, i.e.\ injective on $\pi_1$ (and an isomorphism on all higher homotopy groups which vanish here).}.

Moreover, for any faithful and dominant symmetric monoidal functor $\cE \twoheadrightarrow \cF$ between symmetric multifusion $1$-categories, there is an equivalence of $3$-groupoids between the $3$-groupoid 
\[
\left\{
\parbox{10cm}{\rm Multifusion 2-categories $\mathfrak{C}$ equipped with an identification\\\centering $[\Omega\mathcal{Z}(\mathfrak{C})\rightarrow \mathcal{Z}_{2}(\Omega\mathfrak{C})]\cong [\mathcal{E}\twoheadrightarrow\mathcal{F}]$}
\right\}
\]
and the space\footnote{Formally, this space is the pullback \[\Map\left(\Spec(\cF), \mathrm{BMF1C}^{\ndeg}(\sVect)\right)^{\mathrm{B} \Z/2} \times_{\Map\left(\Spec(\cF),\Witt(\sVect)\right)^{\mathrm{B}\Z/2}} \Map\left(\Spec(\cE), \Witt(\sVect)\right)^{\mathrm{B}\Z/2},\] where $\Map(-,-)^{\mathrm{B}\Z/2}$ denote spaces of $\mathrm{B}\Z/2$-equivariant maps.} of commuting squares of (unpointed!)\ superspaces 
\begin{equation}\label{eq:delphic}
\begin{tikzcd}
\Spec(\cF)
\arrow[d,"\iota"']
\arrow[r, dashed]
&
{\mathrm{BMF1C}^{\ndeg}(\mathbf{sVect})}
\arrow[d,"{[\protect{-}]}"]
\\
\Spec(\cE)
\arrow[r, dashed, "\varpi"']
\arrow[ur,Rightarrow,dashed, shorten <= 1em, shorten >= 1em,"\simeq"]
&
\mathcal{W}itt(\mathbf{sVect})\,.
\end{tikzcd}
\end{equation}
(For better readability, dashed morphisms are part of the data of an object of this $3$-groupoid while solid morphisms are fixed.)
\end{thmalpha}
We will henceforth refer to a square~\eqref{eq:delphic} as a \emph{Delphic square of type $\iota$.}

In the fusion case, where $\Spec(\Omega \cZ(\fC))$ and $\Spec(\cZ_{2}(\Omega \fC))$ are connected, the faithful functor $\iota$ amounts to our inclusion of supergroups $(H,z) \hookrightarrow (G,z)$ from Theorem~\ref{thmalpha:EmergentFermions}. Indeed, the parametrisation data given in Theorems \ref{thmalpha:AllBosons} and \ref{thmalpha:EmergentFermions} above arise from unpacking the contents of a Delphic square;\ the details  are given in \S\ref{subsection:DelphicUnfold}.

\begin{rem}\label{rem:subtlety}
Theorem~\ref{thm:Delphic} sheds light on a subtlety in the phrasing of Theorem~\ref{thmalpha:EmergentFermions}: 
Given a fusion $2$-category $\fC$, we may extract the braided fusion category $\Omega \fC$ with symmetric center $\cZ_{2}(\Omega \fC)$. Only after an (in Theorem~\ref{thmalpha:EmergentFermions}) implicit choice of a fiber functor $\cZ_{2}(\Omega \fC) \to \sVect$ (i.e.\ a basepoint of $\Spec(\cZ_{2}(\Omega \fC))$ may we extract an actual supergroup $(H,z)$ and an $\sVect$-nondegenerate braided fusion category $\cA$ as appears in Theorem~\ref{thmalpha:EmergentFermions}. Without making this arbitrary choice, we merely have a functor $\Spec(\cZ_{2}(\Omega \fC)) \to \mathrm{BF1C}^{\ndeg}(\sVect)$ from a connected but not pointed groupoid. 
\end{rem}

\subsection{From Witt squares to Delphic squares}
In order to prove Theorem \ref{thm:Delphic}, we first prove that the 3-groupoid of multifusion 2-categories is equivalent to the 3-groupoid of \emph{Witt squares}.

\begin{defn}
    Let $F:\cE\twoheadrightarrow\cF$ be a faithful dominant symmetric monoidal functor between two symmetric multifusion 1-categories. A \textit{Witt square}\footnote{Formally, given $\cE \twoheadrightarrow\cF$, the space of Witt squares of type $F$ is the pullback $\Witt(\cE) \times_{\Witt(\cF)} \mathrm{BMF1C}^{\ndeg}(\cF)$.} of type $F:\cE\twoheadrightarrow\cF$ is a commuting square of spaces of the form
    \begin{equation}
\label{eq:DmitriData}
\begin{tikzcd}
\pt
\arrow[r, dashed]
\arrow[d, dashed]
&
\mathrm{BMF1C}^{\ndeg}(\cF)
\arrow[d,"{[-]}"]
\\
\mathcal{W}itt(\cE)
\arrow[ur,dashed, Rightarrow,shorten <= 1em, shorten >= 1em,"\simeq"]
\arrow[r]
&
\mathcal{W}itt(\cF)\,.
\end{tikzcd}
\end{equation}
(Dashed morphisms are part of the data of a Witt square of type $F: \cE\twoheadrightarrow\cF $, while solid morphisms are given.)
\end{defn}
In \eqref{eq:DmitriData}, $\mathrm{BMF1C}^{\ndeg}(\cF)$ denotes the $2$-groupoid of $\cF$-nondegenerate braided multifusion 1-categories, that is, braided multifusion 1-categories equipped with an identification of their symmetric center with $\mathcal{F}$. On the other hand, $\Witt(\cF)$ denotes the $4$-groupoid of $\cF$-nondegenerate braided multifusion $1$-categories and $\cF$-Witt equivalences between them. The right vertical functor sends a category to its Witt class and the bottom horizontal functor is given by base change along $\cE \to \cF$. 

Thus, roughly speaking, the data of a Witt square amounts to a choice of $\cF$-nondegenerate braided multifusion $1$-category $\cD$ (up to braided equivalence!)\ together with the choice of an $\cE$-Witt class $[\cB]$ and an $\cF$-Witt equivalence between $[\cB \boxtimes_{\cE} \cF]$ and $[\cD]$.

A detailed discussion of data of this type may be found in \S\ref{subsection:Wittdata}. Furthermore, the precise construction of the bottom horizontal map in \eqref{eq:DmitriData} is given by the following result, which is proven in Corollary~\ref{cor:Witt}.

\begin{prop}\label{thm:Wittfunctor}
There is a functor $\mathcal{W}itt(-):\mathbf{SMF1C}^{\dom,\faith}\rightarrow \mathbf{Spaces}$ from  the $(2,1)$-category  $\mathbf{SMF1C}^{\dom,\faith}$ of symmetric multifusion 1-categories with dominant and faithful symmetric functors to  the $(\infty,1)$-category $\mathbf{Spaces}$ of spaces, sending a symmetric multifusion 1-category $\mathcal{E}$ to the associated Witt space $\mathcal{W}itt(\mathcal{E})$.
\end{prop}

As a consequence of Deligne's theorem, any symmetric multifusion $1$-category $\cE$ can be reconstructed from the supergroupoid $\Spec(\cE)$. Explicitly, $\cE$ is equivalent to the symmetric monoidal multifusion category $\mathrm{Fun}(\Spec(\cE), \sVect)^{\mathrm{B}\Z/2}$ of $\mathrm{B}\Z/2$-equivariant functors from $\Spec(\cE)$ to $\sVect$. If $\Spec(\cE)$ is a connected supergroupoid, i.e.\ of the form $B(G,z)$ (for possibly trivial $z$), this prescription recovers the category $\Rep(G,z)$ of representations of $G$ with braiding twisted by $z$. 
Recall that a space $X$ with a $\mathrm{B}\Z/2$-action may equivalently be described by the homotopy quotient $X\sslash \mathrm{B}\Z/2$ with its map $X\sslash \mathrm{B}\Z/2 \to \mathrm{B}^2\Z/2$, see~\S\ref{subsection:supergroups}.  Homotopy theoretically, the symmetric multifusion category $\cE$ is then the limit of the functor $\Spec(\cE)\sslash (\mathrm{B}\Z/2) \to \mathrm{B}^2 \Z/2 \simeq \mathrm{B} \Aut^{\br}(\sVect) \hookrightarrow \mathbf{SMF1C}$. 
Thus, the connection between Delphic squares (involving only $\Witt(\sVect)$) and Witt squares (involving $\Witt(\cE)$ for arbitrary symmetric multifusion categories $\cE$) then manifestly becomes a statement about the functor $\Witt(-)$ preserving certain limits. We relegate this point to future work \cite{DHJFPPRNY2}.

\begin{thm}[\cite{DHJFPPRNY2}]\label{thm:Wittlimits}
The functor $\mathcal{W}itt(\text{--}):\mathbf{SMF1C}^{\dom,\faith}\rightarrow \mathbf{Spaces}$ commutes with (homotopy) limits of finite $1$-groupoids.\footnote{\label{foot:sufficientlyfaithful}In fact we can do slightly better:\ We can also consider (homotopy) limits of finite 2-groupoids whose action is \textit{sufficiently faithful}.}
\end{thm}

\begin{cor}
For any faithful dominant symmetric tensor functor $F:\mathcal{E}\twoheadrightarrow\mathcal{F}$, there is an equivalence of spaces between the 3-groupoid of Witt squares of type $F:\mathcal{E}\twoheadrightarrow\mathcal{F}$ and the 3-groupoid of Delphic squares of type $\Spec(F):\Spec(\cF)\hookrightarrow \Spec(\cE)$.
\end{cor}

In order to finish the proof of Theorem \ref{thm:Delphic}, it will therefore suffice to establish the following satetement.

\begin{thmalpha}
\label{thm:Wittsquare}
Let $\cE \twoheadrightarrow\mathcal{F}$ be a faithful and dominant symmetric functor. Then there is an equivalence of 3-groupoids
$$
\left\{
\parbox{5.4cm}{\rm Multifusion 2-categories $\mathfrak{C}$ with\\ $[\Omega\mathcal{Z}(\mathfrak{C})\rightarrow \mathcal{Z}_{2}(\Omega\mathfrak{C})]\cong [\mathcal{E}\twoheadrightarrow\mathcal{F}]$}
\right\}
\cong
\left\{
\parbox{2.6cm}{\rm Witt squares\\ of type $[\mathcal{E}\twoheadrightarrow\mathcal{F}]$}
\right\}.
$$ 
\end{thmalpha}
In fact, in Theorem~\ref{thm:maintheorem} we make Theorem~\ref{thm:Wittsquare}  functorial in $(\cE \to \cF)$, and hence obtain a complete description of the $3$-groupoid of all multifusion $2$-categories and monoidal equivalences. 

The proof of this theorem will be the technical heart of this paper. 
The core of the proof, which can be found in~\S\ref{sec:proof}, can be outlined as follows:
\begin{itemize}
    \item The data of a multifusion $2$-category $\fC$ (up to monoidal equivalence) is equivalent to the data of the regular module $\fC_{\fC}$, seen as a $1$-morphism $\mathbf{2Vect} \nrightarrow \fC$ in the Morita category $\Mor^{\sss}$ of multifusion $2$-categories.
    \item  In Proposition~\ref{prop:MF2Cpullback}, we show that this $1$-morphism in $\Mor^{\sss}$ may be factored uniquely into the $1$-morphism $\Mod(\Omega\fC)_{\Mod(\Omega \fC)}$ --- whose data corresponds to the braided multifusion $1$-category $\Omega \fC$ --- followed by the $1$-morphism $_{\Mod(\Omega \fC)}\fC_{\fC}$. 
    \item That latter $1$-morphism is a $1$-morphism ${}_{\fD_1}\fM_{\fD_2}$ which has the property that the induced $2$-functor $\fD_1 \to \End_{\fD_2}(\fM)$ is fully faithful and such that $\fM$ is faithful as a $\fD_2$-module. 
    In Corollary~\ref{cor:pullbackWittLeftFib}, we prove that such $1$-morphisms are completely determined by the Morita class of their target $\fD_2$ together with a dominant, faithful symmetric functor out of $\cZ_{2}(\fD_2)$ into some symmetric multifusion category.
    \item Using the key result of~\cite{Decoppet2022;centers} that any multifusion $2$-category is Morita equivalent to a connected one, the data in the previous step can be recast in terms of Witt classes of braided multifusion $1$-categories and their interaction with symmetric functors, resulting in Theorem~\ref{thm:Wittsquare}.  
\end{itemize}

\subsection{Outline}  
In \S\ref{section:preliminaries}, we give the necessary background for the proofs of our main theorems. 
In particular, in \S\ref{subsection:highercat}, 
we review the higher categorical setup for flagged and enriched higher categories, fibrations, and (de)\-equivariantizations. 
In \S\ref{subsubsection:MorF2C}, we introduce the Morita 4-category of multifusion 2-categories $\Mor^{\sss}$ and, in \S\ref{subsubsection:Mor2ss}, we introduce the Morita 4-category of braided multifusion 1-categories $\Mor^{\sss}_2$. 
In \S\ref{subsection:adjectives}, we introduce and study the many technical classes of $1$-morphisms of the Morita 4-categories that will appear. 
In \S\ref{subsection:supergroups}, we review supergroups and superspaces.

Sections \ref{section:F2Cparametrization} and \ref{section:Wittsquares} describe the passage between each of the items in the boxes in Figure \ref{fig:layout}:
\begin{figure}[!ht]
\begin{tikzpicture}
    \tikzstyle{box} = [draw, rectangle, minimum width=2cm, minimum height=1cm, text centered]
    \node (box1) [box] at (0,0) {F2C data};
    \node (box2) [box] at (4,0) {Witt squares};
    \node (box3) [box] at (8,0) {Delphic squares};
     \node (box4) [box] at (12,0) {Theorems A and B};
    \draw[<->, shorten >=2mm, shorten <=2mm] (box1.east) -- node[above] {\$} (box2.west);
    \draw[<->, shorten >=2mm, shorten <=2mm] (box2.east) -- node[above] {\pesos} (box3.west);
    \draw[->, shorten >=2mm, shorten <=2mm] (box3.east) -- node[above] {\euro} (box4.west);
      \draw[->, shorten >=2mm, shorten <=2mm] (box4) edge[bend left=23] node[above]{\textsterling} (box1);
\end{tikzpicture}
\caption{This figure serves to display the different ways of presenting the data needed to define fusion 2-categories, and the fact that one can explicitly go between them.}
\label{fig:layout}
\end{figure}
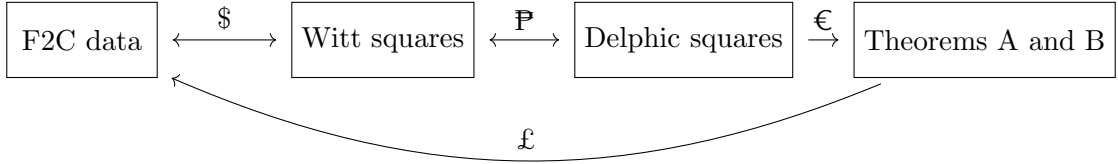

\begin{itemize}
    \item 
    The arrow labeled by (\$) is explained in \S\ref{section:F2Cparametrization} and gives the detailed proof of Theorem \ref{thm:Wittsquare}.
    \item 
    The arrow labeled by ($\pesos$) uses equivariantization to obtain Theorem \ref{thm:Delphic}. In particular, Proposition \ref{thm:Wittfunctor} is constructed in \S\ref{subsection:Wittfunctor} and the detailed unfolding of Witt data and Witt squares is given in  \S\ref{subsection:Wittdata} and  \S\ref{subsection:extractWittdata}.
    \item The arrow labeled by (\euro) \text{} represents the unpacking of the contents of Delphic squares and is presented in \S\ref{subsection:DelphicUnfold}.
    \item The arrow labeled by (\textsterling) \text{} uses the data of Theorems \ref{thmalpha:AllBosons} and \ref{thmalpha:EmergentFermions} to construct a fusion 2-category, and is explained in  \S\ref{subsection:reconstruction}.
\end{itemize}

\section{Preliminaries}\label{section:preliminaries}

\subsection{Glossary}\label{subsection:glossary}

Throughout, for definitiveness, we work over the field $\mathbb{C}$ of complex numbers. All our results remain valid over an arbitrary algebraically closed field of characteristic zero, and we will most often keep the field implicit in our notations. In particular, we will often suppress the word ``linear'' from expressions such as ``(braided) monoidal linear equivalence.'' 
Below, we include a glossary of the categories and groupoids which appear in this article.

\begin{itemize} 
\item We usually write categories in bold, to distinguish them from spaces and groupoids which we write in normal font. 
    \item $\Spec(\cE)$:\ the super 1-groupoid of fiber functors from the symmetric multifusion 1-category $\cE$ to super vector spaces.
    \item $\mathbf{BMF1C}$:\ the $(2,1)$-category of braided multifusion 1-categories and braided monoidal functors.
    \item $\mathbf{BMF1C}^{\adj}$:\ the $(2,1)$-category of braided multifusion 1-categories with ``adjectival'' 1-morphisms, where ``adjectival'' will be replaced by various technical adjectives.
    \item $\mathbf{SMF1C}^{\adj}$:\ the $(2,1)$-category of symmetric multifusion 1-categories and with ``adjectival'' 1-morphisms.
    \item $\mathbf{MF2C}$:\ the $(3,1)$-category of multifusion 2-categories and monoidal functors.
     \item $\mathrm{BMF1C}^{\ndeg}(\mathcal{E})$: the $2$-groupoid of $\mathcal{E}$-nondegenerate braided multifusion 1-categories, that is, braided multifusion 1-categories with symmetric center explicitly identified with the symmetric multifusion 1-category $\mathcal{E}$, and braided monoidal functors compatible with this identification.

    \item $\Mor^{\sss,\adj}$:\ the Morita 4-category of multifusion 2-categories with adjectival 1-morphisms.
    \item $\Mor^{\sss}_2$:\ the Morita 4-category of braided multifusion 1-categories.
    \item $\Mor^{\sss,\adj}_2$:\ the Morita 4-category of braided multifusion 1-categories, with adjectival 1-morphisms.
        \item $\mathrm{MF2C}(\cE\twoheadrightarrow\cF)$:\ the $3$-groupoid of multifusion 2-categories $\fC$ equipped with an identification of $\Omega\mathcal{Z}(\fC)\rightarrow\mathcal{Z}_{2}(\Omega\fC)$ with a given faithful dominant symmetric tensor functor $\cE\twoheadrightarrow\cF$.
    \item $\Mor^{\sss}$: the Morita 4-category of multifusion 2-categories.
    \item $\Witt(\cE)$:\ the $4$-groupoid of braided multifusion 1-categories with symmetric center identified with $\cE$ and $\cE$-Witt equivalences.
    \item $\mathbf{2Kar}$:\ the presentably symmetric monoidal $(3, 1)$-category of small $\mathbb{C}$-linear Karoubi complete 2-categories.
    \item $\ob(\mathcal{C})$:\ the space of objects in the $(\infty,1)$-category $\cC$.
    \item $\Ar(\cC)$:\ the space of arrows in the $(\infty,1)$-category $\cC$.
    \item $\mathrm{Glob}(\cC)$: the globular space associated to the $(\infty,1)$-category $\cC$.
    \item $\mathbf{Alg}_{E_n}(\cC)$:\ the $(\infty,1)$-category of $E_n$-algebras in the symmetric monoidal $(\infty,1)$-category $\cC$.
    \item $\mathbf{Bimod}_A(\mathcal{C})$:\ The (higher) category of bimodules for the $E_1$-algebra $A$ in the monoidal (higher) category $\mathcal{C}$.
\end{itemize}

\subsection{Categorical Background}\label{subsection:highercat}

We will freely use the language and theorems of $(\infty,1)$-categories as developed in~\cite{highertopos, higheralgebra}. 
Following standard usage, we will use the terms ``$\infty$-category" and ``$(\infty,1)$-category" synonymously. For notational simplicity, we will often  use the words ``unique,'' ``fiber'' and ``quotient'' to mean ``unique up to a contractible space of choices,'' ``homotopy fiber'' and ``homotopy quotient,'' etc., as these are the only interpretation of the words  that parse in the language of $\infty$-categories.

\subsubsection{Flagged Categories}\label{subsubection:flaggedcat}

As explained in the introduction, the Morita 4-categories $\mathbf{Mor}^{\sss}$ of multifusion 2-categories and $\mathbf{Mor}_2^{\sss}$ of braided multifusion 1-categories play essential roles in the proof of our main theorem. In the following subsections \S\ref{subsubsection:MorF2C} and \S\ref{subsubsection:Mor2ss}, we will define them as \textit{flagged} sub-4-categories. But in order to do so, we will crucially need to appeal to Proposition \ref{prop:belief} below, which is a technical result that can be used to define subcategories of a flagged $n$-category. To work up to and prove this propositon, we review the theory of flagged categories, and enriched $\infty$-categories. The enrichment will allow us to give a conceptual definition for higher flagged categories via induction.

In the remainder of this section, by ``sub object,'' we always mean a monic\footnote{A morphism $f:x\to y$ in an $\infty$-category $\cC$ is called \emph{monic} if for every object $a\in \cC$ the induced map of spaces $\Map_{\cC}(a, x) \to \Map_{\cC}(a,y)$ is a full subspace inclusion.} (aka a $(-1)$-truncated) 1-morphism in the appropriate $\infty$-category. For example, a subcategory is a monic functor in $\mathbf{Cat}_{\infty}$, the $\infty$-category of $\infty$-categories. Unpacked, a functor is monic when it is faithful and isofull, i.e.\ full on isomorphisms, or equivalently if it induces a full subspace inclusion (aka $(-1)$-truncated map) on spaces of objects and spaces of morphisms. We emphasize that this is not quite the notion of ``subcategory" used, for example, in \cite{maclane} (rather it corresponds to what is classically sometimes called a ``replete subcategory'' or an ``isomorphism-closed subcategory''). This change of usage is appropriate:\ a 1-category in the sense of \cite{maclane} comes with a set of objects, but this set cannot be extracted naturally when thinking of a 1-category as an object of $\mathbf{Cat}_{\infty}$, because equivalent 1-categories have different sets. Rather, a 1-category in the sense of \cite{maclane} is an example of a \emph{flagged 1-category}: a 1-category $\cC$ equipped with a 1-groupoid $\cC_0$, interpreted as ``the 1-groupoid of objects in $\cC$,'' and an essentially surjective functor $\cC_0 \to \cC$. Specifically, a 1-category in the sense of \cite{maclane} is a 1-category flagged by a set. A flagged 1-category is \emph{univalent} when the map $\cC_0 \to \cC$ is an equivalence on spaces of objects\footnote{A 1-category in the sense of \cite{maclane} is only univalent when it is \emph{gaunt}, that is, when the only isomorphisms are the identity morphisms.}, and the \emph{univalification} of a flagged 1-category $\cC_0 \to \cC$ is the result of replacing $\cC_0$ by the full 1-groupoid of objects in $\cC$. A subcategory in the sense of \cite{maclane} is an example of a sub flagged 1-category: a functor which is faithful and monic on 1-groupoids of objects. In particular, the next result follows from the discussion above.

\begin{lem}\label{lem:isofull}
   Let $\cD \hookrightarrow \cC$, be a sub flagged $1$-category of the flagged $1$-category $\cC$. Upon univalification
  of both $\cC$ and $\cD$, 
   $\cD \hookrightarrow \cC$ is a sub category of $\cC$ if and only if it is isofull.
\end{lem}

We now review how this story is implemented for higher categories, essentially following~\cite{ayala2018flagged}.

\begin{defn}\label{def:flaggedcat}
    A \emph{preflagged $(\infty,p)$-category} is a chain
    $\cC_0 \to \cC_1 \to \dots \to \cC_p$
    of $(\infty,p)$-categories. A preflagged $(\infty,p)$-category is \emph{flagged} if for each $k$, $\cC_k$ is an $(\infty,k)$-category and the functor $\cC_k \to \cC_{k+1}$ is essentially surjective on $(\leq k)$-morphisms. 
\end{defn}

\begin{defn}
     A flagged $(\infty,p)$-category is \emph{univalent} if, in addition, the functor $\cC_k \to \cC_{k+1}$ is fully faithful on $(>k)$-isomorphisms for all $0\leq k < p$.
\end{defn}
A univalent flagged $(\infty,p)$-category is entirely determined by $\cC_p$, with each $\cC_k$ the maximal sub-$(\infty,k)$-category of $\cC_p$.  The full inclusion
\[ 
\{\text{univalent flagged categories}\} 
\hookrightarrow 
\{\text{flagged categories}\} 
\]
has a left adjoint called \emph{univalification}. Realizing flagged categories is in general simpler than realizing univalent ones, and there are many examples of higher categories that are naturally flagged, see e.g.\ \cite{ayala2018flagged,ferrer2024dagger}.

Any flagged $(\infty,n)$-category $\mathcal{C}$ supplies a unital globular space $\Glob(\cC)$ given by
\begin{equation}
    \begin{tikzcd}
       \mathfrak{X}:= \mathfrak{X}_0 \arrow[r]& \arrow[l,shift right,shift right, "s",swap  ] \arrow[l,shift left, shift left,"t"]\mathfrak{X}_1 \arrow[r] & \mathfrak{X}_2  \arrow[r] \arrow[l,shift right,shift right,"s",swap ] \arrow[l,shift left, shift left,"t"]&\ldots \arrow[r] \arrow[l,shift right,shift right,"s",swap ] \arrow[l,shift left, shift left,"t"] & \mathfrak{X}_n, \arrow[l,shift right,shift right,"s",swap] \arrow[l,shift left, shift left,"t"]
    \end{tikzcd}
\end{equation}
where $\mathfrak{X}_k$ denotes the space of $k$-morphism.
The maps $s$ and $t$ denote the source and target maps. Given two $k$-morphisms $m,m' \in \mathfrak{X}_k$ we say that they are \textit{composable} if $t(m)=s(m')$.

It is often possible, and often convenient, to recast $(\infty,n)$-categorical notions in terms of iterated enrichment. For example, an $(\infty,n)$-category is precisely an $\infty$-category enriched (in the sense of~\cite{gepner2015enriched}) in the $\infty$-category of $(\infty,n-1)$-categories~\cite{1312.3881}. We will do this style of recasting for flagged $(\infty,p)$-categories, recognizing them as flagged $\infty$-categories enriched in the $\infty$-category of flagged $(\infty,p-1)$-categories.

We very briefly sketch the relevant definitions of enriched $\infty$-category theory, as developed in~\cite{gepner2015enriched}, here following the equivalent approach of~\cite{hinich, heine}. Let us fix $\mathcal{V}$, a presentably symmetric monoidal $(\infty,1)$-category. We introduce the notion of a flagged $\cV$-enriched $\infty$-category.

\begin{defn}  
     Let $\fX_0$ be a space. 
     The $\infty$-category of \textit{$\cV$-graphs $\fX$ over $\fX_0$} is the $\infty$-category $\mathrm{Fun}(\fX_0 \times \fX_0, \cV)$. 
\end{defn}

The $(\infty,1)$-category of $\cV$-graphs over a fixed space $\fX_0$ admits a natural presentably monoidal structure (that is however not braided), see e.g.~\cite{hinich, heine}. 
For example, the monoidal product $\fX_1\boxtimes \fX_1'$ of two $\cV$-graphs $\fX_1$, $\fX_1'$ is given by the functor $\fX_1\boxtimes \fX_1^{'}: \fX_0\times\fX_0\rightarrow \mathcal{V}$ defined by
\[(s,t)  \mapsto \int^{x\in \fX_0} \fX_1(s,x) \otimes_\cV \fX_1(x,t) = \underset{x \in \fX_0}{\operatorname{colim}} \; \fX_1(s,x) \otimes_\cV \fX_1(x,t) \,, \]
where $(s,t)$ ranges over $\fX_0 \times \fX_0$. The identity object in the space of $\cV$-graphs over $\fX_0$ is the graph $(s,t)  \mapsto \underset{x \in \fX_0(s,t)}{\operatorname{colim}} \mathds{1}$ where $\mathds{1}$ is the unit object in $\cV$.

\begin{defn}
         A \textit{globular $\cV$-space} consists of a space $\fX_0$ and an $E_0$-algebra $\fX_1$ in the monoidal $\infty$-category of $\cV$-graphs over $\fX_0$. A \textit{$\mathcal{V}$-enriched flagged category} $\fX$ is a space $\fX_0$ (henceforth referred to as the \emph{space of objects} of $\fX$) and an $E_1$-algebra $\fX_1$ in the monoidal $\infty$-category of $\mathcal{V}$-graphs over $\fX_0$. For any $\mathcal{V}$-enriched flagged category $\fX$, we use $\mathrm{Glob}(\fX)$ to denote taking the underlying globular $\mathcal{V}$-space. 
\end{defn}

\begin{ex}
    Unfolding the definitions (and using the equivalence of (univalent) $(\infty,n)$-categories with (univalent) $\infty$-categories enriched in $(\infty,n-1)$-categories~\cite{1312.3881}), we find that $n$-level globular spaces can be identified with globular spaces enriched in $(n-1)$-level globular spaces, and that
    flagged $(\infty,n)$-categories can be identified with flagged $\infty$-categories enriched in flagged $(\infty,(n-1))$-categories.
\end{ex}

For any $\infty$-category, the forgetful functor $\Alg_{E_1}(\cC) \to \cC$ from the $\infty$-category of $E_1$-algebras in $\cC$ preserves and reflects monomorphisms; i.e.\ a morphism of $E_1$-algebras is a monomorphism exactly when its underlying morphism of objects is a monomorphism.
In particular, a morphism $f : \fX \to \fZ$ of flagged $\cV$-categories is a mono (i.e.\ gives a sub flagged $\mathcal{V}$-category) exactly when it is a mono on the spaces of objects and if for any two objects $x_1, x_2 \in \fX_0$, the induced morphism  $\Hom_{\fX}(x_1,x_2) \to \Hom_{\fZ}(fx_1, fx_2)$ is a monomorphism in $\cV$.

We first prove the following straightforward observation: 
Given a monomorphism $A\hookrightarrow B$ in a monoidal $\infty$-category $\cD$ and suppose that $B$ is equipped with an $E_1$-algebra structure. Then there is a (unique) compatible $E_1$-algebra structures on $A$ if and only if the multiplication $A\otimes A \to B \otimes B \to B$ factors through the subobject $A\hookrightarrow B$. 

To make this precise, let $A_2$ denote the operad representing a unital binary multiplication (without any associativity); this is part of the familiar $A_n$-filtration of the $E_1$-operad $E_0 = A_0 \to A_1 \to A_2 \to \ldots \to E_1$. Then the above observation amounts to the following: 

\begin{lem}
    For an $\infty$-category $\cD$, let $\mathbf{Ar}(\cD)^{\mathrm{mono}}$ denote the full subcategory of the arrow category $\mathbf{Ar}(\cD)$ on the monomorphisms. Then the following is a pullback square of spaces (where $\mathrm{Alg}_{\cO}(-)$ denotes the space of $\cO$-algebras for an operad $\cO$): 
    \[
    \begin{tikzcd}
    \mathrm{Alg}_{E_1}(\mathbf{Ar}(\cD)^{\mathrm{mono}}) \arrow[r] \arrow[d] & \arrow[d] \mathrm{Alg}_{E_1}(\cD) \\
   \mathrm{Alg}_{A_2}(\mathbf{Ar}(\cD)^{\mathrm{mono}}) \arrow[r]&\mathrm{Alg}_{A_2}(\cD)
    \end{tikzcd}
    \]
\end{lem}
In words, an $E_1$-structure on a monomorphism $A\to B$ amounts to an $E_1$-structure on $B$ and a compatible $A_2$-structure on $A$. 
\begin{proof}
    The functor $\mathbf{Ar}(\cD)^{\mathrm{mono}} \to \cD$ is easily seen to be faithful, i.e.\ induces $(-1)$-truncated maps on mapping spaces. On the other hand, the map of operads $A_2 \to E_1$ is surjective on colors and $(-1)$-connected on multi-hom spaces, and is thus in the language of~\cite[Definition 7.5.1]{soergel} $0$-surjective. The lemma is then an immediate consequence of the orthogonality of the $0$-surjective, $0$-faithful factorization system on the $\infty$-category of operads~\cite[Proposition 7.5.3]{soergel}. 
\end{proof}

Applied to $E_1$-algebras in categories of $\cV$-graphs, this immediately implies the following general fact about $\cV$-enriched $\infty$-categories.

\begin{cor}\label{prop:enrichedV}
     Suppose that $\fZ = (\fZ_0, \fZ_1)$ is a $\mathcal{V}$-enriched flagged $\infty$-category, and that $f : \fX = (\fX_0,\fX_1) \hookrightarrow \Glob(\fZ)$ is a $\mathcal{V}$-enriched sub globular space. If, for any $a,b,c \in \fX_0$, the composition $\fX_1(a,b) \otimes \fX_1(b,c) \to \fZ_1(a,c)$, as depicted in the square below
\begin{equation}\label{eqn:ZfactorsthroughX}
    \begin{tikzcd}
        \fX_1(a,b) \otimes \fX_1(b,c) \arrow[r] \arrow[dotted]{d}& \fZ_1(a,b) \otimes \fZ_1(b,c) \arrow[d]\\
        \fX_1(a,c) \arrow[hookrightarrow]{r} &\fZ_1(a,c)\,,
    \end{tikzcd}
\end{equation}
factors through the subobject $\fX_1(a,c) \hookrightarrow \fZ_1(a,c)$, then $\fX = \Glob(\widetilde{\fX})$ for a unique sub flagged $\mathcal{V}$-enriched $\infty$-category $\widetilde{\fX}\subset \fZ$.
\end{cor}

We will apply the previous result to obtain an analogous statement for flagged $(\infty,n)$-categories --- this statement is regularly used, often implicitly, but we could not find a careful proof in the literature and hence have supplied one here. 

\begin{prop}\label{prop:belief}
    Let $\cC$ be a flagged $(\infty,n)$-category with underlying globular space $\Glob(\cC)$. Suppose that there exists a sub globular space $f:\mathfrak{X} \hookrightarrow \Glob(\cC)$ with the property that for every $k$ and for any composable pair of $k$-morphisms $x,x'\in \mathfrak{X}$, their composition is also in $\mathfrak{X}$.\footnote{When say that the ``composition is in $\fX$'' we mean it for all of the $k$ different directions of composition.} Then $\fX = \Glob(\mathcal{D})$ for a unique sub flagged $n$-category $\mathcal{D}\subseteq \mathcal{C}$.
 \end{prop}
 
\begin{proof}
    The globular space $\Glob(\cC)$ for a flagged $(\infty,n)$-category $\cC$ looks like 
    \begin{equation}
        \begin{tikzcd}
              \fZ_0 \arrow[r]& \arrow[l,shift right,shift right, "s",swap  ] \arrow[l,shift left, shift left,"t"] \Glob(\cC_{n-1})\,,
        \end{tikzcd}
    \end{equation} 
    where $\cC_{n-1}$ is some flagged $(\infty,(n-1))$-category. The base case of $n=1$ in the proposition was already established in Corollary \ref{prop:enrichedV}. We now apply the induction step: by assumption we give ourselves an inclusion $\fX \hookrightarrow \Glob(\cC)$ such that $\fX_0 \hookrightarrow \mathfrak{Z}_0$ and $\Glob(\widetilde{\mathcal{C}}_{n-1})\hookrightarrow f^*_0\Glob(\cC_{n-1})$, for some other flagged $(n-1)$-category $\widetilde{\cC}_{n-1}$. We are also given that compositions in $\Glob(\widetilde{\cC}_{n-1})$ lift to compositions in $\fX$, and therefore $\Glob(\widetilde{\mathcal{C}}_{n-1})$ enhances to a sub flagged category of $f^*_0\cC_{n-1}$. We can apply Corollary \ref{prop:enrichedV} to $\fX=(\fX_0,\Glob(\widetilde{\cC}_{n-1}))$ because any enriched flagged category is $E_1$ in the category of $\cV$-graphs over its space of objects $\fX_0$. Thus $\Glob(\widetilde{\cC}_{n-1})$ is $E_1$ in the category of $\cV$-graphs over $\fX_0$, and in particular is $E_0$.
    Therefore we get $\fX=\Glob(\widetilde{\cC})$ for a unique flagged $n$-category $\widetilde{\cC} \subseteq {\cC}$. This concludes the proof.
\end{proof}

The last proposition above guarantees that a sub globular space of an $n$-category defines a sub-$n$-category, and, in particular, it inherits its coherence data from the ambient $n$-category, provided that we can show that the morphisms of this sub globular space are closed under composition at all levels.

\subsubsection{Fibrations}
We recall some notions from \cite[\S2]{highertopos} that will appear in \S\ref{section:Wittsquares}. More precisely, we review the correspondence between sufficiently nice functors of $\infty$-categories $\mathcal{E}\rightarrow\mathcal{C}$ and functors $\mathcal{C}\rightarrow \mathbf{Spaces}$. Throughout, $\cC$ and $\cD$ denote two $\infty$-categories. Recall that $\ob(\cC)$ denotes the space of objects and $\mathrm{ar}(\cC)$ the space of morphisms in $\cC$.

\begin{defn}
A functor $p:\cC \to \cD$ is a \emph{left fibration} if the square
\[
\begin{tikzcd}
    \mathrm{ar}(\cC) \arrow[r, "s"] \arrow[d] & \ob(\cC) \arrow[d] \\
    \mathrm{ar}(\cD) \arrow[r, "s"'] & \ob(\cD) 
\end{tikzcd}
\]
is a pullback square, where both horizontal maps are taking the source of an arrow. 
\end{defn}
Intuitively, a left fibration is therefore a functor $p: \cC \to \cD$ such that for every arrow $g:d \to d'$ in $\cD$, object $c\in \cC$ and identification $p(c) \simeq d$, there exists a contractible space of arrows $f:c\to c'$ with $p(f) \simeq g$.

Let $\mathbf{LFib}(\cC)$ denote the full subcategory of the over-category $(\mathbf{Cat}_{\infty})_{/\cC}$ on the left fibrations. 
It follows from the definition that the fibers of a left fibration $p:\cC\rightarrow\cD$  are spaces. Moreover, this construction is functorial and therefore defines a functor $\cD\rightarrow\mathbf{Spaces}$. Heuristically, this construction can in fact be reversed via the Grothendieck construction. More precisely, it was established in \cite[Proposition 2.2.3.11]{highertopos} that
\[
\mathbf{LFib}(\cC)  \simeq \mathrm{Fun}(\cC, \mathbf{Spaces}).
\]
This equivalence is called the straightening/unstraightening equivalence .

Throughout we will repeatedly use the following lemma which follows immediately from the definition:
\begin{lem}
Given a pullback square of $\infty$-categories
\begin{equation}
\begin{tikzcd}
\widetilde{\mathcal{C}} \arrow[r] \arrow[d, "\widetilde{F}"'] & \mathcal{C} \arrow[d, "F"] \\
\widetilde{D} \arrow[r]                     & \mathcal{D}.          
\end{tikzcd}
\end{equation}
if $p: \cC \to \cD$ is a left (right) fibration, then so is $\widetilde{p}: \widetilde{\cC} \to \widetilde{\cD}.$
\end{lem}

\subsection{Multifusion 2-Categories}\label{subsection:MF2Creview}

Recall from~\cite{douglas2018fusion} that a finite semisimple $2$-category is a $\mathbb{C}$-linear $2$-category all of whose endomorphism categories are multifusion, which has $2$-categorical direct sums, and all of whose $2$-idempotents split. A multifusion $2$-category is a rigid monoidal finite semisimple $2$-category. 

We will now give a rigorous construction of the $(3,1)$-category of multifusion $2$-categories and monoidal $\mathbb{C}$-linear functors.

We call a $\mathbb{C}$-linear $1$-category \emph{Karoubian} if it has finite direct sums and if all idempotents split. We let $\mathbf{Kar}$ denote the presentable  $(2,1)$-category of Karoubian $1$-categories. Equipped with the Karoubi-completion-of-the-enriched-tensor product, $\mathbf{Kar}$ is in fact a presentably symmetric monoidal $(2,1)$-category, i.e.\ it has a symmetric monoidal structure $\boxtimes$ so that $\boxtimes$ preserves small colimits in both slots.

\begin{defn}We let $\mathbf{2Vect} \subseteq \mathbf{Kar}$ denote the full subcategory on the finite semisimple linear categories. 
\end{defn}
Note that finite semisimple linear categories are closed under the tensor product in $\mathbf{Kar}$ and hence $\mathbf{2Vect}$ in fact defines a symmetric monoidal subcategory. 

\begin{defn} We define the $(2,1)$-category  $\mathbf{BMF1C}$ as the full subcategory of $\mathbf{Alg}_{E_2}(\mathbf{2Vect})$ on the braided multifusion $1$-categories. Similarly, the $(2,1)$-category $\mathbf{SMF1C}$ is the full subcategory of $\mathbf{Alg}_{E_{\infty}}(\mathbf{2Vect})$ (and in fact a full subcategory of $\mathbf{BMF1C}$) on the symmetric multifusion $1$-categories. 
\end{defn}

We now move one categorical level higher. 

We say that a $\mathbf{Kar}$-enriched category is \emph{$2$-Karoubian} if it has $2$-categorical direct sums and if $2$-idempotents split.
We now construct the $(3,1)$-category $\mathbf{2Kar}$ of $2$-Karoubian linear $2$-categories and show that it is presentably symmetric monoidal. It was shown in \cite{gepner2015enriched} that if $\cV$ is a presentably symmetric monoidal $\infty$-category, then there is a presentably symmetric monoidal $\infty$-category $\Cat_\cV$ of $\cV$-enriched categories and $\cV$-enriched functors. Taking $\cV = \mathbf{Kar}$, the $(2,1)$-category of Karoubi complete $\mathbb{C}$-linear 1-categories, provides the presentably symmetric monoidal $(2,1)$-category $\mathbf{2LocKar} := \Cat_{\mathbf{Kar}}$.

\begin{defn}
\label{defn:2Karasalocalization}
    We let $\mathbf{2Kar}$ be the localization of  $\mathbf{2LocKar}$ along the following monomorphisms:\footnote{By \emph{walking} we mean that the structure is defined irrespective of any additional ambient structure i.e.\ in its free form.} 
    \begin{align}
    \{\text{The walking pair of objects}\}&\hookrightarrow \{\text{The walking 2-direct sum}\}
    \label{eq:WalkingPairOfObjects}
    \\
    \{\text{The walking 2-idempotent}\}&\hookrightarrow \{\text{The walking 2-retract}\}
    \label{eq:Walking2Idempotent}
    \end{align}
\end{defn}

We now argue that the presentably symmetric monoidal structure of $\mathbf{2LocKar}$ induces a presentably symmetric monoidal structure on $\mathbf{2Kar}$. In order to express this condition, we briefly recall some concepts from the general theory of localization, see~\cite[\S5.5.4]{highertopos}. Let us fix an $\infty$-category $\cV$, let $S$ be a subset of the $\infty$-category of arrows in $\mathcal{V}$.

 \begin{defn}\label{defn:Slocal}
 Let $S$ be a set of morphisms in a presentable $\infty$-category $\cV$. 
     An object $X\in \cV$ is \emph{$S$-local} if for all 1-morphisms $Y \to Z$ in $S$, $\Hom(Z,X) \rightarrow \Hom(Y,X)$ is an isomorphism of spaces.
    The \emph{localization of $\cV$ with respect to $S$}, denoted by $L_S(\cV)$, is the full sub-$\infty$-category of $\cV$  consisting of the $S$-local objects. We let $L_S: \cV \to L_S(\cV)$ be the localization functor that is the left adjoint\footnote{
    The localization functor $L_S$ exists by the adjoint functor theorem since the fully faithful inclusion $L_S(\cV) \hookrightarrow \cV$ preserves limits and $\kappa$-filtered colimits for some regular cardinal $\kappa$, since $S$ is a (small) set.} to the canonical inclusion $L_S(\cV)\hookrightarrow \cV$. We say that a morphism $f:X\to Y$ in $\cV$ is an \emph{$S$-equivalence} if $L_S(f)$ is an isomorphism in $L_S(\cV)$, or equivalently if for every $S$-local object $Z$, $\Map(Y,Z) \to \Map(X,Z) $ is an equivalence.
  \end{defn}

\begin{defn}\label{defn:idealistic} A set $S$ of morphisms in a presentably symmetric monoidal $\infty$-category $\cV$ is \emph{idealistic} if for every $S$-equivalence $f:X\to Y$, also $f \otimes Z: X\otimes Z \to Y \otimes Z$ is an $S$-equivalence for every $Z\in \cV$. 
\end{defn}

By~\cite[Proposition~5.5.4.15]{highertopos}, the collection of $S$-equivalences form the smallest strongly saturated  (\cite[Definition 5.5.4.5]{highertopos}) class of morphisms containing $S$. It follows that it suffices for Definition~\ref{defn:idealistic} to prove that $f\otimes Z$ is an $S$-equivalence for every $f\in S$ and $Z\in \cV$.  Denote the inner Hom in $\cV$, i.e.\ the right adjoint to the tensor product, by $\underline{\Hom}_\cV(-,-)$. By currying, one finds that $S$ is idealistic if and only if, whenever $C \in \cV$ and $X \in L_S(\cV)$, we have $\underline{\Hom}_\cV(C,X) \in L_S(\cV)$.

The following is a consequence of~\cite[Example 2.2.1.7 and Proposition 2.2.1.9 ]{higheralgebra}:
\begin{lem}
    Let $\cV$ be a presentably symmetric monoidal $\infty$-category and $S$ be an idealistic set of morphisms in $\cV$. Then there exists a presentably symmetric monoidal structure on $L_S: \cV \to L_S(\cV)$.
\end{lem}
In this case,  the tensor product in $L_S(\cV)$ of two $S$-local objects $X\otimes^{L_S(\cV)} Y$ is given by $L_S(X \otimes^{\cV} Y)$.

\begin{lem}\label{lem:2Kar}
    The $(3,1)$-category $\tKar$ is presentably symmetric monoidal, as is the localization functor $L_S : \mathbf{2LocKar} \to \tKar$.
\end{lem}

\begin{proof}
Presentability is automatic as $\tKar$ is a localization of a presentable $\infty$-category, at a set of morphisms \cite[Section 5.5.4.2 (3)]{highertopos}. 
It therefore only remains to check that the 1-morphisms in Definition~\ref{defn:2Karasalocalization} are idealistic.

Recall from \cite{hopkins2013ambidexterity} that an $\infty$-category $\cC$ is called ($0$-)\emph{semiadditive} if it has both finite products and coproducts, and the unique map from the initial object $\emptyset$ to the terminal object $*$ is an isomorphism, and for any two objects $A,B$, the map $A \sqcup B \to A\times B$ given by the matrix
\[ \begin{array}{c||c|c}
 & A & B \\ \hline\hline &&\\[-9pt]
 A & \id_A & A \to * \overset\sim\leftarrow \emptyset \to B \\ \hline &&\\[-9pt]
 B & B \to * \overset\sim\leftarrow \emptyset \to A & \id_B
\end{array}\]
is an isomorphism. The localization of $\mathbf{2LocKar}$ at 
\eqref{eq:WalkingPairOfObjects}
consists precisely of those $\mathbf{Kar}$-enriched 2-categories that are semiadditive. This $\infty$-category is denoted by $\mathsf{Cat}(\mathbf{Kar})^{\mathsf{semiadd}}$ in \cite[Appendix A.5]{mazel2021universal}, to which we refer the reader for a more detailed discussion. Then since $\mathbf{Kar}$ has direct sums, a $\mathbf{Kar}$-enriched $\infty$-category is semiadditive as soon as it has finite coproducts. But, if $X$ has finite coproducts, then so does $\underline{\Hom}(C,X)$ for any $C$, since coproducts in functor categories are computed pointwise. Idealism of 
\eqref{eq:WalkingPairOfObjects}
follows.

A similar argument establishes idealism of the map 
\eqref{eq:Walking2Idempotent}.
Indeed, by \cite{Kammermeier}, splitting a 2-idempotent is an example of a $\mathbf{Kar}$-weighted colimit (which is moreover absolute). In other words, a $\mathbf{Kar}$-enriched $\infty$-category $X$ is local for 
\eqref{eq:Walking2Idempotent}
as soon as it has these weighted colimits. 
Thus, if $X$ is local for 
\eqref{eq:Walking2Idempotent},
then so is $\underline{\Hom}(C,X)$ for any $C$, since weighted colimits in enriched functor categories are computed pointwise. Further details about weighted colimits and 2-idempotents can be found in \cite{Kammermeier}.
\end{proof}

\begin{defn}
    Denote by $\mathbf{3Vect}$, the full sub-$(3,1)$-category of $\tKar$ on the finite semisimple 2-categories. 
\end{defn}
By \cite{decoppet2020;comparison}, $\mathbf{3Vect}$ is equivalent to the Morita 3-category of multifusion 1-categories studied in \cite{DSPS13}.

\begin{lem}
    The full subcategory $\mathbf{3Vect} \subset \mathbf{2Kar}$ is a symmetric monoidal subcategory.
\end{lem}
\begin{proof}
    It suffices to show that the unit object $\mathbf{2Vect}$ is finite semisimple --- it is --- and that the monoidal product closes --- which holds only given that we are working over a perfect field, see e.g.~\cite{decoppet2023:2Deligne}.
\end{proof}

\begin{defn}\label{defn:MF2C} We define the $(3,1)$-category $\mathbf{MF2C}$ of multifusion $2$-categories (and linear monoidal functors) as the full subcategory of $\Alg_{E_1}(\mathbf{3Vect})$ on the multifusion $2$-categories. 
\end{defn}

Finally, we construct the fully faithful functor $\Mod(-): \mathbf{BMF1C} \to \mathbf{MF2C}$ which sends a braided multifusion $1$-category $\cB$ to the multifusion $2$-category of finite semisimple $\cB$-module categories. 
In the language of this section, it arises from the fully faithful functor $\mathbf{Alg}_{{E}_1}(\mathbf{Kar})\hookrightarrow \mathbf{Alg}_{{E}_0}(\mathbf{2Kar}_{})$, which takes the one-object delooping and then localizes from $\mathbf{2LocKar}_{}$ to $\mathbf{2Kar}_{}$; one then considers the composition of fully faithful inclusions
$$\mathbf{BMF1C}\hookrightarrow\mathbf{Alg}_{{E}_2}(\mathbf{Kar})\simeq \mathbf{Alg}_{{E}_1}(\mathbf{Alg}_{{E}_1}(\mathbf{Kar}))\hookrightarrow \mathbf{Alg}_{{E}_1}(\mathbf{Alg}_{{E}_0}(\mathbf{2Kar}_{})) \simeq \mathbf{Alg}_{{E}_1}(\mathbf{2Kar}) 
. $$ 
It is simply an objectwise check that the image of this composition is contained in $\mathbf{MF2C}$.
Also see~\cite{decoppet2020;comparison}.

\subsection{The Morita Category of Multifusion 2-Categories}\label{subsubsection:MorF2C}

Intuitively, the \textit{Morita 4-category of multifusion 2-categories}, denoted henceforth $\Mor^{\sss}$, is the 4-category with the following globular set:\ objects are multifusion 2-categories, 1-morphisms are finite semisimple bimodule 2-categories, 2-morphisms are bimodule functors, 3-morphisms are bimodule natural transformations, and 4-morphisms are bimodule modifications. Although this 4-category does have deeper-categorical structure, such as a symmetric monoidal structure, for most of the paper we will solely use its underlying $(4,1)$-category --- in other words, we will use only the invertible $2$-, $3$-, and $4$- morphisms --- and as usual we will treat this $(4,1)$-category as an $(\infty,1)$-category. 

Various constructions of higher Morita categories have already appeared in the literature \cite{haugseng2017higher,JFS2017,higheralgebra}. Specifically, it follows from  \cite{haugseng2017higher,higheralgebra} that, associated to any presentably symmetric monoidal $\infty$-category $\cV$, there is a flagged $(\infty,2)$-category $\Mor(\cV)$ whose objects are $E_1$-algebras in $\cV$ and 1-morphisms are bimodules in $\cV$. For our present purposes, we will take $\mathcal{V}=\mathbf{2Kar}$, the $(3,1)$-category of small $\mathbb{C}$-linear Karoubi complete 2-categories, constructed in~\S\ref{subsubsection:MorF2C}.

We will now construct the $(4,2)$-category $\Mor(\mathbf{2Kar})$ following the approach in~\cite{higheralgebra}, outlined in detail in~\cite[\S4.4]{soergel}.
The Morita 4-category $\Mor^{\sss}$ will then be obtained from $\Mor(\mathbf{2Kar})$ by restricting to the finite semisimple part, as we will explain below.

Let $\mathbf{Mod}_{\cV}:= \mathbf{Mod}_{\cV}(\mathrm{Pr}^L)$ denote the $\infty$-category of presentable $\cV$-module $\infty$-categories and cocontinuous $\cV$-module functors. There is a canonical functor $\Alg_{E_1}(\cV) \rightarrow \Mod_\cV$ sending an $E_1$-algebra $A$ to the $\cV$-module-category $\Mod_A(\cV)$ of $A$-modules (see~\cite[Corollary~4.2.3.7]{higheralgebra}); we will at times drop the $\cV$ in our notation. We write $\Mod_\cV^{\gen}$ for its full image.

\begin{defn}
We let $\Mor(\tKar)$ be the flagged $(\infty,1)$-category defined by \begin{equation}\label{eq:flaggingMor1}
  \ob( \Alg_{E_1}(\cV)) \rightarrow \Mod_\cV^{\gen}\,,
\end{equation}
with functor sending an $E_1$-algebra $A$ to $\Mod_A(\cV)$.
\end{defn}

We note that $\Mod_{\cV}$ is presentably tensored over $\mathbf{Cat}_{\infty}$, hence enriched over $\mathbf{Cat}_{\infty}$, the $\infty$-category of $\infty$-categories. In particular, $\Mod_{\cV}$ enhances to an $(\infty,2)$-category, see e.g. \cite{ben2024naturality};\ a fact that will be employed in \S\ref{subsubsection:Mor2ss}. Now, by construction, there is a surjective-on-objects  functor $[-] : \Alg_{E_1}(\cV) \to \Mor(\cV)$ that sends an algebra to its Morita equivalence class. 
Further, by definition the 1-morphisms in $\Mor(\cV)$ between two Morita classes $[A]$ and $[B]$ are given by cocontinuous $\cV$-module functors $\Mod_A(\cV) \to \Mod_B(\cV)$. It is more common to think of 1-morphisms in higher Morita categories as being given by suitable bimodules. 

The next proposition, proven in \cite[Theorem~4.8.4.1, Remark~4.8.4.9]{higheralgebra} may be considered an $\infty$-categorical \emph{Eilenberg--Watts theorem} and matches these two perspectives together:

\begin{prop}[{\cite[Remark~4.8.4.9, Theorem~4.8.4.1]{higheralgebra}}]
   Let $\cV$ be a presentably symmetric monoidal $\infty$-category. 
     Given $A,B \in \Alg_{E_1}(\cV)$, there is an equivalence of $\infty$-categories
    \[ \Hom_{\Mor(\cV)}([A],[B]) \cong \Bimod_{A-B}(\cV), \]
    the $\infty$-category of $A$-$B$-bimodule objects in $\cV$.
    Moreover, composition of $1$-morphisms in $\Mor(\cV)$ is given by the relative tensor product of bimodules. 
\end{prop}

\begin{cor}\label{cor:globMor1}
    The globular set of $\Mor(\tKar)$ is given by
    \begin{equation}
    \begin{tikzcd}
                \ob(\Alg_{E_1}(\tKar)) \arrow[r]& \ob(\mathbf{Bimod}(\tKar)). \arrow[l, shift right, shift right] \arrow[l, shift left,shift left]
    \end{tikzcd}
    \end{equation}
    The objects are the $E_1$-algebras in $\tKar$, and the 1-morphisms are bimodules. Further, composition of 1-morphisms is given by the balanced tensor product.
\end{cor}

Finally, we define the finite semisimple Morita 4-category $\Mor^{\sss}$.

 We wish to use $\mathbf{3Vect} \subset \mathbf{2Kar}$ to extract a sub flagged $(\infty,2)$-category of $\Mor(\tKar)$, that contains only the objects and 1-morphisms whose underlying 2-categories are finite semisimple. But, in order to define a sub-globular set using the description of $\Mor(\tKar)$ given in Corollary \ref{cor:globMor1}, we need to check that composition closes at all levels. This last property is only satisfied for algebras satisfying an additional property.

Recall from Definition~\ref{defn:MF2C} that we write $\mathbf{MF2C}$ for the full $\infty$-subcategory of $\Alg_{E_1}(\mathbf{3Vect})$ on the multifusion $1$-categories (i.e.\ those algebra objects in $\mathbf{3Vect}$ that are \emph{rigid}, in the sense of having duals for objects). 

Moreover, let $\Bimod^{\sss}(\mathbf{3Vect})$ denote the full subcategory of $\Bimod(\mathbf{3Vect}) \hookrightarrow \Bimod(\tKar)$ on those bimodules ${}_{\fC}\fM_{\fD}$ for which $\fC$ and $\fD$ are multifusion. (Equivalently, $\Bimod^{\sss}(\mathbf{3Vect})$ is the full subcategory of $\Bimod(\tKar)$ on those bimodules ${}_{\fC}\fM_{\fD}$ for which $\fC$ and $\fD$ are multifusion and $\fM$ is finite semisimple.) 
Given two finite semisimple bimodule 2-categories $_{\fC}\fM_{\fD}$ and $_{\fD}\fN_{\fE}$ between multifusion 2-categories, their composition $\fM \boxtimes_\fD \fN$ is finite semisimple by \cite{Decoppet2023;dualizable}. Since the compositions close, then, by Proposition \ref{prop:belief}, this concludes the construction of $\Mor^{\sss}$ as a sub flagged $(\infty,2)$-category of $\Mor(\mathbf{2Kar})$.

\begin{defn}
   Let $\Mor^{\sss}$ be the sub flagged $\infty$-category of $\Mor(\tKar)$ with globular data given by
    \begin{equation}
    \begin{tikzcd}
                \ob(\mathbf{MF2C}) \arrow[r]& \ob(\mathbf{Bimod}^{\sss}(\mathbf{3Vect}))\,. \arrow[l, shift right, shift right] \arrow[l, shift left,shift left]
    \end{tikzcd}
    \end{equation}
    Its objects are the multifusion 2-categories, and its 1-morphisms are bimodules in $\mathbf{3Vect}$ between them, i.e.\ finite semisimple bimodule 2-categories between multifusion 2-categories. Composition of 1-morphisms is given by the relative tensor product.
\end{defn}

In order to avoid having to think about flagged higher categories, we will replace $\Mor^{\sss}$ by its univalifications as defined in \S\ref{subsubection:flaggedcat}. This does not a priori preserve the inclusion of $\Mor^{\sss}$ into $\mathbf{Mod}(\mathbf{2Kar})$. 
To guarantee this last property, it would suffice by Lemma \ref{lem:isofull} to show that $\Mor^{\sss} \rightarrow \Mor(\tKar)$ is isofull. 
The following conjecture, which is a categorification of \cite{Tillmann}, would guarantee this property. 

\begin{conjecture}[2-Tillmann]\label{conj:2till}
An object of the symmetric monoidal 3-category $\mathbf{2Kar}$ is $1$-dualizable if and only if it is a finite semisimple $2$-category. 
\end{conjecture}

\begin{lem}
    Assuming Conjecture \ref{conj:2till}, the inclusion of sub flagged $(\infty,2)$-categories $\Mor^{\sss} \rightarrow \Mor(\tKar)$ is isofull, and hence (by Lemma~\ref{lem:isofull}) remains a subcategory inclusion after univalification.
\end{lem}
\begin{proof}
    Let $_\mathfrak{C} \mathfrak{M}_\mathfrak{D}$ be an equivalence in $\Mor(\mathbf{2Kar})$ with $\mathfrak{C}$ and $\mathfrak{D}$ two multifusion 2-categories. It is enough to show that $\mathfrak{M} = \mathfrak{C} \boxtimes_\mathfrak{C} \mathfrak{M} \boxtimes_\mathfrak{D} \mathfrak{D}$ is finite semisimple. But, $\mathfrak{C}_{\fC}$ and $_{\fD}\mathfrak{D}$ are adjunctible (on both sides) 1-morphisms because $\fC$ and $\fD$ are finite semisimple and hence dualizable as objects of $\mathbf{2Kar}$. Thus, $\mathfrak{M}$ is 1-dualizable in $\Omega \Mor(\mathbf{2Kar}) = \mathbf{2Kar}$. The last conjecture above concludes the proof.
\end{proof}

We will not provide a proof for the 2-Tillman conjecture here, but rather relegate this task to future work. We wish to point out that we will not appeal to either this conjecture nor the subsequent lemma. 
These were merely included to help the reader's intuition.

\subsection{The Morita Category of Braided Multifusion 1-Categories}\label{subsubsection:Mor2ss}

We will now define a $4$-category $\Mor_2^{\sss}$ whose objects are  braided multifusion $1$-categories, whose $1$-morphisms $\cB_1 \nrightarrow \cB_2$ are $\cB_1$--$\cB_2$--central multifusion $1$-categories (i.e.\ multifusion $1$-categories $\cC$ equipped with a braided functor $\cB_1 \boxtimes \cB_2^{\rev} \to \cZ(\cC)$), whose $2$-morphisms are finite semisimple bimodules between them, compatible with the central structure, and whose $3$- and $4$-morphisms are compatible bimodule equivalences and natural isomorphisms. We will define this $4$-category in such a way that there manifestly is a functor $\Mor_2^{\sss} \rightarrow \Mor^{\sss}$.

Although we will only use the underlying $(4,1)$-category in the reminder of the paper, it will be convenient to construct $\Mor_2^{\sss}$ as a $(4,2)$-category. 
Recall from \S\ref{subsubsection:MorF2C} that $\Mor(\tKar)$ inherits by restriction a $\mathbf{Cat}_{\infty}$-enrichment, i.e.\ an $(\infty,2)$-categorical structure, from $\Mod_{\tKar}$. In this section, we will underline, writing e.g. $\underline{\Mor}(\tKar)$ and $\underline{\Mod}_{\tKar}$ --- to remember this $(\infty,2)$-categorical structure and to distinguish them from their underlying $(\infty,1)$-categories --- still written e.g. as $\Mor(\tKar)$ and $\Mod_{\tKar}$.

To build $\Mor_2$, we will equip the $(\infty,2)$-category $\underline{\Mor}(\tKar)$ --- which we will succinctly write as $\underline{\Mor}$ below --- with another finer flagging as a $2$-category. 
To this end, let $\Mod_{\tKar}^{\pt}$ denote\footnote{Formally, $\Mod_{\tKar}^{\pt}$ can be defined as follows: Let $\widehat{\Cat}$ denote the very large $\infty$-category of large $\infty$-categories and let $\widehat{\Cat}_{\infty}^{\mathrm{oplax}\mhyphen\pt} \to \widehat{\Cat}_{\infty}$ denote (the large version of) the unstraightening of the equivalence $(-)^{\op}: \widehat{\Cat} \to \widehat{\Cat}$, i.e.\ informally $\widehat{\Cat}_{\infty}^{\mathrm{oplax}\mhyphen\pt}$ is the $\infty$-category of pointed categories and functors $F:\cC \to \cD$ preserving the pointing laxly, i.e.\ so that there is a morphism $d \to F(c)$. Then$\Mod_{\tKar}^{\pt}$ is defined as the pullback $\Mod_{\tKar}\times_{\widehat{\Cat}_{\infty}} \widehat{\Cat}_{\infty}^{\pt}$.
} the $\infty$-category of presentable $\tKar$-module categories equipped with a preferred object, and cocontinuous module functors preserving this object only laxly, i.e.\ up to a possibly-noninvertible $2$-morphism (which is part of the data of the $1$-morphism). 
Recall that there is a fully faithful functor $\Alg_{E_1} (\tKar) \hookrightarrow (\Mod_{\tKar})_{\tKar/}$, where the latter overcategory is equivalent to the subcategory of $\Mod_{\tKar}^{\pt}$ on those functors which strongly preserve the pointing. 
We let $\Mod_{\tKar}^{\pt\textrm{--}\gen}$ denote the $(\infty,1)$-category obtained as the full image of $\Alg_{E_1} (\tKar) \hookrightarrow (\Mod_{\tKar})_{\tKar/} \to \Mod_{\tKar}^{\pt}$.

Now consider the flagged $(4,2)$-category $\underline{\widehat{\Mor}}$ with flagging
\begin{equation}
    \mathrm{ob}(\Alg_{E_1}(\mathbf{2Kar})) \rightarrow \Mod^{\pt\textrm{--}\gen}_{\mathbf{2Kar}}\rightarrow \underline{\Mod}^{\gen}_{\mathbf{2Kar}}\,,
 \end{equation}

This is a variant of Equation \eqref{eq:flaggingMor1}, where we have included an additional layer of flagging for 1-morphisms.\footnote{This is a pointed version of higher Morita categories as considered for instance in \cite[Section 8]{JFS2017} and \cite{Scheimbauer2015}. By pointed we mean that the bimodules are equipped with distinguished objects, which are to be preserved by higher morphisms.}

Unwinding this definition via the $\infty$-categorical Eilenberg-Watts theorem \cite[Remark~4.8.4.9]{higheralgebra} shows: The space of objects of this flagged $(4,2)$-category $\underline{\widehat{\Mor}}$ is the space $\ob(\Alg_{E_1}(\mathbf{2Kar}))$ of $E_1$-algebras in $\tKar$. The space of $1$-morphisms is the space $\ob(\Bimod_{E_0}(\tKar))$ of pointed bimodules(Here, $\Bimod_{E_0}(-)$ denotes the $\infty$-category of algebras over the operad $BM\otimes E_0$, corepresenting a pair of $E_1$-algebras with a pointed bimodule between them.). The $2$-morphisms are arrows between bimodules (ignoring the pointing).
 Employing the same procedure as that utilized in \S\ref{subsubsection:MorF2C}, we then restrict to the ``semisimple part" of $\widehat{\Mor}$ denoted by $\widehat{\Mor}^{\sss}$, that is we restrict the spaces of objects and $1$-morphisms to the full subspaces 
 \[
 \ob(\mathbf{MF2C}) \hookrightarrow \ob(\Alg_{E_1}(\mathbf{2Kar})) \hspace{1cm} \ob(\Bimod_{E_0}^{\sss}(\mathbf{3Vect})) \hookrightarrow \ob(\Bimod_{E_0}(\tKar)).
 \]

\noindent The canonical map $\underline{\Mor}^{\sss}\rightarrow \underline{\widehat{\Mor}}^{\sss} $ is not an equivalence of flagged $(4,2)$-categories. Nevertheless, it follows from the definitions that this map becomes an equivalence after univalification.

Let $\mathbf{Bimod}_{E_1}(\mathbf{Kar})$ denote the $\infty$-category of algebras over the operad $E_1 \otimes BM$, corepresenting a pair of $E_2$-algebras $\cB_1, \cB_2$, an $E_1$-algebra $\cC$ and a braided functor $\cB_1 \boxtimes \cB_2^{\rev} \to \cZ(\cC)$. Let $\mathbf{Bimod}_{E_1}^{\sss}(\mathbf{2Vect})$ denote the full subcategory where $\cB_1, \cB_2$ and $\cC$ are required to be multifusion. The fully faithful functor $\Alg_{E_1}(\mathbf{Kar}) \to \Alg_{E_0}(\tKar)$ induces a fully faithful functor $\mathbf{Bimod}_{E_1}(\mathbf{Kar}) \hookrightarrow \mathbf{Bimod}_{E_0}(\tKar)$ and hence a fully faithful functor $\mathbf{Bimod}^{\sss}_{E_1}(\mathbf{Kar}) \hookrightarrow \mathbf{Bimod}^{\sss}_{E_0}(\tKar)$

\begin{defn}
Using Proposition~\ref{prop:belief}, we define $\underline{\Mor}_2^{\sss}$ as the sub-flagged-$(\infty,2)$-category of $\underline{\widehat{\Mor}}^{\sss}$ with full subspace of objects $\ob(\mathbf{BMF1C}) \hookrightarrow \ob(\mathbf{MF2C})$ and  full subspace of $1$-morphisms $\ob(\mathbf{Bimod}^{\sss}_{E_1}(\mathbf{2Vect}))\hookrightarrow \ob( \mathbf{Bimod}^{\sss}_{E_0}(\mathbf{3Vect}))$ (and all $2$-morphisms). 
\end{defn}
 We claim that this is compatible with the composition of 1-morphisms, so that the above definition indeed defines a sub-$(\infty,2)$-category by Proposition~\ref{prop:belief}. To see this, let $\mathcal{B}_1$, $\mathcal{B}_2$, and $\cB_3$ be braided multifusion 1-categories, and let $\mathcal{C}_1:\mathcal{B}_1\nrightarrow\mathcal{B}_2$, and $\mathcal{C}_2:\mathcal{B}_2\nrightarrow\mathcal{B}_3$ be objects of $\mathbf{Bimod}_{E_1}^{\sss}(\mathbf{2Vect})$. The corresponding objects of $\mathbf{Bimod}^{\sss}_{E_0}(\mathbf{3Vect})$ are the pointed finite semisimple bimodule 2-categories $_{\mathbf{Mod}(\cB_1)}\mathbf{Mod}(\cC_1)_{\mathbf{Mod}(\cB_2)}$ and $_{\mathbf{Mod}(\cB_2)}\mathbf{Mod}(\cC_2)_{\mathbf{Mod}(\cB_3)}$. But, it follows from \cite[Example 2.3.2]{Decoppet2023;dualizable} that $$\mathbf{Mod}(\cC_1)\boxtimes_{\mathbf{Mod}(\cB_2)}\mathbf{Mod}(\cC_2)\simeq \mathbf{Mod}(\cC_1\boxtimes_{\cB_2}\cC_2),$$ so that the claim follows.

\begin{rem}
    Our construction of the Morita 4-category $\Mor^{\sss}_2$ is motivated by the fact that it admits a canonical functor to $\Mor^{\sss}$. In \cite{BJS}, using different methods, an a priori different Morita 4-category of braided multifusion 2-category denoted by $\mathbf{BrFus}$ was constructed. We expect, but will not show, that these two Morita 4-categories are equivalent.
\end{rem}

\begin{lem}\label{lem:oursketchyMorita4categories}
The inclusion of flagged $(4,2)$-categories $\underline{\Mor}_2^{\sss}\hookrightarrow \underline{\widehat{\Mor}}^{\sss}$ becomes an equivalence after univalification. In particular, after univalification, there is an equivalence of $(4,2)$-categories $\underline{\Mor}_2^{\sss}\simeq\underline{\Mor}^{\sss}$.
\end{lem}
\begin{proof}
The last part follows from the fact that, after univalification, the map of flagged $(4,2)$-categories $\underline{\Mor}^{\sss}\hookrightarrow \underline{\widehat{\Mor}}^{\sss}$ becomes an equivalence. It therefore only remains to argue that the inclusion $\underline{\Mor}_2^{\sss}\hookrightarrow \underline{\widehat{\Mor}}^{\sss}$ becomes an equivalence after univalification. At the level of objects, this follows from \cite[Theorem 4.2.2]{Decoppet2022;centers}, and, at the level of 1-morphisms, this follows from \cite[Lemma 2.2.6]{decoppet2021finite}.
\end{proof}

For the rest of the paper, we will only consider the univalification
\[\Mor_2^{\sss} \simeq \Mor^{\sss}
\]of the underlying $(\infty,1)$-categories constructed here, hence removing the underline.

\subsection{Adjectives for 1-Morphisms }\label{subsection:adjectives}

The purpose of this section is to set up and study various properties of 1-morphisms in the Morita 4-categories that were introduced in \S\ref{subsubsection:MorF2C} and \S\ref{subsubsection:Mor2ss}. This will be especially relevant in the construction of the functor $\Witt(-)$ in  \S\ref{subsection:Wittfunctor}.

\subsubsection{Adjectives for Braided and Multifusion 1-Categories}\label{subsub:BF1C}

We introduce various properties for 1-morphisms in the Morita 4-category of braided multifusion 1-categories $\mathbf{Mor}^{\sss}_2$. Given two braided multifusion 1-categories $\mathcal{B}_1$ and $\mathcal{B}_2$, we write $\mathcal{C}:\mathcal{B}_1\nrightarrow\mathcal{B}_2$ for (a representative of) a 1-morphism in $\mathbf{Mor}^{\sss}_2$, that is, $\mathcal{C}$ is a multifusion 1-category equipped with a braided tensor functor $F:\mathcal{B}_1\boxtimes\mathcal{B}_2^{\rev}\rightarrow\mathcal{Z}(\mathcal{C})$. We also say that $\mathcal{C}$ is a  $\mathcal{B}_1\boxtimes\mathcal{B}_2^{\rev}$-central multifusion 1-category. We will write $F_1:\mathcal{B}_1\rightarrow\mathcal{Z}(\mathcal{C})$ and $F_2:\mathcal{B}_2^{\rev}\rightarrow\mathcal{Z}(\mathcal{C})$ for the braided tensor functors induced by $F$. Moreover, we use $\mathcal{Z}(\mathcal{C},\mathcal{B}_1)$ to denote the centralizer of the image of $\mathcal{B}_1$ under $F_1$ in $\mathcal{Z}(\mathcal{C})$, and we make the obvious analogous definition for $\mathcal{Z}(\mathcal{C},\mathcal{B}_2^{\rev})$.

\begin{defn}
\label{defn:domBF1}
 Let $\mathcal{C}:\mathcal{B}_1\nrightarrow\mathcal{B}_2$ be a 1-morphism in the Morita 4-category of braided multifusion 1-categories. 

\begin{enumerate}[label=(\alph*)]
    \item \label{defn:domBF1:0dom} We say that $\mathcal{C}:\mathcal{B}_1\nrightarrow\mathcal{B}_2$ is \textit{0-dominant} if $\mathcal{B}_2^{\rev}\rightarrow \mathcal{Z}(\mathcal{C}, \cB_1)$ is faithful.
    \item \label{defn:domBF1:1dom} We say that $\mathcal{C}:\mathcal{B}_1\nrightarrow\mathcal{B}_2$ is \textit{1-dominant} if $\mathcal{B}_2^{\rev}\rightarrow \mathcal{Z}(\mathcal{C},\mathcal{B}_1)$ is fully faithful.
    \item \label{defn:domBF1:2dom} We say that $\mathcal{C}:\mathcal{B}_1\nrightarrow\mathcal{B}_2$ is \textit{2-dominant} if $\mathcal{B}_2^{\rev}\rightarrow \mathcal{Z}(\mathcal{C},\mathcal{B}_1)$ is an equivalence.
\end{enumerate}
\end{defn}
Since $\mathcal{Z}(\mathcal{C}, \cB_1) \hookrightarrow \mathcal{Z}(\mathcal{C})$ is fully faithful,
\ref{defn:domBF1:0dom} and \ref{defn:domBF1:1dom} are equivalent to the assertion that the braided functor $\cB_2^{\rev} \to \cZ(\cC)$ is faithful, resp. fully faithful.

\begin{defn}\label{defn:faithBF1}
Let $\mathcal{C}:\mathcal{B}_1\nrightarrow\mathcal{B}_2$ be a 1-morphism in the Morita 4-category of braided multifusion 1-categories.

\begin{enumerate}[label=(\alph*)]
    \item We say that $\mathcal{C}:\mathcal{B}_1\nrightarrow\mathcal{B}_2$ is \emph{0-faithful} if $\mathcal{B}_1\rightarrow \mathcal{Z}(\mathcal{C}, \cB_2^{\rev})$ is faithful.
    \item We say that $\mathcal{C}:\mathcal{B}_1\nrightarrow\mathcal{B}_2$ is \emph{1-faithful} if $\mathcal{B}_1\rightarrow \mathcal{Z}(\mathcal{C},\mathcal{B}_2^{\rev})$ is fully faithful.
    \item We say that $\mathcal{C}:\mathcal{B}_1\nrightarrow\mathcal{B}_2$ is \emph{2-faithful} if $\mathcal{B}_1\rightarrow \mathcal{Z}(\mathcal{C},\mathcal{B}_2^{\rev})$ is an equivalence.
\end{enumerate}
\end{defn}

As stated, the above definitions refer to the underlying braided multifusion categories $\cB_1$ and $\cB_2$ and are thus not obviously well-defined conditions on a $1$-morphism in the (univalent) $4$-category $\Mor_2^{\sss}$. Namely, they might a priori depend on the flagging as in \S\ref{subsubsection:Mor2ss}, i.e.\ on the choice of braided multifusion 1-category representing a given object. So as to remedy this issue, we will argue in Theorem \ref{thm:invertibleIFFWittEq} that every invertible 1-morphism in $\Mor_2^{\sss}$ is both $2$-faithful and $2$-dominant. In particular, the well-definedness of the above faithfulness and dominance conditions will then follow from establishing that they are closed under composition. This shows in addition that these conditions yield subcategories of the Morita $4$-category of braided multifusion $1$-categories. This will be used subsequently to establish corresponding results for 1-morphisms in the Morita 4-category of multifusion $2$-categories.

\begin{prop} \label{prop:nDomComposesMor2}
Let $\cB_1$, $\cB_2$, and $\cB_3$ be braided multifusion $1$-categories and let $\cC_1:\cB_1\nrightarrow\cB_2$, and $\cC_2:\cB_2\nrightarrow\cB_3$ be two $n$-faithful, resp.\ $n$-dominant, 1-morphisms for some $n\in\{0,1,2\}$. Then the composite $\cC_1\boxtimes_{\cB_2}\cC_2:\cB_1\nrightarrow\cB_3$ is $n$-faithful, resp.\ $n$-dominant. 
\end{prop}

Before turning to the proof, we make some general observations about decomposition of $1$-morphisms in $\Mor_2^{ss}$ which will prove helpful in understanding the several dominance and faithfulness conditions. 
\begin{rem}\label{rem:mor2Decomp}
Suppose that $\cC:\cB_1\nrightarrow\cB_2$ is a morphism in $\Mor_2^{\sss}$.
Each braided multifusion $1$-category $\cB_i$ is a direct sum of braided fusion categories, so we can write
\[\cB_1\cong\left(\bigoplus\limits_{j=1}^n\cB_{1,j}\right)\oplus\cB_{1,0},\qquad\qquad\cB_2\cong\left(\bigoplus\limits_{k=1}^m\cB_{2,k}\right)\oplus\cB_{2,0}\]
where each $\cB_{1,j}$ and $\cB_{2,k}$ for $j,k\neq 0$ is fusion and has a nonzero action on $\cC$, and where $\cB_{1,0}$ and $\cB_{2,0}$ act as zero on $\cC$.

If $\cD$ is an indecomposable multifusion summand of $\cC$, then there is a unique $j$ such that $\cB_{2,j}$ has a nonzero action on $\cD$.
To see this, observe that for every distinct $j$ and $\ell$, we have $0\cong 1_{\cB_{2,j}}1_{\cB_{2,\ell}}$, so at least one of the two summands of $\cB_2$ must act as zero on $\cD$.
Since the actions of $\cB_1$ and $\cB_2$ on $\cC$ are unital,  
the multifusion category $\cC$ can be written as a direct sum of indecomposable multifusion categories
 \[\cC\cong\bigoplus\limits_{\ell=1}^p\cC_{\ell}\]
 and that for each $\ell$ there is $j(\ell)$ and $k(\ell)$
 such that $\cC_\ell:\cB_{1,k(\ell)}\nrightarrow\cB_{2,j(\ell)}$ is a nonzero indecomposable morphism in $\Mor_2^{ss}$ with fusion source and target.
 Moreover, since each $\cC_{\ell}$ is indecomposable multifusion, $\cC_\ell$ is Morita equivalent to a fusion $\cB_{1,k(\ell)}-\cB_{2,j(\ell)}$ bimodule category, meaning that up to equivalence of 2-morphisms in $\Mor_2^{\sss}$, we may always assume that each $\cC_\ell$ is fusion as well.

 From the definitions of $0$-faithful and $0$-dominant, we can see that $\cC:\cB_1\nrightarrow\cB_2$ is $0$-faithful if and only if $\cB_{1,0}\cong 0$ and $0$-dominant if and only if $\cB_{2,0}\cong 0$.
 From the definitions of $n$-faithful and $n$-dominant for $n\in\{1,2\}$, we can see that $\cC:\cB_1\nrightarrow\cB_2$ is $n$-dominant ($n$-faithful) if and only if $\cC$ is $0$-dominant ($0$-faithful) and each $\cC_{\ell}:\cB_{1,k(\ell)}\to\cB_{2,j(\ell)}$ is $n$-dominant ($n$-faithful).
\end{rem}

We will prove Proposition~\ref{prop:nDomComposesMor2} in two stages:\ We prove the result in the cases $n=0$ and $n=1$ now, and defer the case $n=2$ until we have established an equivalent characterization of $2$-dominance. 

\begin{lem} \label{lem:0or1DomComposesMor2}
Let $\cB_1$, $\cB_2$, and $\cB_3$ be braided multifusion $1$-categories and let $\cC_1:\cB_1\nrightarrow\cB_2$, and $\cC_2:\cB_2\nrightarrow\cB_3$ be two $n$-faithful, resp.\ $n$-dominant, 1-morphisms where $n\in\{0,1\}$. Then the composite $\cC_1\boxtimes_{\cB_2}\cC_2:\cB_1\nrightarrow\cB_3$ is $n$-faithful, resp.\ $n$-dominant.
\end{lem}
\begin{proof}
We give the proofs for $n$-dominance, and note that the proofs for $n$-faithfulness are dual, since $\cC:\cA\nrightarrow\cB$ is $n$-dominant if and only if $\cC^{\mp}:\cB\nrightarrow\cA$ is $n$-faithful.

First, we consider $n=0$.
Because our actions are unital, it is clear that the composite of nonzero $1$-morphisms with fusion source and target is nonzero.
Decomposing as in Remark~\ref{rem:mor2Decomp}, let  $\cB_2\cong\oplus_k\cB_{1,k}$, $\cB_3\cong\oplus_j\cB_{3,j}$, $\cC_2\cong\oplus_\ell\cC_{2,\ell}$, and $\cC_1\cong\oplus_m\cC_{1,m}$.
To show that $\cC_1\boxtimes_{\cB_2}\cC_2:\cB_1\nrightarrow\cB_3$ is $2$-dominant, it suffices to show that the action of each summand $\cB_j$ on $\cC_1\boxtimes_{\cB_2}\cC_3$ is nonzero.

Thus, as a multifusion category, we have
\[
\cC_1\boxtimes_{\cB_2}\cC_2\cong\bigoplus\limits_j\bigoplus\limits_{\ell:j(\ell)=j}\bigoplus\limits_{m:k(m)=k(\ell)}\cC_{1,m}\boxtimes_{\cB_{2,k(\ell)}}\cC_{2,\ell}
\]
Since each $\cC_{1,m}$ and $\cC_{2,\ell}$ are indecomposable multifusion, $\cB_{2,k(\ell)}$ is fusion, and $\cB_{2,k(\ell)}$ acts has a nonzero action on both whenever $k(\ell)=k(m)$, each relative tensor product $\cC_{1,m}\boxtimes\cB_{2,k(\ell)}\cC_{2,\ell}$ in the direct sum is nonzero.
Moreover, if $j(\ell)=j$, then $\cB_{3,j}$ has a nonzero action on the indecomposable $\cC_{2,\ell}$, and hence on the nonzero $\cC_{1,m}\boxtimes_{\cB_{2,k(\ell)}}\cC_{2,\ell}$.
Therefore, to show that $\cB_{3,j}$ has a nonzero action on $\cC_1\boxtimes_{\cB_2}\cC_3$, it suffices to check that there is always some $\ell$ with $j(\ell)=j$ and some $m$ with $k(m)=k(\ell)$.
The former follows from the fact that $\cC_2:\cB_2\nrightarrow\cB_3$ is $0$-dominant, while the latter follows from the fact that $\cC_1:\cB_1\nrightarrow\cB_2$ is $0$-dominant.

Let us now consider the case $n=1$.
By Remark \ref{rem:mor2Decomp}, now that the case $n=0$ has been proven, we may assume without loss of generality that all the multifusion 1-categories $\cB_i$ and $\cC_i$ are fusion.

By definition, we are given braided functors $\cB_2^{\rev}\to \cZ(\cC_1)$ and $\cB_3^{\rev}\to \cZ(\cC_2)$ that are fully faithful, and we wish to show that the canonical functor $\cB_3^{\rev}\to \cZ(\cC_1\boxtimes_{\cB_2}\cC_2)$ is again fully faithful.
It suffices to see that the canonical functor $\cZ(\cC_2,\cB_2)\to\cZ(\cC_1\boxtimes_{\cB_2}\cC_2)$ is also fully faithful, since $\cB_3^{\rev}$ includes fully faithfully into the former.

Let $K$ denote the canonical \'etale algebra in $\cB_2^{\rev}\boxtimes\cB_2$ corresponding to the monoidal functor $\cB_2^{\rev}\boxtimes\cB_2\rightarrow \cB_2$.
Said differently, the multifusion 1-category of $K$-modules in $\cB_2^{\rev}\boxtimes\cB_2\rightarrow \cB_2$ is identified with $\cB_2$.
Since $\cC_1:\cB_1\nrightarrow\cB_2$ is $1$-dominant, the monoidal functor $\cB_2^{\rev}\to\cZ(\cC_1)$ is fully faithful, so that the maximal subalgebra of $\widetilde{K}$ in $\cZ(\cC_2)$ is $\mathbf{1}$.
Let also $\widetilde{K}$ denote the image of $K$ in $\cZ(\cC_1\boxtimes\cC_2)$. Then, by \cite[Remark~2.8]{DNO2013} and \cite[Theorem~3.20]{DMNO}, we have $\cC_1\boxtimes_{\cB_2}\cC_2\cong \Mod_{\widetilde{K}}(\cC_1\boxtimes\cC_2)$ , and $\cZ(\cC_1\boxtimes_{\cB_2}\cC_2)\cong \Mod^{\loc}_{_{\widetilde{K}}}(\cZ(\cC_1)\boxtimes \cZ(\cC_2))$ (see also \cite[\S2.3]{MR4498161}). 
Now, there is a canonical functor $J:\cZ(\cC_2,\cB_2)\to \Mod_{\widetilde{K}}(\cZ(\cC_1\boxtimes\cC_2))$ given by
\[\cZ(\cC_2,\cB_2)\to\cZ(\cC_2)\to\cZ(\cC_1\boxtimes\cC_2)\to\Mod_{\widetilde{K}}(\cZ(\cC_1\boxtimes\cC_2))\]
The functor $\cZ(\cC_2)\to\Mod_{\widetilde{K}}(\cZ(\cC_1\boxtimes\cC_2))$ is fully faithful because the maximal subalgebra of $\widetilde{K}$ in $\cZ(\cC_2)$ is trivial, so the composite functor $J$ is also fully faithful.
Meanwhile, since $\widetilde{K}$ is in $\cB_2^{\rev}\boxtimes F(\cB_2)\subseteq \cZ(\cC_1)\boxtimes \cZ(\cC_2)$, every module in the image of $J$ is local.
Thus, the canonical braided tensor functor $\cZ(\cC_2,\cB_2)\to\Mod_{\widetilde{K}}^{\loc}(\cZ(\cC_1\boxtimes\cC_2))$ is fully faithful, as desired, verifying the $1$-dominance of $\cC_1\boxtimes_{\cB_2}\cC_2:\cB_1\nrightarrow\cB_3$.
\qedhere
\end{proof}

We now turn to the characterization of invertible $1$-morphisms in $\Mor_2^{\sss}$.
First, we will prove in Theorem~\ref{thm:invertibleIFFWittEq} that  the invertible morphisms are just Witt equivalences over a symmetric multifusion category in the sense of \cite{DNO2013}, characterizing invertibility in terms of centralizers.
We then build on this to characterize invertibility in terms of faithfulness and dominance conditions in Lemma~\ref{lem:invertible2faith2dom}.

The following basic facts about centralizers will prove helpful.
\begin{lem}
 Let $\cB$ be a braided multifusion $1$-category and $\cC$ be a $\cB$ central multifusion $1$-category.
 Then $\cZ(\cC,\cZ(\cC,\cB))\cong\Im(\cB)$, where $\Im(\cB)$ denotes the image of $\cB$ in $\cZ(\cC)$.
 \label{lem:BF1CDoubleCommutant}
\end{lem}
\begin{proof}
 This follows immediately from \cite[Corollary~3.11]{DGNO2010braided} and the fact that $\cZ(\cC)$ is nondegenerate.
\end{proof}
\begin{cor}
 Let $\cB$ be a braided multifusion $1$-category and $\cC$ be a $\cB$ central multifusion $1$-category.
 Then $\cZ_2(\cZ(\cC,\cB))\cong\cZ_2(\Im(B))$.
 \label{cor:centerOfCentralizers}
\end{cor}
\begin{proof}
 Since $\cZ_2(\Im(\cB))\subseteq\Im(\cB)\subseteq\cZ(\cC)$, we have $\cZ_2(\Im(\cB))\cong \Im(\cB)\cap\cZ(\cC,\cB)$.
 Similarly, $\cZ_2(\cZ(\cC,\cB))\cong\cZ(\cC,\cB)\cap \cZ(\cC,\cZ(\cC,\cB))$.
 Applying Lemma~\ref{lem:BF1CDoubleCommutant}, we see that both symmetric centers are the intersection of the same two subcategories of $\cZ(\cC)$.
\end{proof}
We also recall the following definitions from \cite{DNO2013}.
\begin{defn}
 \label{def:nondegOver}
 If $\cE$ is a symmetric multifusion $1$-category, an \textit{$\cE$-nondegenerate braided multifusion $1$-category} is a braided multifusion $1$-category $\cB$ with a choice of equivalence $\cE\to\cZ_2(\cB)$.
 
 If $\cB_1$ and $\cB_2$ are two $\cE$-nondegenerate braided multifusion $1$-categories, then $\cB_1\boxtimes_{\cE}\cB_2$ denotes the relative tensor product using the chosen maps $\cE\cong\cZ_2(\cB_j)\to\cB_j$.
\end{defn}
\begin{defn}
 \label{def:WittEqOver}
 If $\cB_1$ and $\cB_2$ are $\cE$-nondegenerate braided multifusion categories, an \textit{$\cE$-Witt equivalence} from $\cB_1$ to $\cB_2$ is a multifusion category $\cC$ together with a functor $\cB_1\boxtimes\cB_2^{\rev}\to \cZ(\cC)$ which factors through $\cB_1\boxtimes_{\cE}\cB_2^{\rev}$ such that $\cZ(\cC,\cB_1\boxtimes\cB_2^{\rev})\cong\cZ(\cC,\cB_1\boxtimes_{\cE}\cB_2^{\rev})\cong\cE$.
\end{defn}
\begin{rem}
 \label{rem:defWittEqOverCorrect}
 By Lemma~\ref{lem:BF1CDoubleCommutant}, $\cZ(\cC,\cE)\cong\cB_1\boxtimes_{\cE}\cB_2^{\rev}$, so Definition~\ref{def:WittEqOver} agrees with \cite[Def.~5.1]{DNO2013}.
\end{rem}

The next result generalizes \cite[Theorem~2.18]{MR4498161}.
\begin{thm}
 \label{thm:invertibleIFFWittEq}
Suppose $\cC:\cB_1\nrightarrow\cB_2$ is a 1-morphism in $\Mor_2^{\sss}$.
The following are equivalent.
\begin{enumerate}
\item $\cC$ is invertible.
\item $\cZ_2(\cB_1)\cong\cZ_2(\cB_2)=:\cE$ and
$\cC$ is an $\cE$-Witt equivalence.
\item $\cC$ is 2-faithful and 2-dominant.
\end{enumerate}

\end{thm}
\begin{proof}

Let $\cC:\cB_1\nrightarrow\cB_2$ be a 1-morphism in $\Mor_2^{\sss}$. Recall that $\cZ(\cC^{\mp})\simeq \cZ(\cC)^{\rev}$, and we can equip $\cC^{\mp}$ with the braided tensor functor $\cB_2\boxtimes\cB_1^{\rev}\rightarrow \cZ(\cC)^{\rev}$. There is an adjunction between these two 1-morphisms, in which the unit and counit are both versions of ``$\cC$ as a $\cC\boxtimes_{\cB_1 \boxtimes \cB_2^\rev} \cC^{\mp}$-module'' \cite[Proposition 4.14]{GwilliamScheimbauer}\footnote{More precisely, since \cite{GwilliamScheimbauer} uses a different model of Morita higher categories than we do, one cannot simply quote their result, but one can quote their proof: by \cite{MR3415698}, it suffices to check that such an adjunction exists in the homotopy bicategory $\operatorname{ho}_2 \Mor_2^{\sss}$, and it is a routine exercise to unpack therein the graphical arguments of \cite{GwilliamScheimbauer}.}.
Now, it follows from general categorical nonsense that $\cC$ is invertible if and only if the unit and counit of this adjunction are equivalences.

\underline{(1)$\Rightarrow$(2):}
 Suppose $\cC:\cB_1\nrightarrow\cB_2$ is invertible.
 We claim $\cZ_2(\cB_1)\cong\cZ_2(\cB_2)$.
 Invertibility implies that $\cC\boxtimes_{\cB_2}\cC^{\mp}\cong\cB_1$ as $1$-morphisms $\cB_1\nrightarrow\cB_1$, so that
\[
\cZ(\cC\boxtimes_{\cB_2}\cC^{\mp},\cB_1\boxtimes\cB_1^{\rev})\cong\cZ(\cB_1,\cB_1\boxtimes\cB_1^{\rev})\cong\cZ_2(\cB_1)
\]
Similarly,
$\cZ(\cC^{\mp}\boxtimes_{\cB_1}\cC,\cB_2\boxtimes\cB_2^{\rev})\cong\cZ_2(\cB_2)$.
But by construction of the relative tensor product, 
\[\cZ(\cC\boxtimes_{\cB_2}\cC^{\mp},\cB_1\boxtimes\cB_1^{\rev})\cong\cZ(\cC\boxtimes\cC^{\mp},\cB_1\boxtimes\cB_2^{\rev}\boxtimes\cB_2\boxtimes\cB_2^{\rev})\cong\cZ(\cC^{\mp}\boxtimes_{\cB_1}\cC,\cB_2\boxtimes\cB_2^{\rev})\]
giving a canonical equivalence
$\cZ_2(\cB_1)\cong\cZ_2(\cB_2)=:\cE$.

We may now assume that the maps $\cE\to\cB_1\to\cZ(\cC)$ and $\cE\to\cB_2\to\cZ(\cC)$ are equivalent.
The canonical dominant tensor functor $\cC\boxtimes\cC^{\mp}\to\cC\boxtimes_{\cB_2}\cC^{\mp}$ factors through $\cC\boxtimes_{\cE}\cC^{\mp}$, and the image of $\cB_2\boxtimes\cB_2^{\rev}$ in $\cZ(\cC\boxtimes_{\cE}\cC^{\mp})$ is canonically equivalent to $\cB_2\boxtimes_{\cE}\cB_2^{\rev}$.
Together with invertibility of $\cC$, this implies that
\[
 \cE \cong \cZ(\cC\boxtimes_{\cB_2}\cC^{\mp},\cB_1\boxtimes\cB_1^{\rev})
 \cong 
 \cZ(\cC\boxtimes_{\cE}\cC^{\mp},(\cB_1\boxtimes\cB_1^{\rev})\boxtimes(\cB_2^{\rev}\boxtimes_\cE\cB_2))
 \]
 Since the action of $\cB_1\boxtimes\cB_2\to\cZ(\cC)$ factors through $\cB_1\boxtimes_{\cE}\cB_2^{\rev}$, the images of $\cE$ in all four tensorands are identified, so we may rewrite
 \[\cE \cong \cZ(\cC\boxtimes_{\cE}\cC^{\mp},\cB_1\boxtimes_{\cE}\cB_2^{\rev}\boxtimes_{\cE}\cB_2\boxtimes_{\cE}\cB_1^{\rev})\]
 By taking centralizers on both sides of \cite[Prop.~4.7]{DNO2013}, we factorize over $\cE$ to obtain
 \[\cE \cong \cZ(\cC,\cB_1\boxtimes_{\cE}\cB_2^{\rev})\boxtimes_\cE\cZ(\cC^{\mp},\cB_2\boxtimes_{\cE}\cB_1^{\rev})
 \cong \cZ(\cC,\cB_1\boxtimes\cB_2^{\rev})\boxtimes_{\cE}\cZ(\cC,\cB_1\boxtimes\cB_2^{\rev})^{\rev}
 \]
 By \cite[Corollary~4.4]{DNO2013}, we simply have
 $\cE \cong \cZ(\cZ(\cC,\cB_1\boxtimes\cB_2^{\rev}),\cE)$.
 Finally, since $\cZ(\cC,\cB_1\boxtimes\cB_2^{\rev})$ is $\cE$-nondegenerate, we may simplify to
 $\cE \cong \cZ(\cC,\cB_1\boxtimes\cB_2^{\rev})$.
 This shows that $\cC$ is indeed an $\cE$-Witt equivalence.
 
\underline{(2)$\Rightarrow$(3):}
Suppose $\cC:\cB_1\nrightarrow\cB_2$ is an $\cE$-Witt equivalence.
Then $\cZ(\cC,\cE)\cong\cB_1\boxtimes_{\cE}\cB_2^{\rev}$, so $\cZ(\cC,\cB_1)\cong\cB_2^{\rev}$, meaning $\cC$ is $2$-dominant, and $\cZ(\cC,\cB_2^{\rev})\cong\cB_1$, so $\cC$ is $2$-faithful.

\underline{(3)$\Rightarrow$(1):}
By the discussion before (1)$\Rightarrow$(2),
if $\cC$ is invertible, its inverse is $\cC^{\rm mp}: \cB_2\to \cB_1$.
Observe that $\cC$ carries a left action of $\cC\boxtimes_{\cB_2} \cC^{\rm mp}$, since the action of $\cB_2$ on $\cC$ is central.
By construction, $\End_{\cC\boxtimes_{\cB_2}\cC^{\rm mp}}(\cC)\cong\cZ(\cC,\cB_2)$.
Hence $\cC$ is a $\cB_1$-$\cB_1$ Morita equivalence
$\cC\boxtimes_{\cB_2} \cC^{\rm mp} \simeq_{ME} \cZ(\cC,\cB_2^{\rev})$,
and 2-faithfulness is equivalent to $\cZ(\cC,\cB_2^{\rev})\cong \cB_1$.

Dualizing the above argument, 
we have
$\cC^{\rm mp}\boxtimes_{\cB_1} \cC \simeq_{ME}\cZ(\cC,\cB_1^{\rev})\cong \cB_2$,
and we are finished.
\qedhere
\end{proof}

 Taking Theorem~\ref{thm:invertibleIFFWittEq} and Remark~\ref{rem:defWittEqOverCorrect} together, we get the following result.
\begin{cor}\label{rem:davidStatement1}
Let $\cB$ be a braided multifusion category and $\cC$ a multifusion category with a fully faithful braided functor $\cB \to \cZ(\cC)$.
Then, the braided functor $\cB \boxtimes \cZ(\cC, \cB) \to \cZ(\cC)$ defines an invertible $1$-morphism $\cC: \cB \nrightarrow \cZ(\cC, \cB)^{\rev}$ in $\Mor_2^{\sss}$. In particular, $\cZ(\cC, \cB)$ has the same symmetric center as $\cB$. 
 Conversely, any invertible $1$-morphism in $\Mor_s^{\sss}$ out of $\cB$ is equivalent to one of this form. 
%
%
%
%
%
\end{cor}

We continue by introducing a number of other properties of 1-morphisms in the Morita 4-category that refine the above dominance conditions. Their importance will be seen later when they are used to characterize 2-dominant 1-morphisms and furthermore in the constructions of \S\ref{section:Wittsquares} used to define relevant subcategories of $\Mor^{\sss}_2$. 
In doing so, we will repeatedly consider the symmetric centers of the braided multifusion 1-categories under consideration. We therefore introduce the following abridged notation:

\begin{convention*}
    For a braided multifusion 1-category labeled $\cB_i$, we let $\cE_i:=\mathcal{Z}_{2}(\mathcal{B}_i)$ denote its symmetric center,
    and we write
    $\iota_i:\mathcal{E}_i\hookrightarrow \mathcal{B}_i$ 
    for the canonical inclusion.
\end{convention*}

In particular, let us record that, if $\cC:\cB_1\nrightarrow\cB_2$ is 1-dominant, so that the braided tensor functor $F_2:\mathcal{B}_2^{\rev}\rightarrow\cZ(\cC)$ is fully faithful, then we have that $\mathcal{E}_2$ is identified with $\Im(F_2|_{\mathcal{E}_2})$, its (dominant) image under $F_2$.

\begin{defn}
A 1-morphism $\cC:\cB_1\nrightarrow \cB_2$ in $\Mor_2^{\sss}$ is \emph{extensive} if it is 1-dominant and $\cZ(\cC,\cB_1\boxtimes\cB_2^{\rev})=\Im(F_2|_{\mathcal{E}_2})$.
\end{defn}

It follows from the definition that $\Im(F_1|_{\mathcal{E}_1})\subset \Im(F_2|_{\mathcal{E}_2})$ in $\cZ(\cC)$ for any extensive 1-morphism $\cC:\cB_1\nrightarrow \cB_2$. In particular, using that $\cE_2 := \cZ_{2}(\cB_2)$, one can extract a strong symmetric monoidal functor $\cE_1\rightarrow \cE_2$ from any extensive 1-morphism. For our purposes, it will be useful to single out and give an alternative characterization of the 1-morphisms satisfying this last property.

\begin{defn}\label{defn:stronglydominant}
A 1-morphism $\cC:\cB_1\nrightarrow\cB_2$ is \emph{strongly 1-dominant} if it is 1-dominant and $\Im(F_1|_{\mathcal{E}_1})\subset \Im(F_2|_{\mathcal{E}_2})$ in $\cZ(\cC)$.
\end{defn}

\begin{rem}
Every extensive 1-morphism $\cC:\cB_1\nrightarrow \cB_2$ is strongly 1-dominant, but the converse does not hold.
For example, given any non-trivial fusion 1-category $\cC$, the 1-morphism $\cC:\mathbf{Vect}\nrightarrow\mathbf{Vect}$ is strongly 1-dominant but not extensive.
\label{rem:extImpliesS1Dom}
\end{rem}

To a 1-morphism $\cC: \cB_1 \nrightarrow \cB_2$, we associate a symmetric lax monoidal functor $\mathcal{Z}_{2}(\mathcal{C}):\mathcal{E}_1\rightarrow \mathcal{E}_2$: 
\begin{equation}\label{eq:Muegerfunctor1morphisms}\cZ_{2}(\mathcal{C}):=\iota_2^*F_2^*F_1\iota_1:\cE_1\to \cE_2,\end{equation} where $F_2^*$ and $\iota_2^*$ are the right adjoints of $F_2$ and $\iota_2$, and are therefore symmetric lax monoidal.

We wish to understand when such a symmetric lax monoidal functor as~\eqref{eq:Muegerfunctor1morphisms} is strongly monoidal. In order to do so, we will use the following technical lemma.


\begin{lem}\label{lem:anonymous}
Let $\iota_{\cC}: \cC \hookrightarrow \cD$ be a fully faithful monoidal functor between multifusion $1$-categories, and $\cE\subseteq\cD$ the inclusion of a full multifusion sub-1-category.
Let $\iota_{\cC}^*: \cD \to \cC$ denote the (lax monoidal) right adjoint of $\iota_{\cC}$. 
Then$\iota_{\cC}^*|_{\cE}: \cE \to \cC$ is strongly monoidal if and only if the functor $\cE \subseteq\cD$ is contained in the full image of $\iota_{\cC}:\cC \hookrightarrow \cD$. 
\end{lem}
\begin{proof}
Since $\cC\subset \cD$ is full and $\cC,\cD$ are semisimple, we have a direct sum decomposition $\cD \simeq \cC \oplus \cR$ as 1-categories with $\mathbf{1}_\cD=\mathbf{1}_\cC\in\cC$, and observe that $\iota_{\cC}^*: \cD \simeq \cC \oplus \cR \to \cC$ is the corresponding projection. 
In particular, if $\cE\subset\cC$, then $\iota^*_\cC|_\cE: \cE\to \cC$ is the inclusion $\cE\hookrightarrow \cC$, which is strong monoidal.

Conversely, suppose that $\iota^*_\cC|_\cE$ is strong monoidal. 
Let $E\in\cE$ be a simple object. It is also simple in $\cD$ given that $\cE$ is a full subcategory.
Since $\cE$ is rigid, we have a direct summand $\mathbf{1}_\cC=\iota^*_\cC(\mathbf{1}_\cD)\subset \iota^*_\cC(E^*\otimes E)\cong \iota^*_\cC(E)^*\otimes \iota^*_\cC(E)$,
and thus $\iota^*_\cC(E)\neq 0$. This forces $\iota^*_\cC(E)\in\cC$, as $\iota_{\cC}^*$ is the projection onto $\cC$. 
\end{proof}

\begin{lem}\label{lem:characterizationstrongly1dominant}
Let $\cC: \cB_1 \nrightarrow \cB_2$ be a $1$-dominant $1$-morphism. If $\cC$ is strongly $1$-dominant then $\cZ_{2}(\cC): \cE_1\to \cE_2$ is strongly monoidal.
\end{lem}
\begin{proof}
We analyze the map 
 \[\cZ_{2}(F):=\iota_2^*F_2^*F_1\iota_1:\cZ_{2}(\cB_1)\to \cZ_{2}(\cB_2)\]
 where $F_1:\cB_1\to \cZ(\cC)$ and $F_2:\cB_2^{\rev}\to \cZ(\cC)$ are the action functors and $\iota_i:\cZ_{2}(\cB_i)\to\cB_i$ are the inclusions.
 The $1$-morphism $F$ is $1$-dominant if and only if $F_2$ is fully faithful, if and only if $F_2^*$ is monoidal, if and only if $F_2^*F_1\iota_1$ is strong monoidal. By assumption $F$ is 1-dominant and thus  $F_2^*F_1\iota_1$ is strong monoidal.  $\cZ_{2}(F)$ is strong monoidal if  $\iota_2^*|_{\Im(F_2^*F_1\iota_1)}$ is strong monoidal, which is true if  $F_1(\cZ_{2}(\cB_1))\subseteq F_2(\cZ_{2}(\cB_2))$. But this is the case since by assumption $F$ is also strongly $1$-dominant, so we conclude.
\end{proof}

In the next two results, we show that $\cZ_2$ is functorial on strongly $1$-dominant $1$-morphisms.
We first show that strongly 1-dominant 1-morphisms compose. In fact, we will prove a stronger statement in Proposition \ref{prop:Mugerfunctor} below.
\begin{lem}\label{lem:s1DomIFFZ2Monoidal}
The composite of two strongly $1$-dominant $1$-morphisms $\cC_1:\cB_1\nrightarrow\cB_2$ and $\cC_2:\cB_2\nrightarrow\cB_3$ is strongly $1$-dominant.
\end{lem}
\begin{proof}
 By Lemma~\ref{lem:0or1DomComposesMor2}, the composite of strongly $1$-dominant $1$-morphisms is $1$-dominant.
 It remains to show that the image of $\cE_1$ in $\cZ(\cC_1\boxtimes_{\cB_2}\cC_2)$ is contained in the image of $\cE_3$.
 Since $\cC_1$ is strongly $1$-dominant, $\Im(\cE_1)\subseteq\Im(\cE_2)$ in $\cZ(\cC_1)$, and since $\cC_2$ is strongly $1$-dominant, $\Im(\cE_2)\subseteq\Im(\cE_3)\subseteq \cZ(\cC_2)$.
 The map $\cZ(\cC_1)\boxtimes\cZ(\cC_2)\cong\cZ(\cC_1\boxtimes\cC_2)\to\cZ(\cC_1\boxtimes_{\cB_2}\cC_2)$ identifies the images of $\cE_2$ in each factor, so within $\cZ(\cC_1\boxtimes_{\cB_2}\cC_2)$, we have $\Im(\cE_1)\subseteq\Im(\cE_2)\subseteq\Im(\cE_3)$, completing the proof.
\end{proof}

Now we check the functoriality of $\cZ_2$.
\begin{lem}\label{lem:symmetriccompose}
Given strongly $1$-dominant $1$-morphisms $\cC_1:\cB_1\nrightarrow\cB_2$ and $\cC_2:\cB_2\nrightarrow\cB_3$ the composite of the associated symmetric monoidal functors $\cZ_{2}(\cC_1)$ and $\cZ_{2}(\cC_2)$ defined in~\eqref{eq:Muegerfunctor1morphisms} agrees with the symmetric monoidal functor $\cZ_{2}(\cC_1\boxtimes_{\cB_2} \cC_2)$ associated to their composite $1$-morphism. 
\end{lem}
\begin{proof}
 We write  $F:\mathcal{B}_1\boxtimes\mathcal{B}_2^{\rev}\rightarrow\mathcal{Z}(\cC_1)$, $G:\mathcal{B}_2\boxtimes\mathcal{B}_3^{\rev}\rightarrow\mathcal{Z}(\cC_2)$ for the two braided tensor functors providing $\cC_1$ and $\cC_2$ with their central structures. We use $F_1$, $F_2$ to denote the two restrictions of $F$ as above, and $G_1$, $G_2$ to denote the two restrictions of $G$. 
Let $K$ be the canonical commutative separable algebra in $\mathcal{B}_2^{\rev}\boxtimes\mathcal{B}_2$, i.e.\ $\mathbf{Mod}_{K}(\mathcal{B}_2^{\rev}\boxtimes\mathcal{B}_2)\simeq\mathcal{B}_2$ as multifusion 1-categories. Let $\tilde{K}$ be the image of $K$ in $\mathcal{Z}(\mathcal{C}_1)\boxtimes \mathcal{Z}(\mathcal{C}_2)$. Thenwe have $$\mathbf{Mod}^{\loc}_{\tilde{K}}
(\mathcal{Z}(\mathcal{C}_1)\boxtimes \mathcal{Z}(\mathcal{C}_2))\simeq \mathcal{Z}(\mathbf{Mod}^{\loc}_{\tilde{K}}(\mathcal{C}_1\boxtimes \mathcal{C}_2))$$ as braided multifusion 1-categories by \cite[Theorem 3.20]{DNO2013}. The multifusion 1-category $\mathcal{D}:=\mathbf{Mod}^{\loc}_{\tilde{K}}(\mathcal{C}_1\boxtimes \mathcal{C}_2)$ equipped with the canonical braided functor $H:\mathcal{B}_1\boxtimes\mathcal{B}_3^{\rev}\rightarrow\mathcal{Z}(\mathcal{D})$ given by $X\boxtimes Z\mapsto \tilde{K}\otimes (F_1(X)\boxtimes G_2(Z))$ represents the composite of $\cC_1: \cB_1 \nrightarrow \cB_2$ and $\cC_2: \cB_2 \nrightarrow \cB_3$ in $\mathbf{Mor}_2^{ss}$. We write $H_1$, $H_2$ for the two restrictions of $H$. Then, given any objects $X$ in $\cB_1$ and $Z$ in $\cB_3$, we have natural isomorphisms
\begin{eqnarray*}
    \Hom_{\cZ(\cD)}(H_1(X),\, H_2(Z)) 
&\cong& \Hom_{\tilde{K}} (\tilde{K} \otimes ( F_1(X) \boxtimes \mathbf{1}),\, \tilde{K}\otimes (\mathbf{1} \boxtimes G_2(Z))) \\
&\cong & \Hom_{\cZ(\cC_1)\boxtimes \cZ(\cC_2)} ( F_1(X) \boxtimes \mathbf{1},\,  \tilde{K}\otimes (\mathbf{1} \boxtimes G_2(Z))) \\
&\cong& \Hom_{\cZ(\cC_1)\boxtimes \cZ(\cC_2)}  ( F_1(X) \boxtimes G_2(Z^*),\, \tilde{K}) \\
&\cong& \Hom_{\cZ(\cC_1)\boxtimes \cZ(\cC_2)}  ( F_1(X) \boxtimes G_2(Z^*),\, (F_2\boxtimes G_1)(K)) \\
&\cong& \Hom_{\cB_2^{\rev}\boxtimes \cB_2} ( F_2^*F_1(X) \boxtimes G_1^*G_2(Z^*),\, K) \\
&\cong& \Hom_{K} ((F_2^*F_1(X) \boxtimes \mathbf{1})\otimes K,\,(\mathbf{1}\boxtimes G_1^*G_2(Z))\otimes K)\\
&\cong& \Hom_{\cB_2} ( F_2^*F_1(X),\, G_1^*G_2(Z)).
\end{eqnarray*}
By Yoneda's lemma this shows that $$H_2^* H_1 \cong G_2^*G_1F_2^* F_1$$ as lax tensor functors $\cB_1\to \cB_3$. 
It then follows that 
$$
\cZ_{2}(\cC_2)  \cZ_{2}(\cC_1) =
\iota_{3}^* G_3^* G_2 \iota_{2} \iota_{2}^* F_2^* F_1 \iota_{1} 
= \iota_{3}^* G_3^* G_2  F_2^* F_1 \iota_{1} \cong 
\iota_3^* H_2^* H_1 \iota_1 = \cZ_{2}(\cD).
$$
is an isomorphism of symmetric lax monoidal functors,
where we used that the strong $1$-dominance implies that
$F_2^*F_1$ maps $\cE_1$ to $\cE_2$ and that $\iota_{2} \iota_{2}^*$
is the identity on $\cE_2$.
\end{proof}

\begin{defn}
    A 1-morphism $\cC:\cB_1 \nrightarrow \cB_2$ in $\Mor_2^{\sss}$ is \textit{factorizable} if it is $1$-dominant and $\Im(F_1|_{\mathcal{E}_1}) = \Im(F_2|_{\mathcal{E}_2})$ in $\cZ(\cC)$.
\end{defn}
\begin{rem}\label{rem:factorizable1morphisms}
Every factorizable 1-morphisms is strongly 1-dominant.  The term factorizable is motivated by the fact that, given $\cC:\cB_1 \nrightarrow \cB_2$ a factorizable 1-morphism, the canonical map $\cB_2^{\rev}\to\Im(F)\subseteq\cZ(\cC)$ (which is an inclusion by $1$-dominance) is the inclusion of a tensor factor, {i.e.} $\Im(F)\cong \Im(F_1)\boxtimes_{\cE_2}\cB_2^{\rev}$, which follows from \cite[Prop.~{4.3}]{DNO2013}.
\end{rem}

\begin{lem}
 \label{lem:factorizableComposes}
 The composite of two factorizable 1-morphisms is factorizable.
\end{lem}
\begin{proof}
Let $\cC_1:\cB_1\nrightarrow\cB_2$ and $\cC_2:\cB_2\nrightarrow\cB_3$ be two factorizable 1-morphisms. By Lemma~\ref{lem:s1DomIFFZ2Monoidal}, the composite of two strongly $1$-dominant morphisms is strongly $1$-dominant, so $\cC_1\boxtimes_{\cB_2}\cC_2:\cB_1\nrightarrow\cB_3$ is strongly $1$-dominant. By construction, the images of $\cE_2$ under both $\mathcal{B}_2^{\rev}\rightarrow \cZ(\cC_1)$ and $\mathcal{B}_2\rightarrow \cZ(\cC_2)$ are identified in $\cZ(\cC_1\boxtimes_{\cB_2}\cC_2)$. But, due to the fact that we have assumed that both $1$-morphisms are strongly $1$-dominant, the image of $\cE_2$ under $\mathcal{B}_2^{\rev}\rightarrow \cZ(\cC_1)$ coincides with the image of $\cE_1$ under $\mathcal{B}_1\rightarrow \cZ(\cC_1)$ and the image of $\cE_3$ under $\mathcal{B}_3^{\rev}\rightarrow \cZ(\cC_2)$ coincides with the image of $\cE_2$ under $\mathcal{B}_2\rightarrow \cZ(\cC_2)$. This concludes the proof.
\end{proof}

\begin{lem}\label{lem:mondom}
A $1$-morphism $\cC:\cB_1\nrightarrow\cB_2$ is factorizable if and only if $\cZ_{2}(\cC):\cZ_{2}(\cB_1)\to \cZ_{2}(\cB_2)$ is strong monoidal and dominant.
\end{lem}
\begin{proof}
     As $\cC:\cB_1\nrightarrow\cB_2$ is strongly $1$-dominant, we have a strongly symmetric monoidal functor $\cZ_{2}(\cC):\cZ_{2}(\cB_1)\to \cZ_{2}(\cB_2)$.
     Since $\cZ_{2}(\cB_2)$ is equivalent to its image in $\cZ(\cC)$,
     factorizability is equivalent to requiring that $\cZ_{2}(\cC)$ is dominant.
\end{proof}

We now explore the relation between factorizability, extensiveness, and 2-dominance. 

\begin{lem}\label{lem:fact+ext=2dom}
A $1$-morphism is factorizable and extensive if and only if it is $2$-dominant.
\end{lem}
\begin{proof}

We begin by noting that the conditions of $2$-dominance, factorizability, and extensiveness all imply $1$-dominance, so that it is enough to prove the result in the case when all the multifusion 1-categories under consideration are in fact fusion.

We next show that $2$-dominance implies extensiveness. By definition, a $1$-morphism $\cC:\cB_1\nrightarrow\cB_2$ is $2$-dominant if and only if the braided functor $F_2:\cB_2^{\rev}\to \cZ(\cC,\cB_1)$ is an equivalence. Thus, if $\cC:\cB_1\nrightarrow\cB_2$ is $2$-dominant, then $\cZ(\cC,\cB_1\boxtimes\cB_2^{\rev})$ is the centralizer of $\cB_2$ in itself, i.e.\ $\cE_2$, establishing that the $1$-morphism under consideration is indeed extensive. It also follows from the fact that $\cZ_{2}(\Im(\cB_1))\cong\Im(\cE_1)$, which holds by inspection (see also \cite[Cor.~3.24]{DMNO}), that a 2-dominant 1-morphism is factorizable.

To finish the proof, it is enough to show that a $1$-morphism $\cC:\cB_1\nrightarrow\cB_2$ that is both extensive and factorizable is $2$-dominant. It follows from factorizability (see Remark \ref{rem:factorizable1morphisms}) that there is an equivalence $\cZ(\cC,\cB_1)\cong\cB_2^{\rev}\boxtimes_{\cE_2}\cZ(\cC,\cB_1\boxtimes\cB_2^{\rev})$. Extensiveness implies that the latter factor is just $\cE_2$, so that $\cZ(\cC,\cB_1)\cong\cB_2^{\rev}$, showing that the $1$-morphism $\cC:\cB_1\nrightarrow\cB_2$ is $2$-dominant.
This concludes the proof.
\end{proof}

With the result of the previous lemma at our disposal, we are now in a position to prove that the composite of two $2$-dominant, resp.~$2$-faithful, $1$-morphisms is again $2$-dominant, resp. $2$-faithful.

\begin{lem}
    Let $\mathcal{C}_1:\mathcal{B}_1\nrightarrow\mathcal{B}_2$ and $\mathcal{C}_2:\mathcal{B}_2\nrightarrow\mathcal{B}_3$ be two 2-dominant, resp.\ $2$-faithful, 1-morphisms, then the composite $\cC_1\boxtimes_{\cB_2}\cC_2:\cB_1\nrightarrow\cB_3$ is $2$-faithful, resp.\ $2$-dominant.
\end{lem}

\begin{proof}
    Suppose that $\cC_1:\cB_1\nrightarrow\cB_2$ and $\cC_2:\cB_2\nrightarrow\cB_3$ are $2$-dominant $1$-morphisms between braided multifusion $1$-categories, or equivalently by Lemma~\ref{lem:fact+ext=2dom}, that they are extensive and factorizable. As before, since $2$-dominant morphisms are $1$-dominant, we may reduce to the case where each $\cB_i$ and $\cC_j$ is fusion, rather than multifusion.

By Lemma~\ref{lem:fact+ext=2dom}, in order to show that $\cC_1\boxtimes_{\cB_2}\cC_2:\cB_1\nrightarrow\cB_3$ is $2$-dominant, it suffices to show that it is also extensive and factorizable. The composite is factorizable by Lemma~\ref{lem:factorizableComposes}, so that we only need to verify extensiveness. Observe that
\[\cZ(\cC_1\boxtimes\cC_2,\cB_1\boxtimes\cB_3^{\rev})\cong \cZ(\cC_1,\cB_1)\boxtimes \cZ(\cC_2,\cB_3^{\rev}).\] Moreover, by $2$-dominance, we have $\cZ(\cC_1,\cB_1)\cong\cB_2^{\rev}$ and $\cZ(\cC_2,\cB_2)\cong\cB_3^{\rev}$.
By Lemma~\ref{lem:BF1CDoubleCommutant}, we find $\cZ(\cC_2,\cB_3^{\rev})\cong\cZ(\cC_2,\cZ(\cC_2,\cB_2))\cong\Im(F_2)$.
As a consequence, we have \[\cZ(\cC_1\boxtimes\cC_2,\cB_1\boxtimes\cB_3^{\rev}) \cong\cB_2^{\rev}\boxtimes\Im(\cB_2).\]
The image $\widetilde{K}$ of the canonical étale algebra $K$ in $\cB_2^{\rev}\boxtimes\cB_2$ introduced in the proof of Lemma \ref{lem:0or1DomComposesMor2} centralizes $\cB_1\boxtimes\cB_3^{\rev}$ inside $\cZ(\cC_1\boxtimes\cC_2)$. We therefore have that the centralizer of $\cB_1\boxtimes\cB_3^{\rev}$ inside $\Mod_{\widetilde{K}}^{\loc}(\cZ(\cC_1\boxtimes\cC_2))$ is equivalent to the category of local modules in the centralizer:
\begin{align*}
\cZ(\cC_1\boxtimes_{\cB_2}\cC_2,\cB_1\boxtimes\cB_3^{\rev}) &\cong \Mod_{\widetilde{K}}^{\loc}(\cZ(\cC_1\boxtimes\cC_2,\cB_1\boxtimes\cB_3^{\rev}))
\\&\cong \Mod_{\widetilde{K}}^{\loc}(\cB_2^{\rev}\boxtimes\Im(F_2))\,.
\end{align*}
Now, recall that
\[\Mod_{\widetilde{K}}(\cB_2^{\rev}\boxtimes\Im(F_2))\cong\cB_2^{\rev}\boxtimes_{\cB_2}\Im(F_2)\cong\Im(F_2).\]
Furthemore, as the maximal subalgebra of $\widetilde{K}$ contained in $\cB_2^{\rev}$ is trivial, local modules of $\widetilde{K}$ are all in the image of $\cE_2\subseteq\cB_2^{\rev}$ under the free $\widetilde{K}$-module functor, which yields
\[\cZ(\cC_1\boxtimes_{\cB_2}\cC_2,\cB_1\boxtimes\cB_3^{\rev})\cong\Mod_{\widetilde{K}}^{\loc}(\cB_2^{\rev}\boxtimes\Im(F_2))\cong\Im(F_2|_{\mathcal{E}_2}).\]

Let $G:\mathcal{B}_2\boxtimes\mathcal{B}_3^{\rev}\rightarrow\mathcal{Z}(\cC_2)$, and $G_i$ for $i=2,3$ denotes its restriction.
Finally, since $\cC_2:\cB_2\nrightarrow\cB_3$ is factorizable, we have $\Im(F_2|_{\mathcal{E}_2})\cong\Im(G_3|_{\mathcal{E}_3})$. But, $\cC_2:\cB_2\nrightarrow\cB_3$ is $2$-dominant, so that $\Im(G_3)\cong\cB_3$ and, in particular, $\Im(G_3|_{\mathcal{E}_3})\cong\cE_3$. Putting the above discussion together, we find that \[\cZ(\cC_1\boxtimes_{\cB_2}\cC_2,\cB_1\boxtimes\cB_3^{\rev}) \cong\cE_3,\] which shows that the composite $\cC_1\boxtimes_{\cB_2}\cC_2:\cB_1\nrightarrow\cB_3$ is extensive, concluding the proof.
    
 The proof that the composite of $2$-faithful morphisms is $2$-faithful is entirely dual.
\end{proof}

\subsubsection{Adjectives for Multifusion 2-Categories}\label{subsub}

We now move to 
multifusion 2-categories, and give the corresponding definitions of dominance and faithfulness for 1-morphisms in the Morita 4-category $\mathbf{Mor}^{\sss}$.

\begin{defn}\label{def:dominanceMF2C}
Let $\mathfrak{C}$ and $\mathfrak{D}$ be two multifusion 2-categories, and let $\mathfrak{M}:\mathfrak{C}\nrightarrow\mathfrak{D}$ be 
a 1-morphism in the Morita 4-category $\Mor^{\sss}$.

\begin{enumerate}[label=(\alph*)]
    \item The $1$-morphism $\mathfrak{M}$ is \textit{0-dominant} if $\mathfrak{M}$ is faithful as a right $\mathfrak{D}$-module 2-category, i.e.\ the canonical 2-functor $\mathfrak{D}^{\mop}\rightarrow\mathbf{End}_{\fC}(\mathfrak{M})$ is faithful on 2-morphisms.\footnote{This is also equivalent to asking that the canonical 2-functor $\mathfrak{D}^{\mop}\rightarrow\mathbf{End}(\mathfrak{M})$ be faithful on 2-morphisms.}
    \item The $1$-morphism $\mathfrak{M}$ is \textit{1-dominant} if the canonical 2-functor $\mathfrak{D}^{\mop}\rightarrow\mathbf{End}_{\mathfrak{C}}(\mathfrak{M})$ is faithful as a 2-functor, i.e.\ it is fully faithful on 2-morphisms.
    \item The $1$-morphism $\mathfrak{M}$ is \textit{2-dominant} if the canonical 2-functor $\mathfrak{D}^{\mop}\rightarrow\mathbf{End}_{\mathfrak{C}}(\mathfrak{M})$ is fully faithful as a 2-functor, i.e.\ it induces equivalences on $\Hom$-1-categories.
\end{enumerate}
\end{defn}

\begin{defn}\label{defn:faithfulMF2C}
Let $\mathfrak{C}$ and $\mathfrak{D}$ be two multifusion 2-categories, and let $\mathfrak{M}:\mathfrak{C}\nrightarrow\mathfrak{D}$ be a 1-morphism in the Morita 4-category $\Mor^{\sss}$.

\begin{enumerate}[label=(\alph*)]
    \item The 1-morphism $\mathfrak{M}$ is \textit{0-faithful} if $\mathfrak{M}$ is faithful as a left $\mathfrak{C}$-module 2-category, i.e.\ the canonical 2-functor $\mathfrak{C}\rightarrow\mathbf{End}_{\cD}(\mathfrak{M})$ is faithful on 2-morphisms.\footnote{This is also equivalent to asking that the canonical 2-functor $\mathfrak{C}\rightarrow\mathbf{End}(\mathfrak{M})$ be faithful on 2-morphisms.}
    \item $\mathfrak{M}$ is \textit{1-faithful} if the canonical 2-functor $\mathfrak{C}\rightarrow\mathbf{End}_{\mathfrak{D}}(\mathfrak{M})$ is faithful as a 2-functor, i.e.\ it is fully faithful on 2-morphisms.
    \item  $\mathfrak{M}$ is \textit{2-faithful} if the canonical 2-functor $\mathfrak{C}\rightarrow\mathbf{End}_{\mathfrak{D}}(\mathfrak{M})$ is fully faithful as a 2-functor, i.e.\ it induces equivalences on $\Hom$-1-categories.
\end{enumerate}
\end{defn}

As stated, the above definitions refer to the underlying  multifusion 2-categories $\fC$ and $\fD$ and are thus not obviously well-defined conditions on a $1$-morphism in the (univalent) $4$-category $\Mor^{\sss}$. As a counterpart to Theorem \ref{thm:invertibleIFFWittEq}, we also have the following lemma, which follows from \cite[Theorem 5.4.3]{Decoppet2022;Morita}.

\begin{lem}\label{lem:invertible2faith2dom}
    Any invertible 1-morphism in $\Mor^{\sss}$ is both 2-faithful and 2-dominant.
\end{lem}

As in the previous subsection, the well-definedness of the above faithfulness and dominance conditions will now follow once we have proven that they are closed under composition. But, we have proven in \S\ref{subsub:BF1C} that $n$-dominant and $n$-faithful 1-morphisms in $\Mor_2^{\sss}$ compose. We now make use of these results to establish the corresponding statements about 1-morphisms of $\Mor^{\sss}$, and note that this also allows us to define subcategories $\Mor^{\sss}$.

\begin{prop}\label{prop:dominancethesame}
    A $1$-morphism $\cC: \cB_1 \nrightarrow \cB_2$ in $\Mor_2^{\sss}$ is $n$-faithful/$n$-dominant in the sense of Definitions~\ref{defn:domBF1} and~\ref{defn:faithBF1} if and only if the corresponding $1$-morphism $\Mod(\cC): \Mod(\cB_1) \nrightarrow \Mod(\cB_2)$ in $\Mor^{\sss}$ is $n$-faithful/$n$-dominant in the sense of Definitions~\ref{def:dominanceMF2C} and~\ref{defn:faithfulMF2C}. 
\end{prop}
\begin{proof}
Recall from \cite[Lemma 3.2.1]{Decoppet2022;centers} that $\Omega\mathbf{End}_{\mathbf{Mod}(\cB_1)}(\mathbf{Mod}(\cC))\simeq \mathcal{Z}(\mathcal{C},\cB_1)^{\rev}$. Thus, if $\Mod(\cC)$ is $n$-dominant as a $1$-morphism in $\Mor^{\sss}$, then so is $\cC$ as a $1$-morphism in $\Mor_2^{\sss}$. The converse follows since by connectivity a monoidal $2$-functor $\Mod(\cB_1) \to \fD$ into a multifusion $2$-category $\fD$ is faithful on $2$-morphisms/fully faithful on $2$-morphisms/fully faithful  if and only if the corresponding braided functor $\cB_1 \to \Omega \fD$ is faithful/fully faithful/an equivalence.
The proof for $n$-faithfulness works dually. 
\end{proof}

Throughout this paper, we will (often implicitly) use this result. 

\begin{thm}\label{thm:nDomComposesMor1}
Let $\fC_1$, $\fC_2$, and $\fC_3$ be multifusion $2$-categories, and let $\fM_1:\fC_1\nrightarrow\fC_2$ and $\fM_2:\fC_2\nrightarrow\fC_3$
be two $n$-dominant, resp\ $n$-faithful, $1$-morphisms for some $n\in\{0,1,2\}$. Then the composite $\fM_1\boxtimes_{\fC_2}\fM_2:\fC_1\nrightarrow\fC_3$ is $n$-dominant, resp.\ $n$-faithful.
\end{thm}
\begin{proof}
As before, we only consider the case of dominant 1-morphisms, the case of faithful 1-morphisms being entirely dual.
Our strategy will be as follows:\ We begin by showing that each dominance property is preserved by composition with Morita equivalences. Then, as every fusion $2$-category is Morita equivalent to a connected one \cite{Decoppet2022;centers}, we can reduce ourselves to the case where each $\fC_i$ is connected, {i.e.} $\fC_i=\mathbf{Mod}(\cB_i)$, where $\cB_i$ is a braided multifusion $1$-category.
Then,  the finite semisimple bimodule 2-categories $\fM_1:\fC_1\nrightarrow\fC_2$ and $\fM_2:\fC_2\nrightarrow\fC_3$ must be of the form $\mathbf{Mod}(\mathcal{C}_1)$ and $\mathbf{Mod}(\mathcal{C}_2)$, where $\mathcal{C}_1:\mathcal{B}_1\nrightarrow\mathcal{B}_2$ and $\mathcal{C}_2:\mathcal{B}_2\nrightarrow\mathcal{B}_3$ are 1-morphisms in $\mathbf{Mor}_2^{\sss}$. By Proposition~\ref{prop:dominancethesame}, the result then follows readily from Proposition \ref{prop:nDomComposesMor2}.

We claim that the condition of being $n$-dominant for $\fM:\fC\nrightarrow\fD$ is invariant upon tensoring with a Morita invertible bimodule 2-category $\fL:\fB\nrightarrow\fC$ on the left. Namely, the canonical monoidal 2-functor
\[\mathbf{End}_{\fC}(\fM)\rightarrow\mathbf{End}_{\fB}(\fL\boxtimes_{\fC}\fM)\]
is an equivalence by \cite[Proposition 3.3.1]{Decoppet2023;dualizable} (see also \cite[Theorem 5.4.3]{Decoppet2022;Morita}).
Now, let $\fN:\fD\nrightarrow\fE$ be an invertible bimodule 2-category. Said differently, there is a separable algebra $B$ in $\mathfrak{D}$, an equivalence $\mathfrak{N}\simeq \mathbf{Mod}_{B}(\mathfrak{D})$ of left $\mathfrak{D}$-module 2-categories, and an equivalence $\mathfrak{E}\simeq \mathbf{Bimod}_{\mathfrak{D}}(B)$ of multifusion 2-categories \cite{Decoppet2022;Morita}. There is also a separable algebra $A$ in $\mathfrak{C}$ such that the equivalence $\mathfrak{M}\simeq \mathbf{Mod}_{A}(\mathfrak{C})$ is an equivalence of left $\mathfrak{C}$-module 2-categories \cite{decoppet2021finite, Decoppet2022;centers}. Let $F:\mathfrak{D}^{\mop}\rightarrow \mathbf{End}_{\mathfrak{C}}(\mathfrak{M})\simeq \mathbf{Bimod}_{\mathfrak{C}}(A)^{\mop}$ be the monoidal 2-functor supplying the left $\fC$-module $\mathfrak{M}$ with its compatible right $\mathfrak{D}$-module structure.
We have that $F(B)$ is a separable algebra in $\mathbf{Bimod}_{A}(\mathfrak{C})$, i.e.\ $F(B)$ is a separable algebra in $\mathfrak{C}$ equipped with an algebra 1-homomorphism $A\rightarrow F(B)$. Thenit follows from \cite{Decoppet2022;Morita} that there are equivalences of multifusion 2-categories $$\mathbf{End}_{\mathfrak{C}}(\mathfrak{M}\boxtimes_{\mathfrak{D}}\mathfrak{N})\simeq \mathbf{End}_{\mathfrak{C}}(\mathbf{Mod}_{B}(\mathfrak{M}))\simeq \mathbf{End}_{\mathfrak{C}}(\mathbf{Mod}_{F(B)}(\mathfrak{C}))\simeq \mathbf{Bimod}_{F(B)}(\mathfrak{C})^{\mop}.$$ Moreover, the monoidal 2-functor $\mathfrak{E}^{\mop}\rightarrow \mathbf{End}_{\mathfrak{C}}(\mathfrak{M}\boxtimes_{\mathfrak{D}}\mathfrak{N})$ providing $\mathfrak{M}\boxtimes_{\mathfrak{D}}\mathfrak{N}$ with its right action is identified with the monoidal 2-functor $$\mathbf{Bimod}_{B}(F):\mathbf{Bimod}_{B}(\mathfrak{D})\rightarrow \mathbf{Bimod}_{F(B)}(\mathbf{Bimod}_{A}(\mathfrak{C}))\simeq \mathbf{Bimod}_{F(B)}(\mathfrak{C}).$$
It is therefore enough to check that if $F$ is faithful on 2-morphisms, resp.\ fully faithful on 2-morphisms, resp.\ induces equivalences on $\Hom$-1-categories, then so is $\mathbf{Bimod}_{B}(F)$. But, if $F$ has any of these three properties, then $\mathbf{Bimod}_{B}(F)$ has the corresponding property between free bimodules. But, $\mathbf{Bimod}_{B}(\mathfrak{D})$ is the Cauchy completion of its full sub-2-category on the free bimodules (this is a consequence of the proof of \cite[Proposition 3.1.3]{Decoppet2022;rigid}) and all the above three properties are preserved by taking the Cauchy completion (this follows by inspecting the explicit definition of the Cauchy completion given in \cite[Appendix A]{decoppet2020;comparison}), so that $\mathbf{Bimod}_{B}(F)$ indeed has the desired property. This concludes the proof.
\end{proof}

\begin{rem}
As a consequence of Theorem \ref{thm:nDomComposesMor1}, we see that a $1$-morphism in $\Mor_2$ it is $n$-faithful, resp.\ $n$-dominant, iff its image under $\Mod$ in $\Mor^{\sss}$ is $n$-faithful, resp.\ $n$-dominant.
\end{rem}

\subsection{Passing to the Super World}\label{subsection:supergroups}
The formulation of our classification result in the case of emergent fermions necessitates passing to the super world for groups, spaces, and cohomology. 
We presently compile relevant facts 
for completeness and convenience of the reader.

\begin{defn}\label{def:fermionic_symmetry}
    A \textit{finite supergroup} is a finite group $G$ equipped with a central element $z$ of order at most 2, i.e.\ a group homomorphism $\Z/2 \to Z(G)$.
\end{defn}

    Throughout, we let $\mathrm{B} \mathbb{Z}/2$ denote the $2$-group whose underlying groupoid has a single object with automorphisms $\mathbb{Z}/2$ and with its unique  monoidal structure induced by the fact that $\mathbb{Z}/2$ is abelian. The following well-known observation was already mentioned in the introduction:
\begin{lem}
    The $2$-group $\Aut^{\br}(\sVect)$ is equivalent to the $2$-group $\mathrm{B}\mathbb{Z}/2$.
\end{lem}
\begin{proof}
    Every braided autoequivalence of $\sVect$ is equivalent to the identity and the group of monoidal natural automorphisms of the identity functor is $\mathbb{Z}/2$.
\end{proof}

\begin{lem}\label{lemma:supergroup}
Let $G$ be a finite group. Then the data of a $\mathrm{B} \mathbb{Z}/2$-action on 
the groupoid $\mathrm{B}G$ is equivalently a homomorphism $\mathbb{Z}/2 \to Z(G)$, i.e.\ a supergroup structure on $G$. 
\end{lem}
\begin{proof}
    A $\mathrm{B}\mathbb{Z}/2$-action on $\mathrm{B}G$ unpacks to a pointed map $\mathrm{B}^2\mathbb{Z}/2 \to \mathrm{B} \Aut(\mathrm{B}G)$, i.e.\ equivalently to a group homomorphism $\mathbb{Z}/2 \to \Omega \Aut(\mathrm{B}G) = Z(G)$.
\end{proof}

Generalizing Lemma~\ref{lemma:supergroup}, we may therefore define: 
\begin{defn}
 A \emph{superspace} is a space $X$ together with a $\mathrm{B}\mathbb{Z}/2$-action. We let $\textbf{sSpaces}:=\mathrm{Fun}(\mathrm{B}^2 \Z/2, \textbf{Spaces})$ denote the $\infty$-category of superspaces. We say that a superspace is \emph{$k$-truncated} if its underlying space $X$ is $k$-truncated\footnote{This is equivalent to asking it to be $k$-truncated as an object of the $\infty$-category $\textbf{sSpaces}$ of superspaces.}, i.e.\ has vanishing homotopy groups in degrees $>k$. 
\end{defn}

\begin{ex}
    A supergroup is precisely the data of a pointed connected $1$-truncated superspace. 
\end{ex}
For most superspaces we consider, the $\mathrm{B}\mathbb{Z}/2$-action arises as an action by $\Aut^{\br}(\sVect)$, whence the name.

Recall the straightening/unstraightening equivalence 
$$
\textbf{sSpaces}:= \mathrm{Fun}(\mathrm{B}^2 \mathbb{Z}/2 , \textbf{Spaces}) \simeq \textbf{Spaces}_{/\mathrm{B}^2\Z/2}
$$ 
to the over-category of spaces over $\mathrm{B}^2 \Z/2$ sending a superspace $X$ with its $\mathrm{B}\mathbb{Z}/2$-action to the homotopy quotient space $X\sslash \mathrm{B}\mathbb{Z}/2$ with its induced map to $\mathrm{B}^2\Z/2$ and conversely, sending a map $Y \to \mathrm{B}^2\mathbb{Z}/2$ to its homotopy fiber $F$ with induced $\mathrm{B}\Z/2$-action. We will frequently pass back and forth along this equivalence. In particular, the terminal superspace $\pt$ with trivial $\mathrm{B}\Z/2$-action corresponds to the map $\mathrm{B}^2\Z/2 \to \mathrm{B}^2\Z/2$ while the ``regular'' superspace $\mathrm{B}\Z/2$ with its free action corresponds to the map $\pt \to \mathrm{B}^2\Z/2$. We point out that the former object is a ``pointed object'' in \textbf{sSpaces}, i.e.\ admits a map from the terminal (it in fact is the terminal), while the latter object ``has no points,'' i.e.\ no map from the terminal object.

See Figure \ref{fig:superspaces} for a summary of how common superspaces can be viewed from both perspectives.

\begin{ex}
Let $(G,z)$ be a supergroup with $z=1$, i.e.\ with associated superspace $\mathrm{B}G$ with trivial $\mathrm{B}\Z/2$-action. Then the corresponding object in $\textbf{Spaces}_{/\mathrm{B}^2\Z/2}$ is given by the projection $\mathrm{B}G \times \mathrm{B}^2\Z/2 \to \mathrm{B}^2\Z/2$. 
In particular, we emphasize that $1$-truncatedness of a superspace $X$ does not imply $1$-truncatedness of the homotopy quotient $X/(\mathrm{B}\Z/2)$. 

If $(G,z)$ is a supergroup with $z\neq 1$, i.e.\ with associated superspace $\mathrm{B}G$ with non-trivial $\mathrm{B}\Z/2$-action, then let $G_b:= G/z$ with central extension $G$ determined by a $2$-cocycle $\kappa \in \mathrm{H}^2(G_b, \Z/2).$ Then the corresponding object in $\textbf{Spaces}_{/\mathrm{B}^2\Z/2}$ is given by the map $\mathrm{B}G_b \to \mathrm{B}^2\Z/2$ classifying $\kappa$. 
\end{ex}

For our present considerations, supergroups arise via (a special case of) Deligne's theorem \cite{deligne2002}. More precisely, every non-zero symmetric multifusion 1-category $\mathcal{E}$ admits a symmetric monoidal functor to $\mathbf{sVect}$. Further, if $\mathcal{E}$ is fusion, then any two such functors are naturally isomorphic. Said differently, the $1$-groupoid $\Spec(\mathcal{E})$ of symmetric monoidal functor $\mathcal{E}\rightarrow \mathbf{sVect}$ and symmetric natural isomorphisms is non-empty if $\mathcal{E}$ is non-zero, and it is connected if $\mathcal{E}$ is fusion. Moreover, it comes equipped with a canonical action of $\Aut^{\br}(\sVect)= \mathrm{B}\Z/2$ by postcomposition, that is, we have a functor $\Spec: \mathbf{SMF1C}^{op}\rightarrow \tau_{\leq 1}\mathbf{sSpaces}$ to the $\infty$-category of $1$-truncated superspaces. Deligne's theorem may then be re-expressed in the following convenient form, which will be used when unpacking Delphics squares. 
\begin{thm}[\cite{deligne2002}]
     The functor $\Spec:\mathbf{SMF1C}^{\op}\rightarrow \tau_{\leq 1}\mathbf{sSpaces}$ is an equivalence.
\end{thm}

\begin{figure}
\begin{tabular}{c|c|c}
     $\infty$\text{-category of superspaces} & \textbf{Spaces}$_{/\mathrm{B}^2\Z/2}$ &  \text{Spaces with} $\mathrm{B}\Z/2$ \text{action} \\

    \hline
    $\Spec(\sVect)$ & $\mathsf{pt} \rightarrow  \mathrm{B}^2\Z/2$
    &  \begin{tikzcd}
   \mathrm{B}\Z/2 \arrow[loop left,"\mathrm{B}\Z/2"]
    \end{tikzcd}\\
     
     \hline 
    $\Spec(\cE)$, $\mathcal{E}$ super-Tannakian & $\mathrm{B}G_b \rightarrow \mathrm{B}^2\Z/2$ &  \begin{tikzcd}
   \mathrm{B}(G,z) \arrow[loop left,"\mathrm{B}\Z/2"]
    \end{tikzcd}\\
    
    \hline
    $\Spec(\Vect)$ & $\mathrm{B}^2\Z/2 \rightarrow \mathrm{B}^2\Z/2$ & \begin{tikzcd}
  \mathsf{pt} \arrow[loop left,"\mathrm{B}\Z/2"]
    \end{tikzcd}
    \\

\hline 
    $\Spec(\cE)$, $\mathcal{E}$ Tannakian & $\mathrm{B}G\times \mathrm{B}^2\mathbb{Z}/2 \rightarrow \mathrm{B}^2\Z/2$ &  \begin{tikzcd}
   \mathrm{B}G \arrow[loop left,"\mathrm{B}\Z/2"]
    \end{tikzcd}\\
    
\end{tabular}
\caption{\label{fig:superspaces}Different models for $\Spec$ of the symmetric fusion 1-category $\mathcal{E}$}
\end{figure}

Recall the $(2,1)$-category $\mathbf{BMF1C}$ of braided multifusion $1$-categories, braided functors and monoidal natural isomorphisms. Its underlying $2$-groupoid $\ob(\mathbf{BMF1C})$ admits a functor $\cZ_{2}: \ob(\mathbf{BMF1C}) \to \ob(\mathbf{SMF1C})$ sending a braided multifusion $1$-category $\cB$ to its symmetric center and a braided equivalence to the corresponding equivalence of symmetric centers. 

\begin{defn}
We define the $2$-groupoid $\mathrm{BMF1C}^{\ndeg}(\sVect)$ as the fiber of $\ob(\mathbf{BMF1C}) \to \ob(\mathbf{SMF1C})$ at $\sVect \in \ob(\mathbf{SMF1C})$. 
\end{defn}
Unpacked, $\mathrm{BMF1C}^{\ndeg}(\sVect)$ is therefore the $2$-groupoid of braided multifusion $1$-categories $\cB$ equipped with an equivalence $\sVect \to \cZ_{2}(\cB)$ and  braided equivalences and natural isomorphisms compatible with these equivalences. 

\begin{lem}\label{lem:BMF1CsVec}
Let $\mathrm{BMF1C}^{\mathrm{slightly-deg}}$ denote the full $2$-groupoid of $\ob(\mathbf{BMF1C})$ on the slightly degenerate braided fusion categories, i.e.\ those whose symmetric center happens to be equivalent to $\sVect$ without a specified such equivalence. Then there is a fiber sequence
\[
\mathrm{BMF1C}^{\mathrm{ndeg}}(\sVect) \to \mathrm{BMF1C}^{\mathrm{slightly-deg}} \stackrel{\cZ_{2}}{\to} \mathrm{SMF1C}^{\sVect}
\]
where $\mathrm{SMF1C}^{\sVect}$ is the full sub-$2$-groupoid of $\ob(\mathbf{SMF1C})$ on those symmetric multifusion $1$-categories which happen to be equivalent to $\sVect$. 
Because
$\mathrm{SMF1C}^{\sVect} \simeq \mathrm{B} \Aut^{\br}(\sVect) \simeq \mathrm{B}^2 \Z/2$, this induces an $ \mathrm{B}\Z/2 \simeq \Aut^{\br}(\sVect)$-action on $\mathrm{BMF1C}^{\ndeg}(\sVect)$
\end{lem}
\begin{proof}
    By definition there is a fiber sequence \[\mathrm{BMF1C}^{\mathrm{ndeg}}(\sVect) \to \ob(\mathbf{BMF1C}) \to \ob(\mathbf{SMF1C}).\] Restricting the middle term to the full subspace $\mathrm{BMF1C}^{\mathrm{slightly-deg}}$ on the image of the fiber inclusion and the last term to the full subspace $\mathrm{B}\Aut^{\br}(\sVect)$ of $\ob(\mathbf{SMF1C})$ on those symmetric multifusion categories which happen to be equivalent to $\sVect$ results in the claim. 
\end{proof}
Unpacked, the induced  $\mathrm{B}\Z/2 = \Aut^{\br}(\sVect)$-action is given by precomposing the specified equivalence between $\sVect$ and the symmetric center. 

In particular, Lemma~\ref{lem:BMF1CsVec} identifies the homotopy quotient $\mathrm{BMF1C}^{\mathrm{ndeg}}(\sVect)\sslash \Aut^{\br}(\sVect)$ with $\mathrm{BMF1C}^{\mathrm{slightly-deg}}$.

\begin{lem}\label{lem:sdegfixedpoint}
The $2$-groupoid of homotopy fixed points of $\mathrm{BMF1C}^{\ndeg}(\sVect)$ under its action by $\Aut^{\br}(\sVect)$ is equivalent to the full sub-$2$-groupoid $\mathrm{BMF1C}^{\ndeg}$ of the $2$-groupoid $\ob(\BMF1C)$ on the nondegenerate braided multifusion $1$-categories. 
\end{lem}
\begin{proof}
By Lemma~\ref{lem:BMF1CsVec}, and from the perspective of spaces over $\mathrm{B}^2\Z/2$, the $2$-groupoid of homotopy fixed points is the 2-groupoid of lifts 
\[
\begin{tikzcd}
& \mathrm{BMF1C}^{\mathrm{slightly-deg}} \arrow[d, "\cZ_{2}"] \\
\mathrm{B}^2 \Z/2  \arrow[ur, dashed]\arrow[r, "\simeq"'] & \mathrm{SMF1C}^{\sVect}
\end{tikzcd}.
\]
Here, recall that $\mathrm{BMF1C}^{\mathrm{slightly-deg}}$ denotes the full sub-$2$-groupoid of braided monoidal multifusion $1$-categories on those whose symmetric center happens to be equivalent to $\sVect$ and where $\mathrm{SMF1C}^{\sVect}$ denotes the full sub-$2$-groupoid of the $2$-groupoid of symmetric multifusion $1$-categories on those that happen to be equivalent to $\sVect$. The horizontal equivalence is the functor selecting $\sVect$ amongst categories which happen to be equivalent to $\sVect$ and remembers its $\mathrm{B}\Z/2$-action. 

Explicitly, an object of this $2$-groupoid is a choice of slightly degenerate braided multifusion $1$-category $\cB$ together with a group homomorphism $\phi: \Z/2 \to \Omega \Aut^{\br}(\cB)$, i.e.\ an order $2$ monoidal natural automorphism of the identity of $\cB$,  (together encoding the dashed diagonal map) and a choice of braided equivalence $\cZ_{2}(\cB) \simeq \sVect$ such that if one restricts $\phi|_{\cZ_{2}(\cB)}$ and transports it along this equivalence to $\sVect$, it becomes $(-1)^f$ (this encodes the homotopy filling the triangle). 
This data of $\phi$ is equivalent to a $\Z/2$-grading on $\cB$ whose trivial part $\cB_0$ is the full subcategory on those objects $b$ with $\phi_b = \id_b$ and the data of the equivalence $\cZ_{2}(\cB) \simeq \sVect$ (together with the condition that $\phi$ becomes $(-1)^f$ under this equivalence) ensures that $\cB_0$ is nondegenerate and identifies $\cB \simeq \cB_0 \boxtimes \sVect$. 
The $1$- and $2$-morphisms may be handled similarly. 
\end{proof}

In \S\ref{subsection:Wittfunctor}, we will show that the functor $\cZ_{2}: \ob(\mathbf{BMF1C}) \to \ob(\mathbf{SMF1C})$ factors through $\cZ_{2}: \ob(\Mor_2^{\sss}) \to \ob(\mathbf{SMF1C})$ (indeed, we will even extend this to a functor on a certain subcategory of $\Mor_2^{\sss}$).
\begin{defn}\label{defn:Wittspace}
For a symmetric multifusion $1$-category $\cE$, we define the $4$-groupoid $\Witt(\cE)$ as the fiber of $\cZ_{2}: \ob(\Mor_2^{\sss}) \to \ob(\mathbf{SMF1C})$ at $\cE$. 
\end{defn}
Unpacked, $\Witt(\sVect)$ is the $4$-groupoid of braided fusion $1$-categories $\cB$ equipped with an equivalence $\sVect \to \cZ_{2}(\cB)$ and (higher) Witt equivalences compatible with these identifications of symmetric centers.

By definition,  this $4$-groupoid $\Witt(\sVect)$ inherits a canonical action by $\Aut^{\br}(\mathbf{sVect})$, and is therefore a superspace. Its homotopy groups are given by $$\begin{tabular}{|c|c|c|c|c|}
\hline
$\pi_0$ & $\pi_1$ & $\pi_2$ & $\pi_3$ & $\pi_4$ \\
\hline \\[-1em]
 $\pi_0(\Witt(\sVect))$ & $0$ & $\mathbb{Z}/2$ & $\mathbb{Z}/2$ & $\mathbb{C}^\times$ \\
\hline
\end{tabular}$$ where $\pi_0(\Witt(\sVect))$ is the super-Witt group introduced in \cite{DNO2013} (Though the group structure does not manifest directly from our construction.).

The superspace $\Witt(\sVect)$ is intimately related to the notion of supercohomology, one of the remaining ``superingredients" featured in the statement of Theorem \ref{thmalpha:EmergentFermions}. We define a superspace $\Sigma^4\SH:=\Witt(\sVect) \langle 1 \rangle$ as the connective cover of the superspace $\Witt(\sVect)$, so that $\Sigma^4\SH$ inherits the homotopy groups of $\Witt(\sVect)$ in degree 1 and above. Said differently, let $[\mathbf{sVect}]$ denote the class of $\mathbf{sVect}$ in the space $\Witt(\sVect)$. We have $\Sigma^3\mathrm{SH}=\mathcal{A}ut_{\mathbf{sVect}}([\mathbf{sVect}])$. This shows that supercohomology may also be defined as the double delooping of the Picard spectrum $\mathbf{Mod}(\sVect)^\times$, and this identification is compatible with the canonical actions by $\Aut^{\br}(\sVect)$. The superspace $\Sigma^4\SH$ was described explicitly as a space over $\mathrm{B}^2\mathbb{Z}/2$ in \cite{decoppet:extension}. We wish to point out that, in the physics literature, this also goes by the name extended supercohomology, which was introduced in \cite{Wang:2017moj}.

\begin{defn}
Let $X$ be a superspace. We use $\SH^n(X)$ to denote the group of homotopy classes of maps of superspaces (i.e.\ $\mathrm{B}\Z/2$-equivariant maps) from $X$ to $\Sigma^n\SH$.
\end{defn}

We note that supercohomology is shifted so that $\SH^0(\mathsf{pt})=\mathbb{C}^{\times}$.

It remains to define the anomaly of an action of a supergroup on a slightly degenerate braided fusion 1-category. Let $\cA$ be a slightly degenerate braided fusion 1-category, $(H,z)$ a finite super group with $z\neq 1$, and $\rho:(H,z) \rightarrow (\Aut^{\br}_{\sVect}(\cA),(-1)^f)$ an action, that is, a map of spaces $\mathrm{B}H\rightarrow \mathrm{B}\Aut^{\br}_{\sVect}(\cA)$ compatible with the actions by $\mathrm{B}\mathbb{Z}/2$. The canonical map $[-]:\ob(\mathbf{BMF1C}^{\ndeg}(\sVect))\rightarrow \mathcal{W}itt(\mathbf{sVect})$ sending a slightly degenerate braided fusion 1-category to the corresponding $\mathbf{sVect}$-Witt equivalence class is a map of superspaces. In particular, there is a corresponding map of superspaces $[-]:\mathrm{B}\Aut_{\mathbf{sVect}}^{\br}(\mathcal{A})\rightarrow \mathrm{B}\Aut_{\mathbf{sVect}}([\mathcal{A}])$.

\begin{defn}
The \textit{superanomaly} of the action $\rho:(H,z) \rightarrow (\Aut^{\br}_{\sVect}(\cA),(-1)^f)$ is the composite map of superspaces \begin{equation}\label{eq:superanomaly}\mathrm{B}(H,z)\xrightarrow{\rho}\mathrm{B}\Aut_{\mathbf{sVect}}^{\br}(\mathcal{A})\xrightarrow{[-]} \mathrm{B}\Aut_{\mathbf{sVect}}([\mathcal{A}]).\end{equation}
\end{defn}

Upon choosing a minimal nondegenerate extension for $\mathcal{A}$, which exists by \cite{JFR}, there is an equivalence of superspaces $\mathrm{B}\Aut_{\mathbf{sVect}}([\mathcal{A}])\simeq \mathrm{B}\Aut_{\mathbf{sVect}}([\mathbf{sVect}]) = \Sigma^4\SH$. Indeed, at the level of spaces, there is a canonical equivalence $\mathrm{B}\Aut_{\mathbf{sVect}}([\mathcal{A}])\simeq \mathrm{B}\Aut_{\mathbf{sVect}}([\mathbf{sVect}])$ since $[ \cA]$ is an invertible object in $\Witt(\sVect)$, but upgrading this to a $\mathrm{B}\mathbb{Z}/2$-equivariant equivalence, i.e.\ an equivalence of superspaces, requires a choice of minimal nondegenerate extension. 
Having made such a choice, the superanomaly $[\rho]$ yields a class in $\SH^{4}(H,z)$. Without making such a choice, the superanomaly lives in $\pi_0$ of the space of $\mathrm{B}\mathbb{Z}/2$-equivariant maps $\pi_0 \Map(\mathrm{B}H, \mathrm{B}\Aut_{\sVect}([\cA]))^{\mathrm{B} \mathbb{Z}/2}$, which is a torsor over $\SH^4(H,z)$.

\begin{rem}\label{rem:ENOanomaly}
In the bosonic case, that is, when $\mathcal{A}$ is nondegenerate braided fusion 1-category with a group action $\rho:H\rightarrow\mathcal{A}ut^{\br}(\mathcal{A})$, the anomaly of $\rho$ is the composite $$\mathrm{B}H\xrightarrow{\rho}\mathrm{B}\Aut^{\br}(\mathcal{A})\xrightarrow{[-]} \mathrm{B}\Aut([\mathcal{A}])\simeq \mathrm{B}^4\mathbb{C}^{\times}.$$ In particular, $[\rho]$ is a class in $H^4(H;\mathbb{C}^{\times})$. This cohomology class was denoted by $O_4(\rho)$ in \cite{ENO2010}, where it was explained that there is a fiber sequence of ordinary spaces $$\mathrm{B}\mathcal{P}ic(\cA)\rightarrow \mathrm{B}\Aut^{\br}(\mathcal{A})\xrightarrow{[-]} \mathrm{B}\Aut([\mathcal{A}]),$$ where $\mathcal{P}ic(\mathcal{A})$ denotes the space of invertible $\mathcal{A}$-module 1-categories. In particular, the anomaly $[\rho]$ measures exactly the failure of the map $\rho:\mathrm{B}H\rightarrow \mathrm{B}\Aut^{\br}(\mathcal{A})$ to lift to a map $\mathrm{B}H\rightarrow \mathrm{B}\mathcal{P}ic(\cA)$, and thereby yield an $H$-crossed braided extension of $\cA$.
\end{rem}

\section{Witt squares}\label{section:Wittsquares}

In this section, we establish $\Witt(-)$ as a functor (i.e.\ prove Proposition~\ref{thm:Wittfunctor}) and prove Theorem~\ref{thm:Wittsquare}.

\subsection{From Multifusion 2-Categories to Morita Categories and Back}\label{subsection:extractWittdata}
Recall that $\mathbf{MF2C}$ denotes the $(3,1)$-category of multifusion 2-categories and monoidal 2-functors, and $\Mor^{\sss}$ denotes the $(4,1)$-category of multifusion $2$-categories and semisimple bimodule $2$-categories.  Let us write $\Mor^{\sss,0\mhyphen\dom}$ for the (non-full) subcategory of $\Mor^{\sss}$ on the $0$-dominant 1-morphisms (Definition \ref{def:dominanceMF2C}), which exists thanks to Proposition \ref{prop:nDomComposesMor2} and Lemma~\ref{lem:invertible2faith2dom}. The canonical functor $$[-]:\mathbf{MF2C}\to \Mor^{\sss}$$ sending a multifusion 2-category to the corresponding Morita class $[\fC]$ factors through this  subcategory $\Mor^{\sss,0\mhyphen\dom}$: Namely, given a monoidal functor $F:\fC\to\fD$, the associated finite semisimple bimodule 2-category is $\fD$ with the obvious $\fC$-$\fD$-bimodule structure. Such a bimodule 2-category is always $0$-dominant. 

Let $\Mor_{\mathds{1}/}^{\sss,0\mhyphen\dom}:= \left(\Mor^{\sss,0\mhyphen\dom}\right)_{\mathds{1}/}$ denote the slice category of objects with a morphism from~$\mathds{1}$. The following proposition allows us to describe the category of multifusion $2$-categories and monoidal functors in terms of Morita categories: 
\begin{prop}\label{prop:MF2CViaMor} The canonical functor $\mathbf{MF2C} \to\Mor_{\mathds{1}/}^{\sss,0\mhyphen\dom} $
sending a multifusion $2$-category $\fC$ to the right module $\fC_{\fC}$ is an equivalence of $\infty$-categories. 
\end{prop}
\begin{proof}
By~\cite[Theorem~4.8.5.11]{higheralgebra},  the functor $\Alg_{E_1}(\tKar) \to (\Mod_{\tKar})_{\tKar/}$ sending a monoidal $2$-Karoubian category $\fC$ to the category $\Mod_{\fC}(\tKar)$ of $2$-Karoubian $\fC$-modules is fully faithful. Moreover, by definition (see~\S\ref{subsubsection:MorF2C}), it factors through the full subcategory $\Mor(\tKar)$. By the discussion in \S\ref{subsubsection:MorF2C},  the induced fully faithful functor $\Alg_{E_1}(\tKar) \to \Mor(\tKar)_{\mathds{1}/}$ sends a $\fC$ to the module $\fC_{\fC}$ and a monoidal functor $F:\fC \to \fD$ to the bimodule $_{\fC} \fD_{\fD}$ together with the bimodule equivalence $\fD\simeq \fC \boxtimes_{\fC} \fD_{\fD} $. 
It follows from this description, that the fully faithful composite $\mathbf{MF2C} \hookrightarrow \Alg_{E_1}(\tKar)\hookrightarrow \Mor(\tKar)_{\mathds{1}/}$ factors through the (a priori non-full) subcategory $\Mor_{\mathds{1}/}^{\sss,0\mhyphen\dom} \to \Mor(\tKar)_{\mathds{1}/}$ and hence results in a fully faithful functor $\mathbf{MF2C} \to \Mor_{\mathds{1}/}^{\sss,0\mhyphen\dom}$. Finally, we prove that this functor is essentially surjective on objects: Given any $0$-dominant $1$-morphism $\fM_{\fD}$ in $\Mor^{\sss}$ out of $\mathds{1}$, it suffices to construct a multifusion $1$-category $\fC$, a Morita equivalence $\fN: \fC \nrightarrow\fD$ and an equivalence $ \fC \boxtimes_{\fC}\fN_{\fD} \simeq \fM_{\fD}. $ Taking $\fN:= \fM$ and $\fC:= \End_{\fD}(\fM)$ finishes the proof by \cite[Theorem 5.4.3]{Decoppet2022;Morita}.
\end{proof}

Proposition~\ref{prop:MF2CViaMor} also allows us to make contact with braided multifusion 1-categories. More precisely, let $\Mor^{\sss}_2$ denote the (univalent) Morita $4$-category of braided multifusion 1-categories, as defined in \S\ref{subsubsection:Mor2ss}. 

Just as every multifusion 2-category $\fC$ determines, and is determined by (up to monoidal equivalence), a 1-morphism $\fC: \mathds{1} \nrightarrow \fC$ in $\Mor^{\sss}$, so too every braided multifusion 1-category $\cB$ determines, and is determined by, a 1-morphism $\cB_\cB : \mathds{1} \nrightarrow \cB$ in $\Mor^{\sss}_2$.

\begin{prop}
    \label{lem:obBMF1C}
    The functor $\cB \mapsto \cB_\cB$ defines an equivalence of $\infty$-categories $$\BMF1C\to \Mor^{\sss,1\mhyphen\dom}_{2,\mathds{1}/}.$$
\end{prop}
\begin{proof}
Restrict the equivalence $\mathbf{MF2C} \to \Mor_{\mathds{1}/}^{\sss,0\mhyphen\dom}$ along the fully faithful functor $\mathbf{BMF1C}\hookrightarrow \mathbf{MF2C}$ and use the equivalence $\Mor_2^{\sss} \simeq \Mor^{\sss}$ to obtain a fully faithful functor $\mathbf{BMF1C} \hookrightarrow \Mor_{2,\mathds{1}/}^{\sss,0\mhyphen\dom}$ which sends a braided multifusion $1$-category $\cB$ to the $\cB$-central multifusion category~$\cB$, thought of as a 1-morphism $\mathds{1} \nrightarrow \cB$. In particular, this fully faithful functor factors through the (a priori non-full) subcategory  $\Mor_{2,\mathds{1}/}^{\sss,1\mhyphen\dom}$ (since for any $\cB$, the functor $\cB \to \cZ(\cB)$ is fully faithful). It then suffices to prove that the fully faithful functor $\mathbf{BMF1C} \to \Mor_{2,\mathds{1}/}^{\sss,1\mhyphen\dom}$ is essentially surjective on objects. Equivalently, for any braided multifusion $1$-category $\cB$ and any multifusion $1$-category $\cC$ with a fully faithful braided functor $\cB^{\rev} \to \cZ(\cC), $
there exists another braided multifusion category $\cB'$ and a Witt equivalence $\cD:\cB' \nrightarrow \cB$ so that $ \cB' \boxtimes_{\cB'}\cD \simeq \cC$. Take $\cB':= \cZ(\cC, \cB^{\rev})$ and as $1$-morphism $\cB' \nrightarrow \cB$  the multifusion category $\cD:=\cC$ with evident $\cB'^{\rev} \boxtimes \cB$ central structure, it follows from fully faithfulness of $\cB^{\rev} \to \cZ(\cC)$ that $\cD$ is indeed an invertible $1$-morphism in $\Mor_2^{\sss}$ by Corollary \ref{rem:davidStatement1}, finishing the proof. 
\end{proof}

Recall from \S\ref{subsubsection:Mor2ss} that $\Mod: \Mor_2^{\sss} \to \Mor^{\sss}$ defines an equivalence (after univalification); the most interesting part of this  is that every multifusion $2$-category is Morita equivalent to $\Mod(\cB)$ for a braided multifusion $1$-category, as shown in~\cite{Decoppet2022;Morita}. 

Thus, the previous statements may be summarized in the following diagram:
\begin{equation}\label{eq:BMF1CtoMor2}
    \begin{tikzcd}
        \BMF1C \arrow[r,hook] \arrow[d,"\rotatebox{90}{$\sim$}"'] & \MF2C \arrow[d,"\rotatebox{90}{$\sim$}"] \\
        \Mor_{2,\mathds{1}/}^{\sss,1\mhyphen\dom} \stackrel{\Mod}{\simeq}  \Mor_{\mathds{1}/}^{\sss,1\mhyphen\dom} \arrow[r,hook]& \Mor_{\mathds{1}/}^{\sss,0\mhyphen\dom}
    \end{tikzcd}
\end{equation}
Notice that it follows that the bottom horizontal arrow is also fully faithful, which is a priori not obvious.

\subsection{\texorpdfstring{Constructing $\Witt$ as a Functor}{Constructing Witt as a Functor}}\label{subsection:Wittfunctor}

In this section, we will construct the functor $\Witt(-)$ taking a symmetric multifusion category to its associated Witt space.

\begin{lem}\label{lem:firstleftfibrations}
    The following functors are left fibrations: $$[-] : \mathbf{MF2C}\to \Mor^{\sss,0\mhyphen\dom}$$
    $$[-]: \mathbf{BMF1C} \to \Mor_2^{\sss,1\mhyphen\dom}$$
\end{lem}
\begin{proof}
    By Propositions~\ref{prop:MF2CViaMor} and~\ref{lem:obBMF1C}, it suffices to show that the forgetful functors $\left(\Mor^{\sss,0\mhyphen\dom}\right)_{\mathds{1}/} \to \Mor^{\sss,0\mhyphen\dom}$ and $\left(\Mor^{\sss,1\mhyphen\dom}\right)_{\mathds{1}/} \to \Mor^{\sss,1\mhyphen\dom}$ are left fibrations, which follows from~\cite[Corollary 2.1.2.2]{highertopos}.
\end{proof}

We will now construct $\cZ_{2}$ as a functor from a certain non-full subcategory of the Morita $(4,1)$-category $\Mor^{\sss}_2$ of braided multifusion 1-categories to the $(2,1)$-category $\mathbf{SMF1C}$ of symmetric multifusion 1-categories. Restricting further to a smaller subcategory, the functor $\mathcal{Z}_{2}$ will be a left fibration onto the category of dominant and faithful symmetric functors. It therefore follows from unstraightening that the assignment sending a symmetric multifusion 1-category $\mathcal{E}$ to the corresponding Witt space $\Witt(\cE)$ (i.e.\ the fiber of $\cZ_{2}$ at $\cE$) is functorial with respect to dominant and faithful symmetric tensor functors.

\begin{prop}\label{prop:Mugerfunctor}
There is a functor of $\infty$-categories 
$\mathcal{Z}_{2}:\mathbf{Mor}_{2}^{\sss,\str1\mhyphen\dom}\to \mathbf{SMF1C}$ from the sub-$(3,1)$-category of the 
Morita (4,1)-category of braided multifusion 1-categories with the same objects but only with the strongly 1-dominant 1-morphisms (Definition~\ref{defn:stronglydominant}), to the $(2,1)$-category of symmetric multifusion 1-categories with symmetric tensor functors.
\end{prop}

\begin{proof}
Since the target is a $(2,1)$-category it suffices to construct this functor on the $2$-truncation of $\mathbf{Mor}_{2}^{\sss,\str1\mhyphen\dom}$ which we will do here ``by hand'' at the level of objects, $1$-morphisms, and equivalence classes of $2$-morphisms. 
On objects, $\cZ_{2}$ is defined in the obvious way, i.e., if $\cB_1$ is a braided multifusion 1-category, we let $\mathcal{Z}_{2}(\cB_1)$ be the symmetric center of $\cB_1$. Now, recall from \eqref{eq:Muegerfunctor1morphisms}, that to any strongly 1-dominant 1-morphisms $\cC_1: \cB_1 \nrightarrow \cB_2$ in $\mathbf{Mor}_{2}^{\sss}$, we can associate a symmetric tensor functor $\mathcal{Z}_{2}(\mathcal{C}_1):\cZ_{2}(\cB_1) \to \cZ_{2}(\cB_2)$, and from Lemma~\ref{lem:symmetriccompose} that these compose correctly. 
Finally, it is straightforward to define the functor $\cZ_{2}$ on isomorphism classes of invertible 2-morphisms and check functoriality.
\end{proof}

\begin{lem}\label{prop:Mugerfunctoradjectives}
The functor 
$\mathcal{Z}_{2}:\mathbf{Mor}_{2}^{\sss,\sss,2\mhyphen\dom,0\mhyphen\faith}\to \mathbf{SMF1C}^{\dom,\faith}$ sends 2-dominant and 0-faithful 1-morphisms to dominant faithful symmetric tensor functors.
\end{lem}

\begin{proof}
Firstly, recall from section \S\ref{subsection:adjectives} that every 2-dominant 1-morphism in $\Mor_2^{\sss}$ is in particular strongly 1-dominant. Now, let $\cC: \cB_1 \nrightarrow \cB_2$ be a 2-dominant 1-morphism. The fact that $\cC$ is 2-dominant implies that it is factorizable by Lemma~\ref{lem:fact+ext=2dom}, which means that $\Im(\cE_1) = \cE_2$, so the symmetric tensor functor $\mathcal{Z}_{(2)}(\cC):\cE_1 \to \cE_2$ is dominant as desired. Secondly, if $\cC: \cB_1 \nrightarrow \cB_2$ is $0$-faithful and strongly 1-dominant, then the braided tensor functor $F_1:\cB_1\rightarrow \mathcal{Z}(\cC)$ is faithful. As $\Im(\cE_1)\subseteq\cE_2$ in $\mathcal{Z}(\cC)$ by strong 1-dominance, it follows that the symmetric tensor functor $\mathcal{Z}_{(2)}(\cC):\cE_1 \to \cE_2$ is faithful.
\end{proof}

Using Lemma~\ref{prop:Mugerfunctoradjectives}, we will henceforth only consider the restriction
\[\cZ_{2}: \Mor_2^{\sss,2\mhyphen\dom,0\mhyphen\faith}  \to \mathbf{SMF1C}^{\dom, \faith}
\]
of the functor $\cZ_{2}$ to the sub-$(3,1)$-category $\Mor_2^{\sss,2\mhyphen\dom,0\mhyphen\faith}$ with the same objects but only with the 2-dominant and 0-faithful 1-morphisms.  

The functoriality of the Witt space follows directly from the following main theorem of this section:
\begin{thm}\label{thm:Mor2SMFCleftfibration}
The functor $\cZ_{2} : \Mor_2^{\sss,2\mhyphen\dom,0\mhyphen\faith} \to \mathbf{SMF1C}^{\dom,\faith}$ is a left fibration.
\end{thm}
We will establish the fibrancy asserted in Theorem~\ref{thm:Mor2SMFCleftfibration} by piecing together a series of fibrancy conditions, which we state in the following lemmas.
Let $\mathbf{BMF1C}^{\dom, \faith} \to \mathbf{BMF1C}$ denote the non-full subcategory on the faithful and dominant braided monoidal functors. 

\begin{lem}\label{cor:B1CtoSym}
    The induced map $[-] : \mathbf{BMF1C}^{\dom,\faith} \to \Mor_2^{\sss,2\mhyphen\dom,0\mhyphen\faith}$
    is a left fibration.
\end{lem}
\begin{proof}
We claim that the square of $\infty$-categories
    \begin{equation*}
    \begin{tikzcd}
    \mathbf{BMF1C}^{\dom,\faith} \arrow[d] \arrow[r] & \mathbf{BMF1C} \arrow[d]\\
    \Mor_2^{\sss,2\mhyphen\dom,0\mhyphen\faith} \arrow[r] & \Mor_2^{\sss,1\mhyphen\dom}
\end{tikzcd}\end{equation*}
is a pullback. Equivalently, it suffices to show that a braided functor $F:\cB_1 \to \cB_2$ is faithful and dominant if and only if the associated $1$-morphism $\cB_1:\cB_1\nrightarrow\cB_2$ is $2$-dominant and $0$-faithful. The equivalence of the faithfulness conditions is obvious. The dominance condition amounts to the statement that $F: \cB_1 \to \cB_2$ is dominant if and only if the induced functor $\cB_2 \to \cZ(\cB_2, \cB_1)$ is an equivalence. Note that one can reduce to only having to check this statement when $\cB_1$ and $\cB_2$ are fusion.  It follows from \cite[Theorem 3.14]{DGNO2010braided} and \cite[Proposition 6.3.4]{EGNO} that the composite is an equivalence if and only if $\Im(F)=\cB_2$, that is, if and only if $F$ is dominant.

Since pullbacks of left fibrations are left fibrations, the result follows from Lemma~\ref{lem:firstleftfibrations}.
\end{proof}

The following proposition is a homotopical rephrasing of \cite[Corollary 3.24]{DMNO}, and is at the heart of our construction of $\Witt$:

\begin{prop}\label{lemma:BMF1Ccomposite}
 The composite
\[ \mathbf{BMF1C}^{\dom, \faith} \overset{[-]}\longrightarrow \Mor_2^{\sss,2\mhyphen\dom,0\mhyphen\faith} \overset{ \cZ_{2}}\longrightarrow \mathbf{SMF1C}^{\dom,\faith} \]
is a left fibration.
\end{prop}
\begin{proof}
This composite sends a braided multifusion 1-category $\cB$ to its M\"uger centre $\cZ_{2}(\cB)$ and a braided functor $F:\cB_1 \to \cB_2$ to the induced braided functor $\cZ_{2}(\cB_1) \to \cB_1 \to \cB_2$ which lands in the full subcategory $\cZ_{2}(\cB_2) \subseteq \cB_2$ since $F$ is dominant and hence defines a symmetric functor $\cZ_{2}(F): \cZ_{2}(\cB_1) \to \cZ_{2}(\cB_2)$ (which is itself dominant and faithful).

The left fibrancy unpacks to the following statement:\
Fix a braided multifusion $1$-category $\cB$.
Then the $2$-functor $\cZ_{2}(-)$ from the $2$-groupoid $\ob(\mathbf{BMF1C}^{\dom, \faith}_{\cB/})$ of braided multifusion $1$-categories $\cB'$ equipped with a dominant faithful braided functor  $\cB \to \cB'$ to the $2$-groupoid $\ob(\mathbf{SMF1C}^{\dom, \faith}_{\cZ_{2}(\cB)/})$ of symmetric multifusion $1$-categories $\cE$ with a symmetric functor $\cZ_{2}(\cB) \to \cE$ is an equivalence. 
It is enough to show this when $\cB$ is fusion in which case an explicit inverse is given by the functor sending an object $\{\cZ_{2}(\cB) \to \cE\}$ of the $2$-groupoid $\ob(\mathbf{SMF1C}^{\dom, \faith}_{\cZ_{2}(\cB)/})$  to $\{\cB \to \cB \otimes_{\cZ_{2}(\cB)} \cE\} \in \ob(\mathbf{BMF1C}^{\dom, \faith}_{\cB/})$.  That this defines an inverse follows from \cite[Corollary 3.24]{DMNO}.
\end{proof}

\begin{proof}[Proof of Theorem \ref{thm:Mor2SMFCleftfibration}]
Left fibrations satisfy a  $2$-out-of-3-property~\cite[Proposition 2.4.1.3 (3)]{highertopos}:\ For composable functors $F,G$, if $F$ and $G\circ F$ are left fibrations and $F$ is essentially surjective on objects, then $G$ is a left fibration. Applying this statement to the commutative triangle
$$\begin{tikzcd}
\mathbf{BMF1C}^{\dom,\faith} \arrow[rr] \arrow[rd,"\protect{[-]}"'] &            & \mathbf{SMF1C}^{\dom,\faith} \\
 & \Mor_2^{\sss,2\mhyphen\dom,0\mhyphen\faith} \arrow[ru, "\mathcal{Z}_{2}"'] &  \end{tikzcd}$$
\noindent
gives the desired result.
\end{proof}

Straightening the left fibration $\Mor^{\sss,2\mhyphen\dom,0\mhyphen\faith}_2 \rightarrow \mathbf{SMF1C}^{\dom,\faith}$ then immediately gives rise to our desired functor: 

\begin{cor}\label{cor:Witt}
     There is a functor $\Witt(\text{--}):\mathbf{SMF1C}^{\dom,\faith} \rightarrow \mathbf{Spaces}$ which sends a symmetric multifusion $1$-category $\cE$ to the space $\Witt(\cE)$ from Definition~\ref{defn:Wittspace}.
\end{cor}

\subsection{The Proof of Theorem~\ref{thm:Wittsquare}}\label{sec:proof}

In this section, we prove Theorem~\ref{thm:Wittsquare}. 

The idea  behind the first step of the proof is the following: 
We expect, but do not prove, that the $1$-dominant and $2$-faithful $1$-morphisms from Definitions~\ref{def:dominanceMF2C} and~\ref{defn:faithfulMF2C} form a \emph{factorization system} on $ \Mor_2^{\sss} \simeq \Mor^{\sss}$ (see e.g.~\cite[Definition 5.2.8.8]{highertopos} for a definition of factorization system in the world of $\infty$-categories). In particular, this means that it should be possible to factor any $1$-morphism in $\Mor^{\sss}$ \emph{(contractibly) uniquely} into a $1$-dominant one followed by a $2$-faithful one. 

On the other hand, by Proposition~\ref{prop:MF2CViaMor}, the data of a multifusion $2$-category is equivalent to the data of a $0$-dominant $1$-morphism out of the unit in $\Mor^{\sss}$. Thus, we may try to factor this into a $1$-dominant morphism out of the unit; corresponding by Proposition~\ref{lem:obBMF1C} to the data of a braided multifusion $1$-category, followed by a $2$-faithful (and still $0$-dominant) $1$-morphism. 
Explicitly, given a multifusion $2$-category $\fC$, this braided multifusion $1$-category will be $\Omega \fC$ and the $2$-faithful $0$-dominant $1$-morphism will be the one induced by the inclusion $\Mod(\Omega \fC) \to \fC$. 
Although we do not construct the full factorization system we now prove (contractible) uniqueness of this particular factorization:

\begin{prop}\label{prop:MF2Cpullback}
There is a pullback square of spaces: $$\begin{tikzcd}[sep=small]
\ob(\mathbf{MF2C}) \arrow[dd, "\fC \mapsto {[}\fC: ~\Mod( \Omega \fC) \nrightarrow \fC{]}"'] \arrow[rr, "\Omega"]                  &  & \ob(\mathbf{BMF1C}) \arrow[dd, "\cB \mapsto {[}\Mod(\cB){]}"]     \\
 &  & \\
{\Ar(\Mor^{\sss,0\mhyphen\dom,2\mhyphen\faith})} \arrow[rr, "s"'] &  & {\ob(\Mor^{\sss}).}
\end{tikzcd}$$
\end{prop}

In words, Proposition~\ref{prop:MF2Cpullback} asserts that the data of a multifusion $2$-category $\fC$ (up to monoidal equivalence) is completely captured by the data of the braided multifusion $1$-category $\Omega \fC$ (up to braided monoidal equivalence) together with the data of the $1$-morphism $\{\fC: \Mod(\Omega \fC) \nrightarrow \fC\}$ in $\Mor^{\sss}$ (up to equivalence in $\Mor^{\sss}$). 
\begin{proof}
    We claim that the following is a pullback square of $\infty$-categories:
\begin{equation}\begin{tikzcd}\label{eq:1dom0faisquare}
    \mathbf{MF2C}^{\ff} \arrow[d] \arrow[r] & \mathbf{MF2C} \arrow[d]\\
    \Mor^{\sss,0\mhyphen\dom, 2\mhyphen\faith} \arrow[r] & \Mor^{\sss, 0\mhyphen\dom}.
\end{tikzcd}\end{equation}
\noindent Here, $\mathbf{MF2C}^{\ff}$ denotes the sub-$(3,1)$-category of the $(3,1)$-category $\mathbf{MF2C}$ on the multifusion $2$-category and the \emph{fully faithful} monoidal $2$-functors between them, i.e.\ those monoidal $2$-functors which induce equivalences on $\Hom$-1-categories.
The square~\eqref{eq:1dom0faisquare} being a pullback is equivalent to the assertion that a monoidal $2$-functor $F: \fC \to \cD$ between multifusion $2$-categories is fully faithful if and only if its induced bimodule $_{\fC} \fD_{\fC}$ is $2$-faithful in the sense of Definition~\ref{defn:faithfulMF2C}, which is evidently true by definition. 

Now, since the right vertical functor is a left fibration, so is the left vertical one. Unpacking the definition of a left fibration, we obtain the following pullback square of spaces: 
\begin{equation}\label{eq:arrMF2C}
    \begin{tikzcd}[sep=small]
\Ar(\MF2C^{\ff}) \arrow[rr, "s"] \arrow[dd]      &  & \ob(\mathbf{MF2C}^{\ff}) =\ob(\mathbf{MF2C}) \arrow[dd]     \\
&  & \\
{\Ar(\Mor^{\sss,0\mhyphen\dom,2\mhyphen\faith})} \arrow[rr, "s"'] &  & {\ob(\Mor^{\sss,0\mhyphen\dom,2\mhyphen\faith}) = \ob(\Mor_2^{\sss}),}
\end{tikzcd}
\end{equation}
where $\mathrm{ar}$ and $\mathrm{ob}$ denote the spaces of morphisms and objects in an $\infty$-category, respectively. 

Let us now consider the fully faithful functor $\Mod: \mathbf{BMF1C}\hookrightarrow \mathbf{MF2C}$ sending a braided multifusion 1-category to the multifusion 2-category of finite semisimple module 1-categories. The image of $\Mod$ consists precisely of the full subcategory of $\mathbf{MF2C}$ on those objects $\fC$ such that the inclusion of the monoidal unit $\mathbf{1}_\fC : \mathbf{2Vect} \to \fC$ is surjective on connected components, i.e.\ the multifusion 2-category $\fC$ is connected. Hence, the map $\ob(\mathbf{BMF1C})\to \ob(\mathbf{MF2C})$ is the inclusion of a full subspace. Pulling back~\eqref{eq:arrMF2C} along this inclusion, we obtain the pullback square 
\begin{equation}
    \begin{tikzcd}[sep=small]
\Ar^{\mathrm{sc}}(\mathbf{MF2C}^{\ff}) \arrow[rr, "s"] \arrow[dd]                                         &  & \ob(\mathbf{BMF1C}) \arrow[dd] \\
&  & \\
\Ar(\mathbf{Mor}^{0\mhyphen\dom,2\mhyphen\faith}) \arrow[rr, "s"]      &  & {\ob(\Mor^{\sss})\, ,}  
\end{tikzcd}
\end{equation}
where $\Ar^{\mathrm{sc}}(\mathbf{MF2C}^{\ff}) \hookrightarrow \Ar^{\mathrm{sc}}(\mathbf{MF2C}^{\ff})$ denotes the full subspace on those fully faithful monoidal $2$-functors whose source is a connected fusion $2$-category. Since any object of $\Ar^{\mathrm{sc}}(\mathbf{MF2C}^{\ff})$ is necessarily the inclusion of the connected component of the monoidal unit, it follows that taking the target 
$$t:\Ar^{\mathrm{sc}}(\mathbf{MF2C}^{\ff})\rightarrow \ob(\mathbf{MF2C})$$ is an equivalence. 

Under this equivalence, the map $s:\Ar^{\mathrm{sc}}(\mathbf{MF2C}^{\ff})\rightarrow \ob(\mathbf{BMF1C})$ is identified with $\Omega$, the functor sending a multifusion 2-category to its associated braided multifusion 1-category of endomorphisms.
\end{proof}

We have now factored the data of a multifusion $2$-category into the data of a braided multifusion $2$-category together with a $0$-dominant, $2$-faithful $1$-morphism in $\Mor^{\sss} \simeq \Mor_2^{\sss}$. To conclude our proof of Theorem~\ref{thm:Wittsquare}, it therefore only remains to describe this space $\Ar(\Mor_2^{\sss,0\mhyphen\dom,2\mhyphen\faith})$ of $0$-dominant and $2$-faithful $1$-morphisms in $\Mor_2^{\sss}$. 
 In order to do so, it is convenient to use the equivalence \begin{equation}\label{eq:reverse}\Ar(\Mor_2^{\sss,0\mhyphen\dom,2\mhyphen\faith})\cong \Ar(\Mor_2^{\sss,2\mhyphen\dom,0\mhyphen\faith})\end{equation} given by taking adjoints of $1$-morphisms. 

By Theorem~\ref{thm:Mor2SMFCleftfibration}, we understand this $\infty$-category $\Mor_2^{\sss,2\mhyphen\dom,0\mhyphen\faith}$ well as the source of the left fibration $\cZ_{2}:\Mor_2^{\sss,2\mhyphen\dom,0\mhyphen\faith} \to \mathbf{SMF1C}^{\dom, \faith} $. Unwinding the definition of a left fibration, we immediately obtain:

\begin{cor}\label{cor:pullbackWittLeftFib}
The following is a pullback square of spaces
\[
\begin{tikzcd}  \Ar(\Mor_2^{\sss,2\mhyphen\dom,0\mhyphen\faith}) \arrow[r, "s"] \arrow[d, "\cZ_{2}"'] & \ob(\Mor_2^{\sss})\arrow[d, "\cZ_{2}"]\\
    \Ar(\mathbf{SMF1C}^{\dom, \faith}) \arrow[r, "s"]& \ob(\mathbf{SMF1C})\, ,
\end{tikzcd}
\]
where the horizontal maps take the source, and the vertical maps are given by applying the functor $\cZ_{2}$ from Proposition~\ref{prop:Mugerfunctor}.
\end{cor}

Combining the pullback squares from Proposition~\ref{prop:MF2Cpullback} and Corollary~\ref{cor:pullbackWittLeftFib} results in our main result: a full description of the $3$-groupoid $\ob(\mathbf{MF2C})$ of multifusion $2$-categories:

\begin{thm}\label{thm:maintheorem}
The following is a pullback square of spaces:
\begin{equation}\label{eq:doublepullback}
\begin{tikzcd}
\ob(\mathbf{MF2C}) 
\arrow[dd, "\fC \mapsto \left(\vphantom{\frac{a}{b}}\Omega \cZ(\fC)\to \cZ_{2}(\Omega \fC){,}{[}\fC{]}\right)"']
\arrow[rr, "\Omega"] 
&  & \ob(\mathbf{BMF1C}) \arrow[dd, "{[-]}"]     \\
&  &  \\
\Ar(\mathbf{SMF1C}^{\dom,\faith})\ \underset{\ob(\mathbf{SMF1C})}{\times}\ob(\Mor_2^{\sss}) \arrow[rr, "{\left([\cB],  F:\cZ_{2}(\cB) \to \cE \right)} \mapsto  \Witt(F)({[\cB]})\in \Witt(\cE)"'{yshift=-10pt} ] &  & {\ob(\Mor_2^{\sss})}
\end{tikzcd}
\end{equation}
Here, the pullback at the bottom left is taken over \[\Ar(\mathbf{SMF1C}^{\dom,\faith}) \stackrel{s}{\to} \ob(\mathbf{SMF1C}) \stackrel{\cZ_{2}}{\leftarrow} \ob(\mathbf{Mor}_2^{\sss}),\]
the left vertical map implicitly uses the equivalence $\Mor^{\sss} \simeq \Mor_2^{\sss}$, 
and the bottom horizontal map uses functoriality of $\Witt(-)$.
\end{thm}
\begin{proof}
    This follows directly from combining the pullback from Proposition~\ref{prop:MF2Cpullback} with the equivalence $\Mor_2^{\sss} \simeq \Mor^{\sss}$, the equivalence~\eqref{eq:reverse} induced from taking adjoints and the pullback from Corollary~\ref{cor:pullbackWittLeftFib}.
\end{proof}

Taking fibers at a fixed symmetric monoidal functor $\cE \to \cF$ then immediately results in Theorem~\ref{thm:Wittsquare}.

\begin{proof}[Proof of Theorem~\ref{thm:Wittsquare}]
Consider the commuting diagram
\[
\begin{tikzcd}
    \Ar(\mathbf{SMF1C}^{\dom,\faith})\ \underset{\ob(\mathbf{SMF1C})}{\times}\ob(\Mor_2^{\sss}) \arrow[r ] \arrow[d] &  {\ob(\Mor_2^{\sss})}\arrow[d, "\cZ_{2}"]\\
    \Ar(\mathbf{SMF1C}^{\dom, \faith}) \arrow[r, "t"'] &\ob(\mathbf{SMF1C})
\end{tikzcd}
\]
where the top horizontal map is the bottom horizontal map from Theorem~\ref{thm:maintheorem} and the left vertical map is the projection. (Note that this is \emph{not} a pullback diagram and in particular \emph{not} the defining pullback of the top left entry since neither the top horizontal map is the projection, nor the bottom horizontal map is the ``source'' map.)
Fibering the pullback square~\ref{eq:doublepullback}  at a fixed $(F:\cE \to \cF) \in \Ar(\mathbf{SMF1C}^{\dom,\faith})$ results in a pullback square
\[
\begin{tikzcd}[sep=small]
\ob(\mathbf{MF2C})(\cE\twoheadrightarrow\cF) \arrow[d, "{[-]}"'] \arrow[r, "\Omega"] &  \ob(\mathbf{BMF1C}^{\ndeg}(\cF)) \arrow[d, "{[-]}"]     \\
 {\mathcal{W}itt(\cE)} \arrow[r] &  {\mathcal{W}itt(\cF).}
\end{tikzcd}
\]
Here, $\ob(\mathbf{MF2C})(\cE\twoheadrightarrow\cF)$  denotes the fiber of 
\[
\mathrm{ob}(\mathbf{MF2C}) \stackrel{\fC\mapsto( \Omega \cZ(\fC) \to \cZ_{2}(\Omega \fC))}{\longrightarrow} \Ar(\mathbf{SMF1C}^{\dom,\faith}),
\]
i.e.\ $\ob(\mathbf{MF2C})(\cE\twoheadrightarrow\cF)$ can be thought of as the $3$-groupoid of multifusion $2$-categories equipped with an identification $(\Omega \cZ(\fC) \to \cZ_{2}(\Omega \fC)) \simeq (\cE \to \cF)$.
Thus $\ob(\mathbf{MF2C}(\cE\twoheadrightarrow\cF))$ can be identified with the space of Witt squares of type $\cE \twoheadrightarrow \cF$ as introduced in Equation \eqref{eq:DmitriData}. This completes the proof of Theorem~\ref{thm:Wittsquare}.
\end{proof}

\subsection{Unpacking Witt Squares and Witt Data}\label{subsection:Wittdata}
Having completed the proof of Theorem \ref{thm:Wittsquare} in the previous section, we will in this section work incoherently to give a more hands on unpacking of the ingredient of the Witt squares.
We fix a faithful dominant symmetric tensor functor $F:\cE\rightarrow\cF$ between symmetric multifusion 1-categories. Recall from \eqref{eq:DmitriData} that a Witt square of type $F:\cE\twoheadrightarrow\cF$ is a commuting square of spaces of the form
$$\begin{tikzcd}
\pt
\arrow[r]
\arrow[d]
&
\ob(\mathbf{BMF1C}^{\ndeg}(\cF))
\arrow[d,"\protect{[-]}"]
\\
\mathcal{W}itt(\cE)
\arrow[ur,Rightarrow,shorten <= 1em, shorten >= 1em,"\simeq"]
\arrow[r]
&
\mathcal{W}itt(\cF)\,.
\end{tikzcd}$$
\noindent We begin by unpacking the content of such squares in the form of \textit{Witt data}. Subsequently, inspired by \cite{Decoppet2022;centers}, we give a down-to-earth account of the relation between multifusion 2-category data and Witt data.

Firstly, by construction, the space $\mathcal{W}itt(\cE)$ has objects $\cE$-nondegenerate braided multifusion 1-categories given in Definition \ref{def:nondegOver}.
Furthermore, given $\cB_1,\, \cB_2$ two $\cE$-nondegenerate braided multifusion 1-categories, a 1-morphism between them in $\mathcal{W}itt(\cE)$ is an $\cE$-Witt equivalence given in Definition \ref{def:WittEqOver}.
Said differently, an $\cE$-Witt equivalence is a 1-equivalence $\cC:\cB_1\nrightarrow\cB_2$ in $\mathbf{Mor}_2^{\sss}$ together with an identification of the two canonical braided tensor functors $\cE\rightarrow\mathcal{Z}(\cC)$. This notion was first considered in \cite{DNO2013}, and the set of connected components of $\mathcal{W}itt(\cE)$ is precisely the $\cE$-Witt group considered in this last reference. Then the content of the Witt squares of type $F:\cE\twoheadrightarrow\cF$ unpacks as in the next definition.

\begin{defn}
\label{WDdef}
A {\em Witt datum} of type $\cE \twoheadrightarrow \cF$ is a triple
$(\cB,\, \cD,\, \cC)$, where $\cB$ is an $\cF$-nondegenerate braided multifusion 1-category, $\cD$ is an $\cE$-nondegenerate braided multifusion 1-category, and
$\cC: \cD \boxtimes_{\cE} \cF \nrightarrow \cB$ is an $\cF$-Witt equivalence.
\end{defn}
The key structural difference between the datum $\cB$ and $\cD$  in Definition~\ref{WDdef} is that $\cB$ is considered \emph{up to $\cF$-compatible braided monoidal equivalence} while $\cD$ is merely given up to \emph{$\cE$-Witt equivalence}. Indeed, this distinction becomes visible when considering higher morphisms in the groupoid of Witt squares (or when considering isomorphism classes of Witt squares).

For the reader's convenience, we explicitly describe 1-morphisms.

\begin{defn}
\label{WDequiv}
An equivalence between two Witt data 
$(\cB_1,\, \cD_1,\, \cC_1)$ and $(\cB_2,\, \cD_2,\, \cC_2)$ of type $\cE \to \cF$ is a triple
$(F,\, \cU,\, \cM)$, where $F: \cB_1 \to \cB_2$
is a braided monoidal equivalence compatible with the identification of symmetric centers with $\cF$ (this is data which we suppress from notation), $\cU: \cD_1\nrightarrow\cD_2$
is an $\cE$-Witt equivalence, and 
$$
\begin{tikzcd}
\cD_1 \boxtimes_{\cE} \cF
\arrow[dd,"\cC_1"']
\arrow[rr, "\cU \boxtimes_{\cE} \cF"]
& &
\cD_2 \boxtimes_{\cE} \cF
\arrow[dd,"\cC_2"]
\\
& &
\\
\cB_1
\arrow[rr, "_{F}\cB_2"']
\arrow[uurr,Rightarrow,shorten <= 1em, shorten >= 1em,"\cM"]
& &
\cB_2
\end{tikzcd}
$$
is an equivalence between $\cF$-Witt equivalences, where the bottom horizontal arrow is the Witt equivalence $\cB_1 \nrightarrow \cB_2$ induced from the braided monoidal equivalence $F: \cB_1 \to \cB_2$.  
\end{defn}

One can similarly give explicit descriptions of 2- and 3-morphisms in the space of Witt data of type $F:\cE \twoheadrightarrow \cF$. 

Incoherently, the relation between Witt squares and multifusion $2$-categorie can be summarized as follows, starting from \cite[Remark 4.2.9]{Decoppet2022;centers}:

\begin{rem}\label{prop:thibault}
    Since $\Mor^{\sss} \simeq \Mor_2^{\sss}$, it follows that multifusion 2-categories up to monoidal equivalence consist of the following data: 
\begin{enumerate}
    \item A symmetric multifusion 1-category $\cE$ up to symmetric monoidal equivalence,
    \item A braided multifusion 1-category $\mathcal{B}$ such that $\mathcal{Z}_{2}(\cB) \simeq \cE$, up to  $\cE$-Witt equivalence,
    \item A $\cB$-central multifusion 1-category $\cC$ such that $\mathcal{B}\rightarrow \mathcal{Z}(\mathcal{C})$ is faithful, up to $\cB$-central Morita equivalence.
\end{enumerate}
Moreover, the corresponding multifusion 2-category is the Morita dual to $\mathbf{Mod}(\cB)$ with respect to $\mathbf{Mod}(\cC)$ equipped with the obvious action.
\end{rem}

We note that the faithfulness condition on the functor $\mathcal{B}\rightarrow \mathcal{Z}(\mathcal{C})$ is there to ensure that the action of $\mathbf{Mod}(\cB)$ on $\mathbf{Mod}(\cC)$ is faithful in the sense of \cite[Definition 5.4.1]{Decoppet2022;Morita}, so that $\mathbf{Mod}(\cC)$ describes a Morita equivalence.

We will shortly explain how this set of data is related to Witt data. Before doing so, we need to establish the following technical result.

\begin{lem}\label{lem:technicalsymmetrictensorfunctor}
For a multifusion 2-category $\mathfrak{C}$ parameterised according to the data of Remark \ref{prop:thibault}, the symmetric tensor functor $\Omega\mathcal{Z}(\fC)\rightarrow \mathcal{Z}_{2}(\Omega\fC)$ is identified with the restriction of $\mathcal{B}\rightarrow \mathcal{Z}(\mathcal{C})$ to $\cE\twoheadrightarrow\Im(\cE)$. In particular, the symmetric tensor functor $\Omega\mathcal{Z}(\fC)\rightarrow \mathcal{Z}_{2}(\Omega\fC)$ is always faithful and dominant.
\end{lem}
\begin{proof}
We have that $\Omega\mathcal{Z}(\mathbf{Mod}(\cB))\simeq\cE$ by \cite[Lemma 2.1.6]{JFR}. Moreover, Morita equivalences induce braided monoidal equivalences at the level of Drinfeld centers \cite[Theorem 2.3.2]{Decoppet2022;centers}. Further, it was shown in \cite[Lemma 3.2.1]{Decoppet2022;centers} that $\Omega\mathfrak{C}\simeq \mathcal{Z}(\cC,\cB)$ as $\mathfrak{C}$ is the Morita dual to $\mathbf{Mod}(\cB)$ with respect to $\mathbf{Mod}(\cC)$. This shows that $\mathcal{Z}_{2}(\Omega\mathfrak{C})\simeq \Im(\cE)$. Moreover, by inspecting the proof of \cite[Theorem 2.3.2]{Decoppet2022;centers}, we find that $\Omega\mathcal{Z}(\fC)\twoheadrightarrow \mathcal{Z}_{2}(\Omega\fC)$ is identified with the canonical map $\cE\twoheadrightarrow\Im(\cE)$ as claimed.
\end{proof}

\begin{rem}
The last part of Lemma \ref{lem:technicalsymmetrictensorfunctor} also follows from our previous considerations. Though the arguments are essentially identical, we give some details for the reader's convenience. Given a multifusion 2-category $\fC$, Proposition \ref{prop:MF2Cpullback} factors the canonical 1-morphism $_{\mathbf{2Vect}}\fC_{\fC}$ in $\mathbf{Mor}^{\sss}$ as the composite of $_{\mathbf{2Vect}}\mathbf{Mod}(\Omega\fC)_{\mathbf{Mod}(\Omega\fC)}$ with $_{\mathbf{Mod}(\Omega\fC)}\fC_{\fC}$. Then, thanks to Proposition \ref{prop:Mugerfunctor}, we can apply the functor $\mathcal{Z}_{2}$ to the 1-morphism $_{\fC}\fC_{\mathbf{Mod}(\Omega\fC)}$ (note the change in direction), and, so doing, obtain a faithful dominant symmetric tensor functor. But, by construction, we have $$\mathcal{Z}_{2}(_{\fC}\fC_{\mathbf{Mod}(\Omega\fC)}):\mathcal{Z}_{2}([\fC])\simeq\Omega\mathcal{Z}(\fC)\twoheadrightarrow \Omega\mathbf{End}_{\fC-\mathbf{Mod}(\Omega\fC)}(\fC)\simeq\mathcal{Z}_{2}(\Omega\fC)\xleftarrow{\simeq} \mathcal{Z}_{2}([\mathbf{Mod}(\Omega\fC)]),$$ which is identified with the canonical symmetric tensor functor $\Omega\mathcal{Z}(\fC)\rightarrow \mathcal{Z}_{2}(\Omega\fC)$.
\end{rem}

Starting from Remark~\ref{prop:thibault} and Lemma~\ref{lem:technicalsymmetrictensorfunctor} we may now sketch the equivalence between Witt squares and multifusion $2$-categories:

\begin{prop}\label{thm:metoDelphic}
Let us fix $F:\mathcal{E}\twoheadrightarrow\mathcal{F}$, a faithful and dominant symmetric monoidal functor between symmetric multifusion 1-categories.
Multifusion 2-categories $\fC$ for which $\Omega\mathcal{Z}(\fC)\twoheadrightarrow \mathcal{Z}_{2}(\Omega\cF)$ is identified with $F$ are classified by triples $([\cB],\mathcal{D},[\cC])$, where $[\cB]$ is a class in $\Witt(\cE)$, $\mathcal{D}$ is an $\cF$-nondegenerate braided multifusion 1-category, and $[\cC]$ is (the Morita class of) an $\cF$-Witt equivalence between $[\cB\boxtimes_{\cE} \cF]$ and $[\mathcal{D}]$.
\end{prop}
\begin{proof}
Let $\fC$ be a multifusion 2-category equipped with an identification of $\Omega\mathcal{Z}(\fC)\twoheadrightarrow \mathcal{Z}_{2}(\Omega\cF)$ with $F:\mathcal{E}\twoheadrightarrow\mathcal{F}$. With respect to the data in Remark \ref{prop:thibault}, this fixes the symmetric multifusion 1-category $\cE$, but also $\cB$ up to (pointed!)\ $\cE$-Witt equivalence. Moreover, this also prescribes the restriction of $\cB\rightarrow\cZ(\cC)$ to $\cE$. We can then consider $\cD := \mathcal{Z}(\mathcal{C},\mathcal{B})^{\rev}$, the centralizer of the image of $\cB$ in $\cZ(\cC)$, which is an $\cF$-nondegenerate braided multifusion 1-category. Observe that this does not depend on the representative $\cB$ for the class $[\cB]$ in $\Witt(\cE)$. In particular, we have by construction and \cite[Proposition 4.3]{DNO2013} that $$\mathcal{Z}(\mathcal{C},\cF)\simeq (\cB\boxtimes_{\cE}\cF)\boxtimes_{\cF}\cD^{\rev},$$ i.e.\ the Morita class $[\cC]$ provides an $\cF$-Witt equivalence between $[\cB\boxtimes_{\cE}\cF]$ and $[\cD]$. This correspondence is manifestly bijective, so that the result follows.
\end{proof}

\section{(Multi)Fusion 2-Categories and their Parametrization}\label{section:F2Cparametrization}
Having developed a coherent method for extracting a Witt square from a (multi)fusion 2-category, we will now explain the correspondence between Witt squares and Delphic squares. (This corresponds to arrow (\pesos) in Figure \ref{fig:layout}.) With these relations in place we can then unpack the data of the Delphic squares so as to obtain the statements of Theorems \ref{thmalpha:AllBosons} and \ref{thmalpha:EmergentFermions}. (This corresponds to arrow (\euro) in Figure \ref{fig:layout}.) Finally, we describe how to explicitly construct a fusion 2-category given the data in our main theorems. (This corresponds to arrow (\textsterling) in Figure \ref{fig:layout}.)

\subsection{From Witt Squares to Delphic Squares}

We explain how to go from Witt squares to Delphic squares. In order to do so, we will use the language of superspaces developed in \S\ref{subsection:supergroups}, and, in particular, the restatement of Deligne's theorem \cite{deligne2002} given therein. 

As explained in~\S\ref{subsection:supergroups}, any symmetric multifusion $1$-category $\cF$ arises by Deligne's theorem \cite{deligne2002} as a limit of a certain functor to $\mathbf{SMF1C}$ which sends every object to $\mathbf{sVect}$ (but is nontrivial as a functor if $\cF$ is non-Tannakian). 

It then follows from \cite{DGNO2010braided} that there is an equivalence of spaces $$\mathbf{BMF1C}^{\ndeg}(\cF)\simeq \sHom(\Spec(\cF),\mathbf{BMF1C}^{\ndeg}(\sVect)),$$ given by de-equivariantization, where $\sHom$ denotes the space of supermaps, i.e.\ $\mathrm{B}\Z/2$-equivariant maps. Moreover, it also follows from Theorem \ref{thm:Wittlimits} that there is an equivalence of spaces $$\mathcal{W}itt(\cF)\simeq \sHom(\Spec(\cF),\mathcal{W}itt(\sVect)).$$ Now, if $F:\cE\twoheadrightarrow \cF$ is a faithful dominant symmetric tensor functor between symmetric multifusion 1-category, there is a corresponding map of superspaces $\mathsf{Spec}(F):\mathsf{Spec}(\mathcal{F})\rightarrow \mathsf{Spec}(\mathcal{E})$. The last identification above is in fact functorial, so that the canonical map $\mathcal{W}itt(\cF)\rightarrow \mathcal{W}itt(\cE)$ induced by $F$ can be identified with $\sHom(\Spec(\cE),\mathcal{W}itt(\sVect))\rightarrow \Hom(\Spec(\cF),\mathcal{W}itt(\sVect))$ induced by pulling back along the map of superspaces $\mathsf{Spec}(F)$. Putting the above discussion together, we find that Witt squares of type $\cE\twoheadrightarrow\cF$ can be equivalently rewritten as: 
\begin{equation}
\begin{tikzcd}
\pt
\arrow[r]
\arrow[d]
&
\sHom(\Spec(\cF),\mathbf{BMF1C}^{\ndeg}(\sVect))
\arrow[d,"{[-]}"]
\\
\sHom(\Spec(\cE),\Witt(\sVect))
\arrow[ur,Rightarrow,shorten <= 1em, shorten >= 1em,"\simeq"]
\arrow[r]
&
\sHom(\Spec(\cF),\Witt(\sVect))\,.
\end{tikzcd}
\end{equation}
\noindent Upon currying, these squares are exactly the Delphic squares of type $\Spec(\cF)\hookrightarrow\Spec(\cE)$ of Theorem~\ref{thm:Delphic}. This shows that, upon assuming that Theorem \ref{thm:Wittlimits} holds, then Theorem \ref{thm:Delphic} follows from Theorem \ref{thm:Wittsquare}.

\subsection{Examples}

\begin{ex}
    As an application of Theorem \ref{thmalpha:EmergentFermions} we show how fermionic strongly fusion 2-categories \cite{JFY:2020ivj} fit into the classification thereby recovering a result of \cite{decoppet:extension}. Recall that, by definition, a fermionic strongly fusion 2-category is a fusion 2-category $\fC$ such that $\Omega \fC \simeq \sVect$.
    In this case, the $\sVect$-nondegenerate braided fusion 1-category in Theorem \ref{thmalpha:EmergentFermions}
    is simply $\mathcal{A}=\sVect$. The next piece of data is a finite supergroup $(G,z)$ with $z\neq 1$, with ``minimal'' supersubgroup $(H,z)=(\mathbb{Z}/2,z)$. In particular, the super action $\rho$ is necessarily trivial. The remaining piece of data is a class $\varpi$ in $\SH^4(G,z)$, and we write $\mathbf{2sVect}^{\varpi}_{(G,z)}$ for the corresponding fermionic strongly fusion 2-category. Namely, the homotopy between $\rho$ and $\varpi|_{(\mathbb{Z}/2,z)}$ is essentially unique as $\SH^3(\mathbb{Z}/2,z)=0$. 
    This shows that fermionic strongly fusion 2-categories up to monoidal equivalence are classified by a pair $((G,z), \varpi \in \SH^4(G,z))$ of a supergroup $(G,z)$ and a class $\varpi$ up to isomorphism $\phi: (G,z) \to (G', z')$  of supergroups and the relation $\phi^* \varpi' \sim \varpi \in \SH^4(G,z).$

In other words, the set of monoidal equivalence classes of fermionic strongly fusion $2$-categories is in bijection with the set 
\[
\bigsqcup_{[(G,z)]} \SH^4(G,z)/\Aut(G,z)
\]
where the disjoint union is over the set of all isomorphism classes of finite supergroups. 
   This recovers the classification of \cite{decoppet:extension}.
\end{ex}

It follows from our classification that the last example above, as well as its bosonic counterpart, can be generalized. More precisely, we have the following generalizations of \cite[Theorems\ A \& B]{JFY:2020ivj}.

\begin{cor}
Let $\fC$ be a fusion 2-category such that $\Omega\mathfrak{C}$ is either a nondegenerate or a slightly degenerate braided fusion 1-category. Then the set of connected components $\pi_0(\mathfrak{C})$ inherits a group structure from the monoidal product. In particular, $\fC$ is faithfully graded by its connected components.
\end{cor}

In particular, all the (necessarily bosonic) fusion 2-categories $\fC$ for which $\Omega\fC$ is nondegenerate are as described in Example \ref{ex:generalizedbosonicstronglyfusion}. That is, they are of the form $\mathbf{Mod}(\mathcal{A})\boxtimes \mathbf{2Vect}_G^{\pi}$, where $\cA$ is a nondegenerate braided fusion 1-category, $G$ is a finite group, and $\pi$ is a 4-cocycle for $G$ with coefficients in $\mathbb{C}^{\times}$. For later use, we also record the fermionic version of this result.

\begin{ex}\label{ex:generalizedFSF2C}
Fix a slightly-degenerate braided fusion $1$-category $\cA$. 

Fermionic fusion 2-categories equipped with an identification of $\Omega\mathfrak{C}$ with $\cA$ (up to monoidal equivalence fixing $\cA$) are classified by a pair of a finite supergroup $(G,z)$ and a class $\varpi$ in 
$$\sHom(\mathrm{B}(G,z),\mathrm{B}\mathcal{A}ut_{\mathbf{sVect}}([\cA]))$$
up to the evident equivalence relation. 
In fact, upon chosing a minimal nondegenerate extension for $\cA$, we get an identification $$\sHom(\mathrm{B}(G,z),\mathrm{B}\mathcal{A}ut_{\mathbf{sVect}}([\cA]))\cong \SH^4(G,z).$$
In particular, the set  of fusion $2$-categories $\fC$ with an identification $\Omega \fC \simeq \cA$ up to monoidal equivalence which fix $\cA$ is in bijection with the set 
\[
\bigsqcup_{[(G,z)]} \SH^4(G,z)/\Aut(G,z)
\]
where the disjoint union is over the set of isomorphism classes of finite supergroups.

We emphasize that unlike in the bosonic case, there is no canonical decomposition of such fermionic fusion 2-categories into a 2-Deligne tensor product. (Though such a non-canonical decomposition does exist up to Morita equivalence by \cite{Decoppet2022;centers} by picking a minimal nondegenerate extension for $\cA$.) Rather, the fermionic fusion 2-categories under consideration arise via extension theory \cite{decoppet:extension}. This is more straightforward to see using spaces over $K(\mathbb{Z}/2,2)$. Namely, in this perspective, the relevant objects are $G_b$, equipped with the map classifying the extension $(G,z)$, and $\mathcal{A}ut([\cA])$, also equipped with its canonical map. We then have $$\sHom(\mathrm{B}(G,z),\mathrm{B}\mathcal{A}ut_{\mathbf{sVect}}([\mathcal{A}]))\simeq \Hom_{/K(\mathbb{Z}/2,2)}(\mathrm{B}G_b,\mathrm{B}\mathcal{A}ut([\cA])),$$ and the $G_b$-graded extension is classified by the class in $\Hom(\mathrm{B}G_b,\mathrm{B}\mathcal{A}ut([\cA]))$.

Conversely, given a class $\varsigma$ in $\Hom(\mathrm{B}G_b,\mathrm{B}\mathcal{A}ut([\cA]))$, we can consider the corresponding fermionic fusion 2-category $\mathfrak{C}$, which is a faithfully $G_b$-graded extension of $\mathbf{Mod}(\mathcal{A})$. The extension $(G,z)$ of $G_b$ can be recovered using the canonical map $\mathrm{B}\mathcal{A}ut([\cA])\rightarrow \mathrm{B}\Aut^{\br}(\mathbf{sVect})$. For later use, we wish to describe the group $\Inv(\mathfrak{C})$ of invertible objects of $\mathfrak{C}$. To this end, recall that there is a fiber sequence
\begin{equation}\label{eq:fibseqPic}
    \mathcal{P}ic(\mathcal{A})\rightarrow \Aut^{\br}(\mathcal{A})\rightarrow \Aut([\mathcal{A}]).
\end{equation}
 This follows, for instance, from \cite[Proposition 5.4]{DN}. By definition, there is an equivalence $\Omega\Aut([\mathcal{A}])\simeq \mathcal{Z}(\mathbf{Mod}(\cA))^{\times}$. Furthermore it is shown in \cite[Theorem 4.10]{DN} that the braided fusion 2-category $\mathcal{Z}(\mathbf{Mod}(\cA))$ is identified with $\mathbf{Mod}^{\br}(\mathcal{A})$, the braided monoidal 2-category of finite semisimple braided $\cA$-module 1-categories, and $\mathbf{Mod}^{\br}(\mathcal{A})^\times= Pic^{\br}(\mathcal{A})$. By looking at the long exact sequence in homotopy groups for \eqref{eq:fibseqPic} we find: $$0\rightarrow \Inv(\cA)\rightarrow Aut^{\br}(Id_{\mathcal{A}})\rightarrow Pic^{\br}(\cA)\xrightarrow{\phi} Pic(\cA)\rightarrow Aut^{\br}(\cA),$$ where the map $\phi$ is simply forgetting the braided structure. Finally, it follows by unpacking the construction of a graded extension given in \cite{decoppet:extension} that the canonical short exact sequence $$\label{eq:InvCextesnion}0\rightarrow Pic(\cA)\rightarrow \Inv(\mathfrak{C})\rightarrow G_b\rightarrow 1$$ is classified by $\phi_*\varsigma$, a class in $\operatorname{H}^2(G_b,Pic(\cA))$.
\end{ex}

\subsection{Unfolding Delphic Squares}\label{subsection:DelphicUnfold}
We now move on to unpacking the content of the Delphic squares for fusion 2-categories. More precisely, we explain how the parametrizations in Theorem \ref{thmalpha:AllBosons} and \ref{thmalpha:EmergentFermions} arise from the Delphic squares, and therefore explain the arrow (\euro) in Figure \ref{fig:layout}.

We begin by considering the all bosons case. In particular, the faithful dominant symmetric tensor functor $\mathcal{E}\twoheadrightarrow \mathcal{F}$ goes between two Tannakian symmetric fusion 1-categories, and therefore corresponds upon taking $\mathsf{Spec}$ to an inclusion of finite supergroups $(H,z)\hookrightarrow (G,z)$ with $z=1$. In particular, it follows from Lemma \ref{lem:sdegfixedpoint} that the top horizontal map of superspaces $\mathrm{B}H\rightarrow \ob(\mathbf{BMF1C}^{\ndeg}(\sVect))$ corresponds to a map of spaces $\mathrm{B}H\rightarrow \ob(\mathbf{BMF1C}^{\ndeg}(\Vect))$. Likewise, the bottom horizontal map of superspaces $\mathrm{B}(G,1)\rightarrow \mathcal{W}itt(\sVect)$ corresponds to a map of spaces $\mathrm{B}G\rightarrow \mathcal{W}itt(\Vect)$ by Theorem \ref{thm:Wittlimits}. The simplified Delphic square is the following commutative square of spaces:
$$
\begin{tikzcd}
    \mathrm{B}H \arrow[r,"\rho"] \arrow[d, "\iota"'] & \ob(\mathbf{BMF1C}^{\ndeg}(\Vect)) \arrow[d, "{[-]}"] \\
    \mathrm{B}G \arrow[r, "\pi"'] \arrow[ur,Rightarrow,shorten <= 1em, shorten >= 1em,"\simeq"] & \mathcal{W}itt(\Vect)\,.
\end{tikzcd}
$$
Moreover, as $\mathrm{B}H$ is a connected space, we may choose a unique (up to nonunique isomorphism!)\ point in $\mathrm{B}H$ which then selects a point in the space $\ob(\mathbf{BMF1C}^{\ndeg}(\Vect))$, i.e.\ a nondegenerate braided fusion 1-category $\mathcal{A}$. (This arbitrary choice of a basepoint of $\mathrm{B}H$ leads to the subtleties encountered in Remarks~\ref{rem:vague} and~\ref{rem:subtlety}.)
Likewise, the bottom horizontal map selects a Witt class in $\mathcal{W}itt(\Vect)$, which is $[\mathcal{A}]$. But, we have $\mathcal{A}ut([\mathcal{A}])\simeq \mathrm{B}^3\mathbb{C}^{\times}$, so that the last square above may be simplified even further to
\begin{equation}\label{eq:bosonicDelphic}
\begin{tikzcd}
    \mathrm{B}H \arrow[r,"\rho"] \arrow[d, "\iota"'] & \mathrm{B}\mathcal{A}ut^{\br}(\mathcal{A}) \arrow[d, "{[-]}"] \\
    \mathrm{B}G \arrow[r, "\pi"'] \arrow[ur,Rightarrow,shorten <= 1em, shorten >= 1em,"\simeq"] & \mathrm{B}^4\mathbb{C}^{\times}\,.
\end{tikzcd}
\end{equation}
Such squares encode exactly the data of Theorem \ref{thmalpha:AllBosons}. Recall from Remark \ref{rem:ENOanomaly} that the left vertical arrow admits a description in terms of the classical extension extension theory of fusion 1-categories developed in \cite{ENO2010}. In particular, the composite $[\rho]$, the anomaly of the action $\rho$, is exactly the obstruction to lifting $\rho$ to a map $\mathrm{B}H\rightarrow \mathrm{B}\mathcal{P}ic(\cA)$. Finally, the commutativity of the square \eqref{eq:bosonicDelphic} is witnessed by a homotopy between $[\rho]$ and $\pi|_H$.

Some of the data featured in the square \eqref{eq:bosonicDelphic} above can easily be extracted from a given bosonic fusion 2-category $\mathfrak{C}$. More precisely, we have:
\begin{itemize}
    \item The Tannakian symmetric fusion 1-categroy $\Rep(H)$ is $\mathcal{Z}_{2}(\Omega \mathfrak{C})$. 
    \item The nondegenerate braided fusion 1-category $\mathcal{A}$ is given by $(\Omega \mathfrak{C})_H = \Omega \mathfrak{C}\boxtimes_{\Rep(H)}\Vect$, the de-equivariantization of $\Omega \mathfrak{C}$.
    \item The Tannakian symmetric fusion 1-category $\Rep(G)$ is $\Omega \cZ(\mathfrak{C})$, and the map $\Omega \cZ (\fC) \rightarrow \cZ_{2} (\Omega \fC)$ is identified with $\Rep(G)\twoheadrightarrow\Rep(H)$ induced by $H\subseteq G$.
\end{itemize}

\begin{rem}
The data of $\pi$ can be recovered from the Drinfeld center. More precisely, it is known by \cite{Decoppet2022;centers} that, with $\mathfrak{C}$ the bosonic fusion 2-cagegory corresponding to the square \eqref{eq:bosonicDelphic}, we have $\mathcal{Z}(\fC)\simeq \mathcal{Z}(\mathbf{2Vect}^{\pi}_G)$ as braided fusion 2-categories. Conversely, given $\mathcal{Z}(\fC)$ as a plain fusion 2-category, there is a canonical Morita equivalence between $\mathcal{Z}(\fC)$ and $\mathbf{2Vect}^{\pi}_G\boxtimes(\mathbf{2Vect}^{\pi}_G)^{\mop}$ supplied by the canonical separable algebra $\mathbf{Vect}$ in $\mathcal{Z}(\fC)^0\simeq \mathbf{2Rep}(G)$.
\end{rem}

We now unpack the Delphic square in the fermionic case. In particular, the faithful dominant symmetric tensor functor $\mathcal{E}\twoheadrightarrow \mathcal{F}$ goes between two super-Tannakian symmetric fusion 1-categories. Upon taking $\mathsf{Spec}$, it therefore corresponds to an inclusion of finite supergroups $(H,z)\subseteq (G,z)$ with $z\neq 1$. Because supergroups are connected as superspaces, we may again pick a basepoint of $B(H,z)$ (whose unique but not contractible choice again leads to the discussed subtleties). The image of this basepoint under the top horizontal map then selects a slightly degenerate braided fusion 1-category $\mathcal{A}$, and the corresponding Delphic square becomes:
\begin{equation}\label{eq:delphic2}
\begin{tikzcd}
\mathrm{B}(H,z)
\arrow[d,"\iota"']
\arrow[r, "\rho"]
&
{\mathrm{B}\mathcal{A}ut^{\br}_{\sVect}(\mathcal{A})}
\arrow[d,"{[-]}"]
\\
\mathrm{B}(G,z)
\arrow[r, "\varpi"']
\arrow[ur,Rightarrow,shorten <= 1em, shorten >= 1em,"\simeq"]
&
{\mathrm{B}\mathcal{A}ut_{\sVect}([\mathcal{A}])}\,.
\end{tikzcd}
\end{equation}

\noindent The bottom horizontal map of superspaces $\varpi$ is classified by a torsor over $\SH^4(G,z)$. Namely, as explained at the end of \S\ref{subsection:supergroups}, the superspace ${\mathrm{B}\mathcal{A}ut_{\sVect}([\mathcal{A}])}$ is non-canonically identified with ${\mathrm{B}\mathcal{A}ut_{\sVect}([\sVect])}=\Sigma^4\SH$. The commutativity of the above square of superspaces is witnessed by a homotopy between $[\rho]$, the superanomaly of the action $\rho$ as defined in \eqref{eq:superanomaly}, and $\varpi|_{(H,z)}$. The different choices for this homotopy form a torsor over $\SH^3(H,z)$. Thence, we have obtained the statement of theorem \ref{thmalpha:EmergentFermions}.

Similarly to the bosonic case, some of the above data can be effortlessly extracted from a fermionic fusion 2-category $\mathfrak{C}$. More precisely, we have:
\begin{itemize}
    \item The super-Tannakian symmetric fusion 1-category $\Rep(H,z)$ is $\mathcal{Z}_{2}(\Omega \mathfrak{C})$.
    \item The slightly degenerate braided fusion 1-category $\cA$ is $(\Omega\fC)_{H_b} = \Omega\fC\boxtimes_{\Rep(H,z)}\sVect$, the de-equivariantization of $\Omega\fC$ with respect to $H_b = H/z$.
    \item The faithful dominant symmetric tensor functor $\Rep(G,z)\twoheadrightarrow \Rep(H,z)$ is given by the inclusion $(H,z) \hookrightarrow (G,z)$ is $\Omega \cZ (\fC) \rightarrow \cZ_{2} (\Omega \fC)$ .
\end{itemize}

\subsection{Reconstructing the Fusion 2-Category}\label{subsection:reconstruction}

We now explain how to use the data of Theorems \ref{thmalpha:AllBosons} and \ref{thmalpha:EmergentFermions} to reconstruct a fusion 2-category. This is the content of the arrow (\textsterling) in Figure \ref{fig:layout}.

Firstly, consider the data from Theorem \ref{thmalpha:AllBosons}, i.e., 
a nondegenerate braided fusion 1-category $\cA$, an inclusion of finite groups $H\subset G$, an action $\rho: H\to \mathcal{A}ut^{\br}(\cA)$ (identified with a monoidal functor 
$G\to \textbf{Mod}(\cA)^\times)$, a 4-cocycle $\pi$ in $Z^4(G,\, \mathbb{C}^\times)$, and a $3$-cochain $\mu$ in $C^3(H,\, \mathbb{C}^\times)$ such that $[\rho]= \pi|_H\, d\mu$. We construct a fusion $2$-category $\mathfrak{C}$ corresponding to this data as follows:

\begin{enumerate}
    \item Consider the fusion $2$-category $\mathfrak{D}= \mathbf{2Vect}_G^{\pi} \boxtimes \mathbf{Mod}(\cA)$, and let us write $X_g$ with $g\in G$ for the simple objects of $\mathbf{2Vect}_G^{\pi}$. 
    \item We consider the diagonal map 
\begin{align*}
R: H &\to \mathbf{2Vect}_G^{\pi} \boxtimes \mathbf{Mod}(\cA) \\
h &\mapsto\ \ \ \ \ \ \,  X_h \boxtimes \rho(h)^{-1}
\end{align*}
whom the 3-cocycle $\mu$ provides with a monoidal $2$-functor structure because the anomaly $[\rho^{-1}]$ cancels out with the restriction of $\pi$ to $H$. While the map $\rho$ is defined to have the target $\mathcal{A}ut(\cA)$, we take $\rho$ in the above map as its lift to $\Mod(\cA)^\times$.
\item The direct sum $A:=\bigoplus_{h\in H}\, R(h)$ is a strongly connected rigid algebra in $\mathfrak{D}$ by construction. The fusion 2-category $\mathfrak{C}$ is given by $\mathbf{Bimod}_{\fD}(A)$,
the fusion 2-category of $A$-$A$-bimodules in $\mathfrak{D}$.
\end{enumerate}

\begin{rem}
    There is a ``lower categorical'' analogue of the data from Theorem \ref{thmalpha:AllBosons} and the corresponding reconstruction. Given a fusion 1-category $\cA$ and a group  homomorphism  $\rho: H \to \Inv(\cA)$ there is an anomaly $O_3(\rho)\in \mathrm{H}^3(H,\, \mathbb{C}^\times)$ lifted from the associativity constraint of the categorical group $\Inv(A)$. Consider the data consisting of a subgroup $H \subset G$ of a finite group, $\pi \in Z^3(G,\, \mathbb{C}^\times)$, a fusion 1-category $\cA$, a homomorphism $\rho: H \to \Inv(\cA)$, and a $2$-cochain $\mu \in Z^2(H,\, \mathbb{C}^\times)$ such that $O_3(\rho) = \pi|_H\, d\mu$. There is a diagonal monoidal functor 
    \begin{align*}
    H &\to \Vect_G^{\pi} \boxtimes \cA \\
    h &\mapsto\ \ \ \ \ \ \   h \boxtimes \rho(h)^{-1}.
    \end{align*}
    Let $A$ be the image of the group algebra $\mathbb{C}[H]$ in $\mathcal{D}=\Vect_G^{\pi} \boxtimes \cA$, we can construct the fusion 1-category $\mathbf{Bimod}_{\cD}(A)$. Let us also note that, by definition, all group-theoretical fusion 1-categories are obtained via this construction.

\end{rem}

We now turn our attention to the fermionic case. Instead of using the data in Theorem \ref{thmalpha:EmergentFermions} directly, we give a construction based on the closely related data given in Equation \eqref{eq:delphic2}, since this has the advantage of not requiring the choice of a minimal nondegenerate extension. Let $\fC$ be a fermionic fusion 2-category. The data in Equation \eqref{eq:delphic2} consists of an inclusion of supergroups $(H,z)\hookrightarrow (G,z)$, a class $\varpi$ in $\sHom(\mathrm{B}(G,z),\mathrm{B}\mathcal{A}ut_{\mathbf{sVect}}([\mathcal{A}]))$, a superaction $\rho:(H,z)\rightarrow (\mathcal{A}ut^{\br}_{\mathbf{sVect}}(\mathcal{A}), (-1)^f)$, as well as a homotopy $\mu$ between the superanomaly $[\rho]$ and the restriction of $\varpi|_{(H,z)}$. We construct a fusion $2$-category $\mathfrak{C}$ corresponding to this data as follows:

\begin{enumerate}
    \item As discussed in Example \ref{ex:generalizedFSF2C}, we can consider the fusion $2$-category $\mathfrak{D}$ that is the $G_b$-graded extension of $\mathbf{Mod}(\cA)$ classified by $\varpi$.\footnote{There is a canonical Morita equivalence between $\fC$ and $\fD$ (see \cite{Decoppet2022;centers}). This is another reason for not choosing a minimal nondegenerate extension.}
    \item It follows from the existence of the homotopy $\mu$ that the class $\phi_*(\varpi|_{(H,z)})$ in $\mathrm{H}^2(H_b,Pic(\cA))$ classifying the extension of Equation \eqref{eq:InvCextesnion} is trivial. In particular, there is a canonical section $H_b\hookrightarrow \Inv(\mathfrak{C})$ to the projection map $\Inv(\mathfrak{C})\twoheadrightarrow G_b$, which we write as $h\mapsto X_h$. We can therefore consider the diagonal map 
\begin{align*}
R: H_b &\to\ \ \ \ \, \mathfrak{D} \\
h &\mapsto X_h \boxtimes \rho(h)^{-1}
\end{align*}
whom the 3-cocycle $\mu$ provides with a monoidal structure because the anomaly $[\rho^{-1}]$ cancels out with the restriction of $\varpi$ to $H$.
\item The direct sum $A:=\bigoplus_{h\in H}\, R(h)$ is a strongly connected rigid algebra in $\mathfrak{D}$ by construction. The fusion 2-category $\mathfrak{C}$ is given by $\mathbf{Bimod}_{\fD}(A)$,
the fusion 2-category of $A$-$A$-bimodules in $\mathfrak{D}$.
\end{enumerate}

\subsection{Profits}\label{sec:proofsofprofits}

The classification of fusion 2-categories allows us to state a categorified version of rank finiteness and Ocneanu rigidity for fusion 2-categories in \S\ref{section:profits}.
We now fill in the proofs of these results. We start with a slight elaboration on \cite[Lemma 4.2.1]{Decoppet2022;centers}, which will be used throughout this subsection.

\begin{lem}\label{lem:doublecoset}
Let $\fC$ be a fusion 2-category such that $G:=\pi_0(\fC)$ forms a group. (For instance this holds if every connected component of $\fC$ contains an invertible object.) Let $A$ be a strongly connected rigid algebra  (which implies separable by \cite{Decoppet2022;centers}). Let $H$ be the support of $A$ in $\pi_0(\fC)$, which is necessarily a group. 
Then 
$$\pi_0(\mathbf{Bimod}_{\fC}(A))\cong H\backslash G /H.$$
\end{lem}
\begin{proof}
We begin by showing that $H$, the support of $A$, is a subgroup of $G$. Given that $A$ is strongly connected, it follows that the unit $\eta^m: Id_{A\Box A}\rightarrow m^*\circ m$ of the adjunction between $m$ and $m^*$ as $A$-$A$-bimodule 1-morphisms is the inclusion of a summand. Thus, if $C$ and $D$ are simple objects of $\fC$ with non-zero 1-morphisms $f:C\rightarrow A$ and $g:D\rightarrow A$, then $f\Box g$ is non-zero. Namely, if $f\Box g=0$, then $C\Box D = 0$, which is impossible. We therefore find that $m^*\circ m\circ (f\Box g)$ is a non-zero 1-morphism in $\fC$. In particular, $m\circ (f\Box g)$ is non-zero as desired.

Recall that $\mathbf{Bimod}_{\fC}(A)$ is a finite semisimple 2-category (see \cite{Decoppet2022;rigid} and \cite{Decoppet2021;compact}). Let $C$ be a simple object of $\fC$. Then $A\Box C\Box A$ is a simple $A$-$A$-bimodule. Namely, under the canonical equivalence $$\Hom_{A\mathrm{-}A}(A\Box C\Box A,A\Box C\Box A)\simeq \Hom_{\fC}(C,A\Box C\Box A)$$ the identity 1-morphism on $A\Box C\Box A$ is sent to $i\Box Id_C\Box i$. But, as $C$ is simple and $i:I\hookrightarrow A$ is the inclusion of a simple summand by hypothesis, the second 1-morphism is necessarily simple.

Now, given any simple $A$-$A$-bimodule $P$, pick any simple object $D$ in $\fC$ together with a non-zero 1-morphism $D\rightarrow P$ in $\fC$. Then we get an induced 1-morphism of $A$-$A$-bimodules $A\Box D\Box A\rightarrow P$, which is non-zero. Moreover, $D$ was arbitrary, so that we have shown that for any two simple objects $C$ and $D$ of $\fC$, there exists a non-zero 1-morphism of $A$-$A$-bimodules between $A\Box C\Box A$ and $A\Box D\Box A$ if the classes of $C$ and $D$ in $H\backslash G /H$ agree. But, if these classes are distinct, then there is no 1-morphism in $\fC$ between $A\Box C\Box A$ and $A\Box D\Box A$, so that, a fortiori, there is non-zero 1-morphism of $A$-$A$-bimodules between. This finishes the proof.
\end{proof}

Combining the last result above together with our classification theorems yield a number of interesting corollaries, which we now explain. Firstly, recall that given a fusion 1-category $\cC$, its rank $\mathrm{rk}(\cC)$ is given by its number of equivalence classes of simple objects.

\begin{defn}
The rank of a fusion 2-category $\fC$ is given by $\mathrm{rk}(\fC):=|\pi_0(\fC)|\cdot \mathrm{rk}(\Omega\fC)$.
\end{defn}

\begin{cor}[Rank finiteness]
Up to equivalence, there exists only finitely many fusion 2-categories of a given fixed rank.
\end{cor}
\begin{proof}
It was established in \cite{JMNR2010rank} that there are only finitely many equivalence classes of braided fusion 1-categories of any fixed rank. In the parameterisations of Theorems \ref{thmalpha:AllBosons} and \ref{thmalpha:EmergentFermions}, this determines $\cA$ and $H$, or $(H,z)$, completely. Now, given any finite (super) group $G$ with an embedding $H\hookrightarrow G$, we have $|G|\leq |H\backslash G /H|\cdot |H|^2$. It therefore follows from the discussion in \S\ref{subsection:reconstruction} and Lemma \ref{lem:doublecoset} that the order of $G$ is bounded above. Finally, the remaining data consists of classes in positive degree in either cohomology with $\mathbb{C}^{\times}$ coefficients or supercohomology of finite groups. It is well-known that the former are always finite groups, as for the latter, this follows from the Atiyah-Hirzebruch spectral sequence. This concludes the proof.
\end{proof}

\begin{cor}[Ocneanu rigidity]
Fusion 2-categories admit no non-trivial deformations. \qed
\end{cor}

Finally, our classification allows one to quickly read off the inhabited fusion channels in any fusion $2$-category. We will phrase the precise result in terms of ``possibilistic hypergroups''\footnote{Hypergroups, which date to  \cite{MartyHypergroup}, axiomatize group-like structures with multi-valued composition. The literature contains many inequivalent definitions.}. 

\begin{defn}
Let $X$ be a set, and $\cP_{<\infty}(X)$ the set of finite subsets of $X$, thought of as an abelian idempotent semigroup under $+ := \cup$; we can identify $X$ with the singletons in $\cP_{<\infty}(X)$, so that every element of $\cP_{<\infty}(X)$ is a sum of elements of $X$. A \emph{possibilistic hypergroup structure} on $X$ is an semiring structure with underlying abelian semigroup $\cP_{<\infty}(X)$ such that the multiplicative unit $1$ is singleton, and such that for each $x \in X$, there is are unique (necessarily equal) $y,y' \in X$ such that $1 \in x\cdot y$ and $1 \in y' \cdot x$. 
\end{defn}

Given a fusion 2-category $\fC$, we define its \emph{fusion possibilistic hypergroup} $\cH(\fC)$ to have underlying set $\pi_0\fC$, with multiplication law
\[ \lbrack C\rbrack\cdot \lbrack D\rbrack = \{ \lbrack E\rbrack  \text{ s.t.\ } \Hom_{\fC}(C\Box D,E) \text{ is nonzero} \}.\]
We leave it to the reader to verify that this indeed is well-defined and does not depend on the choices of simples $C,D,E$ in the components $[C], [D], [E]$.
\begin{warn}
    In contrast, we point out that the set $\pi_0 \fC$ does not carry in any obvious way the structure of an `integral fusion ring'; for example, the number of simple objects in the semisimple category $\Hom_{\fC}(C \Box D, E)$ \emph{does} in general depend on the choice of simples $C,D,E$ and is not well-defined at the level of components.
\end{warn}

Likewise, to any inclusion $(H,z)\hookrightarrow (G,z)$ of supergroups, we associate a possibilistic hypergroup $\mathcal{H}(H\backslash G/H)$ on the set of double cosets $H \backslash G / H$ with multiplication law
\[ \lbrack f\rbrack\cdot \lbrack g\rbrack = \{ \lbrack k \rbrack  \text{ s.t.\ } k \in fHg \}. \]

\begin{cor}
Let $\mathfrak{C}$ be a fusion 2-category whose associated inclusion of supergroups is $(H,z)\hookrightarrow (G,z)$. Then there is an isomorphism of possibilistic hypergroups $$\mathcal{H}(\mathfrak{C})\cong \mathcal{H}(H\backslash G/H).$$
\end{cor}
\begin{proof}
Thanks to the discussion in \S\ref{subsection:reconstruction}, the fusion 2-category $\mathfrak{C}$ is equivalent to $\mathbf{Bimod}_{\mathfrak{D}}(A)$ where $\mathfrak{D}$ is a fusion 2-category with enough invertible objects and such that $\pi_0(\mathfrak{D})\cong G/z$, and $A$ is a strongly connected rigid algebra in $\mathfrak{D}$ supported on $H/z\subseteq G/z$. Thanks to Lemma \ref{lem:doublecoset}, we find that $\pi_0(\mathfrak{C})\cong (H/z)\backslash (G/z) / (H/z) = H\backslash G/H$. Moreover, if $D$ and $E$ are simple objects of $\mathfrak{D}$, then we have established in the proof of Lemma \ref{lem:doublecoset} that both $A\Box D\Box A$ and $A\Box E\Box A$ are simple objects of $\mathbf{Bimod}_{\mathfrak{D}}(A)$. It follows from \cite{Decoppet2022;Morita} that the monoidal structure of $\mathbf{Bimod}_{\mathfrak{D}}(A)$ is given by $\Box_A$ the relative tensor product over $A$. In particular, we have $$(A\Box D\Box A)\Box_A(A\Box E\Box A)\simeq A\Box (D\Box A\Box  E)\Box A.$$ This proves the result.
\end{proof}

\section*{Acknowledgments}
We are grateful to Christopher Douglas and Corey Jones for many discussions on various topics related to this paper, as well as Lukas Müller for discussions in Section 2.

It is our great pleasure to thank the American Institute of Mathematics for having supported this project through the SQuaREs program.
In addition:
TD, TJF, and JP were supported by the Simons Collaboration on Global Categorical Symmetries;
PH was supported by US ARO grant W911NF2310026;
TJF was supported by the NSERC grant  RGPIN-2021-02424;
DN was supported by the NSF grant DMS-2302267;
DP was supported by the NSF grant DMS-2154389;
JP was supported by the NSF Grant DMS-2146392;
DR was supported by the Deutsche Forschungsgemeinschaft under the Emmy Noether program -- 493608176, and the Collaborative Research Center (SFB) 1624 ``Higher structures, moduli spaces and integrability'' -- 506632645;
and
MY was supported by the EPSRC Open Fellowship EP/X01276X/1.

\bibliographystyle{alpha}
{\footnotesize{
\bibliography{f2c}

\newcommand{\etalchar}[1]{$^{#1}$}
\begin{thebibliography}{LMGR{\etalchar{+}}24}

\bibitem[AF18]{ayala2018flagged}
David Ayala and John Francis.
\newblock Flagged higher categories.
\newblock {\em Topology and quantum theory in interaction, Contemporary Mathematics}, 718:137--173, 2018.
\newblock \arxiv{1801.08973}.

\bibitem[AH99]{MR1686551}
Marta Asaeda and Uffe Haagerup.
\newblock Exotic subfactors of finite depth with {J}ones indices {$(5+\sqrt{13})/2$} and {$(5+\sqrt{17})/2$}.
\newblock {\em Communications in Mathematical Physics}, 202(1):1--63, 1999.
\newblock \arXiv{math.OA/9803044}.

\bibitem[AMF16]{MR3543452}
David Aasen, Roger S.~K. Mong, and Paul Fendley.
\newblock Topological defects on the lattice: {I}. {T}he {I}sing model.
\newblock {\em Journal of Physics A}, 49(35):354001, 46, 2016.
\newblock \arxiv{1601.07185}.

\bibitem[AMP23]{MR4565376}
Narjess Afzaly, Scott Morrison, and David Penneys.
\newblock The {C}lassification of {S}ubfactors with {I}ndex at {M}ost {$5\frac14$}.
\newblock {\em Memoirs of the American Mathematical Society}, 284(1405), 2023.
\newblock \arxiv{1509.00038}.

\bibitem[BBFP24]{BBFP}
Thomas Bartsch, Mathew Bullimore, Andrea E.~V. Ferrari, and Jamie Pearson.
\newblock {Non-invertible symmetries and higher representation theory {II}}.
\newblock {\em SciPost Physics}, 17:067, 2024.
\newblock \arXiv{2212.07393}.

\bibitem[BBSNT23]{BBSNT}
Lakshya Bhardwaj, Lea~E. Bottini, Sakura Schafer-Nameki, and Apoorv Tiwari.
\newblock {Non-invertible symmetry webs}.
\newblock {\em SciPost Physics}, 15(4):160, 2023.
\newblock \arxiv{2212.06842}.

\bibitem[BJS21]{BJS}
Adrien Brochier, David Jordan, and Noah Snyder.
\newblock On dualizability of braided tensor categories.
\newblock {\em Compositio Mathematica}, 3:435--483, 2021.
\newblock \arxiv{1804.07538}.

\bibitem[BM24]{ben2024naturality}
Shay Ben-Moshe.
\newblock Naturality of the $\infty$-categorical enriched yoneda embedding.
\newblock {\em Journal of Pure and Applied Algebra}, 228(6):107625, 2024.
\newblock \arXiv{2301.00601}.

\bibitem[BP24]{Bullimore:2024khm}
Mathew Bullimore and Jamie~J. Pearson.
\newblock {Towards All Categorical Symmetries in 2+1 Dimensions}.
\newblock 8 2024.
\newblock \arXiv{2408.13931}.

\bibitem[BPSN{\etalchar{+}}24]{Bhardwaj:2024qiv}
Lakshya Bhardwaj, Daniel Pajer, Sakura Schafer-Nameki, Apoorv Tiwari, Alison Warman, and Jingxiang Wu.
\newblock Gapped phases in (2+1)d with non-invertible symmetries: Part {I}.
\newblock 8 2024.
\newblock \arxiv{2408.05266}.

\bibitem[D{\'e}c21]{decoppet2021finite}
Thibault~D. D{\'e}coppet.
\newblock Finite semisimple module 2-categories.
\newblock {\em To appear in \textit{Selecta Mathematica}}, 2021.
\newblock \arxiv{2107.11037}.

\bibitem[D{\'e}c22a]{Decoppet2022;centers}
Thibault~D. D{\'e}coppet.
\newblock {D}rinfeld centers and {M}orita equivalence classes of fusion 2-categories.
\newblock {\em To appear in \textit{Compositio Mathematica}}, 2022.
\newblock \arxiv{2211.04917}.

\bibitem[D{\'e}c22b]{decoppet2020;comparison}
Thibault~D. D{\'e}coppet.
\newblock Multifusion categories and finite semisimple 2-categories.
\newblock {\em Journal of Pure and Applied Algebra}, 226(8), 2022.
\newblock \arXiv{2012.15774}.

\bibitem[D{\'e}c23a]{Decoppet2021;compact}
Thibault~D. D{\'e}coppet.
\newblock Compact semisimple 2-categories.
\newblock {\em Transactions of the American Mathematical Society}, 376(12):8309–8336, 2023.
\newblock \arXiv{2111.09080}.

\bibitem[D{\'e}c23b]{Decoppet2022;Morita}
Thibault~D. D{\'e}coppet.
\newblock The {M}orita theory of fusion 2-categories.
\newblock {\em Higher Structures}, 7(1):234--292, 2023.
\newblock \arXiv{2208.08722}.

\bibitem[D{\'e}c23c]{Decoppet2023;dualizable}
Thibault~D. D{\'e}coppet.
\newblock On the dualizability of fusion 2-categories.
\newblock {\em To appear in \textit{Quantum Topology}}, 2023.
\newblock \arxiv{2311.16827}.

\bibitem[D{\'e}c23d]{Decoppet2022;rigid}
Thibault~D. D{\'e}coppet.
\newblock Rigid and separable algebras in fusion 2-categories.
\newblock {\em Advances in Mathematics}, 419:108967, 2023.
\newblock \arXiv{2205.06453}.

\bibitem[D{\'e}c24a]{decoppet2023:2Deligne}
Thibault~D. D{\'e}coppet.
\newblock The 2-{D}eligne tensor product.
\newblock {\em Kyoto Journal of Mathematics}, 64(1):1--29, 2024.
\newblock \arXiv{2103.16880}.

\bibitem[D{\'e}c24b]{decoppet:extension}
Thibault~D. D{\'e}coppet.
\newblock Extension theory and fermionic strongly fusion 2-categories (with an appendix joint with {Theo Johnson-Freyd}).
\newblock {\em Symmetry, Integrability and Geometry: Methods and Applications (SIGMA)}, 20:092, 2024.
\newblock \arXiv{2403.03211}.

\bibitem[Del02]{deligne2002}
Pierre Deligne.
\newblock Cat\'egories tensorielles.
\newblock {\em Moscow Mathematical Journal}, 2(2):227--248, 2002.

\bibitem[DGNO10]{DGNO2010braided}
Vladimir Drinfeld, Shlomo Gelaki, Dmitri Nikshych, and Victor Ostrik.
\newblock On braided fusion categories {I}.
\newblock {\em Selecta Mathematica}, 16:1--119, 2010.
\newblock \arXiv{0906.0620}.

\bibitem[DHJF{\etalchar{+}}]{DHJFPPRNY2}
Thibault~D. Décoppet, Peter Huston, Theo Johnson-Freyd, David Penneys, Julia Plavnik, David Reutter, Dmitri Nikshych, and Matthew Yu.
\newblock ({D}e-)equivariantization for higher categories.
\newblock In preparation.

\bibitem[DMNO13]{DMNO}
Alexei Davydov, Michael M{\"u}ger, Dmitri Nikshych, and Victor Ostrik.
\newblock The {W}itt group of non-degenerate braided fusion categories.
\newblock {\em Journal für die reine und angewandte Mathematik}, 2013(667), 2013.
\newblock \arXiv{1009.2117}.

\bibitem[DN21]{DN}
Alexei Davydov and Dmitri Nikshych.
\newblock Braided {P}icard groups and graded extensions of braided tensor categories.
\newblock {\em Selecta Mathematica}, 27(65), 2021.
\newblock \arXiv{2006.08022}.

\bibitem[DNO13]{DNO2013}
Alexei Davydov, Dmitri Nikshych, and Victor Ostrik.
\newblock On the structure of the {W}itt group of braided busion categories.
\newblock {\em Selecta Mathematica}, 19(1):237--269, 2013.
\newblock \arXiv{1109.5558 }.

\bibitem[DR18]{douglas2018fusion}
Christopher~L. Douglas and David~J. Reutter.
\newblock Fusion 2-categories and a state-sum invariant for 4-manifolds.
\newblock 2018.
\newblock \arXiv{1812.11933}.

\bibitem[DSPS20]{DSPS13}
Christopher~L. Douglas, Christopher Schommer-Pries, and Noah Snyder.
\newblock Dualizable tensor categories.
\newblock {\em Memoirs of the American Mathematical Society}, 268(1308):vii+88, 2020.
\newblock \arXiv{1312.7188}.

\bibitem[DT24]{Delcamp:2023kew}
Clement Delcamp and Apoorv Tiwari.
\newblock {Higher categorical symmetries and gauging in two-dimensional spin systems}.
\newblock {\em SciPost Physics}, 16:110, 2024.
\newblock \arXiv{2301.01259}.

\bibitem[DY23a]{DY;grouptheoretical}
Thibault~D. D{\'e}coppet and Matthew Yu.
\newblock Fiber 2-functors and {T}ambara--{Y}amagami fusion 2-categories.
\newblock 2023.
\newblock \arXiv{2306.08117}.

\bibitem[DY23b]{Decoppet:2022dnz}
Thibault~D. D\'ecoppet and Matthew Yu.
\newblock {Gauging noninvertible defects: a 2-categorical perspective}.
\newblock {\em Letters in Mathematical Physics}, 113(2):36, 2023.
\newblock \arxiv{2211.08436}.

\bibitem[EG14]{MR3167494}
David~E. Evans and Terry Gannon.
\newblock Near-group fusion categories and their doubles.
\newblock {\em Advances in Mathematics}, 255:586--640, 2014.
\newblock \arxiv{1208.1500}.

\bibitem[EGNO15]{EGNO}
Pavel Etingof, Shlomo Gelaki, Dmitri Nikshych, and Victor Ostrik.
\newblock {\em Tensor categories}, volume 205 of {\em Mathematical Surveys and Monographs}.
\newblock American Mathematical Society, Providence, RI, 2015.

\bibitem[EGO04]{etingof2004classification}
Pavel Etingof, Shlomo Gelaki, and Viktor Ostrik.
\newblock Classification of fusion categories of dimension {$pq$}.
\newblock {\em International Mathematics Research Notices}, 2004(57):3041--3056, 2004.
\newblock \arXiv{math/0304194}.

\bibitem[Elg07]{Elgueta}
Josep Elgueta.
\newblock Representation theory of 2-groups on {K}apranov and {V}oevodsky’s 2-vector spaces.
\newblock {\em Advances in Mathematics}, 213(1):53–92, 2007.
\newblock \arXiv{math/0408120}.

\bibitem[ENO05]{ENO}
Pavel Etingof, Dmitri Nikshych, and Viktor Ostrik.
\newblock On fusion categories.
\newblock {\em Annals of Mathematics (2)}, 162(2):581--642, 2005.
\newblock \arXiv{math/0203060}.

\bibitem[ENO10]{ENO2010}
Pavel Etingof, Dmitri Nikshych, and Victor Ostrik.
\newblock Fusion categories and homotopy theory.
\newblock {\em Quantum topology}, 1(3):209--273, 2010.
\newblock \arXiv{0909.3140}.

\bibitem[FFRS04]{KWdefect}
J\"urg Fr\"ohlich, J\"urgen Fuchs, Ingo Runkel, and Christoph Schweigert.
\newblock Kramers--{W}annier duality from conformal defects.
\newblock {\em Physical Review Letters}, 93:070601, Aug 2004.
\newblock \arXiv{cond-mat/0404051}.

\bibitem[FHJF{\etalchar{+}}24]{ferrer2024dagger}
Giovanni Ferrer, Brett Hungar, Theo Johnson-Freyd, Cameron Krulewski, Lukas M{\"u}ller, David Penneys, David Reutter, Claudia Scheimbauer, Luuk Stehouwer, Chetan Vuppulury, et~al.
\newblock Dagger $ n $-categories.
\newblock 2024.
\newblock \arxiv{2403.01651}.

\bibitem[GH15]{gepner2015enriched}
David Gepner and Rune Haugseng.
\newblock Enriched $\infty$-categories via non-symmetric $\infty$-operads.
\newblock {\em Advances in mathematics}, 279:575--716, 2015.
\newblock \arxiv{1312.3178}.

\bibitem[GMP{\etalchar{+}}23]{MR4598730}
Pinhas Grossman, Scott Morrison, David Penneys, Emily Peters, and Noah Snyder.
\newblock The {E}xtended {H}aagerup fusion categories.
\newblock {\em Annales Scientifiques de l'\'{E}cole Normale Sup\'{e}rieure. (4)}, 56(2):589--664, 2023.
\newblock \arxiv{1810.06076}.

\bibitem[GS18]{GwilliamScheimbauer}
Owen Gwilliam and Claudia Scheimbauer.
\newblock Duals and adjoints in higher {M}orita categories.
\newblock 2018.
\newblock \arXiv{1804.10924}.

\bibitem[Hau15]{1312.3881}
Rune Haugseng.
\newblock Rectification of enriched $\infty${\textendash}categories.
\newblock {\em Algebraic \& Geometric Topology}, 15(4):1931--1982, 2015.
\newblock \arXiv{1312.3881}.

\bibitem[Hau17]{haugseng2017higher}
Rune Haugseng.
\newblock The higher morita category of {$E_n$}--algebras.
\newblock {\em Geometry \& Topology}, 21(3):1631--1730, 2017.
\newblock \arXiv{1412.8459}.

\bibitem[Hei23]{heine}
Hadrian Heine.
\newblock An equivalence between enriched {$\infty$}-categories and {$\infty$}-categories with weak action.
\newblock {\em Advances in Mathematics}, 417:Paper No. 108941, 140, 2023.
\newblock \arXiv{2009.02428}.

\bibitem[Hin20]{hinich}
Vladimir Hinich.
\newblock {Yoneda} lemma for enriched $\infty$-categories.
\newblock {\em Advances in Mathematics}, 367:107129, 2020.

\bibitem[HL13]{hopkins2013ambidexterity}
Michael Hopkins and Jacob Lurie.
\newblock Ambidexterity in {$K(n)$}-local stable homotopy theory.
\newblock 2013.
\newblock Available at \url{https://people.math.harvard.edu/~lurie/papers/Ambidexterity.pdf}.

\bibitem[IO24]{Inamura:2023qzl}
Kansei Inamura and Kantaro Ohmori.
\newblock {Fusion surface models: 2+1d lattice models from fusion 2-categories}.
\newblock {\em SciPost Physics}, 16:143, 2024.
\newblock \arxiv{2305.05774}.

\bibitem[Izu01]{MR1832764}
Masaki Izumi.
\newblock The structure of sectors associated with {L}ongo--{R}ehren inclusions. {II}. {E}xamples.
\newblock {\em Reviews in Mathematical Physics}, 13(5):603--674, 2001.
\newblock \doi{10.1142/S0129055X01000818}.

\bibitem[JFR24]{JFR}
Theo Johnson-Freyd and David Reutter.
\newblock Minimal non-degenerate extensions.
\newblock {\em Journal of the American Mathematical Society}, 37(1):81--150, 2024.
\newblock \arXiv{2105.15167}.

\bibitem[JFS17]{JFS2017}
Theo Johnson-Freyd and Claudia Scheimbauer.
\newblock {(Op)lax natural transformations, twisted quantum field theories, and “even higher” {M}orita categories}.
\newblock {\em Advances in Mathematics}, 307:147--223, 2017.
\newblock \arXiv{1502.06526}.

\bibitem[JFY21]{JFY:2020ivj}
Theo Johnson-Freyd and Matthew Yu.
\newblock {Fusion 2-categories with no line operators are grouplike}.
\newblock {\em Bulletin of the Australian Mathematical Society}, 104(3):434--442, 2021.
\newblock \arxiv{2010.07950}.

\bibitem[JFY22]{JFY:2021tbq}
Theo Johnson-Freyd and Matthew Yu.
\newblock {Topological Orders in (4+1)-Dimensions}.
\newblock {\em SciPost Physics}, 13(3):068, 2022.
\newblock \arxiv{2104.04534}.

\bibitem[JMNR21]{JMNR2010rank}
Corey Jones, Scott Morrison, Dmitri Nikshych, and Eric~C. Rowell.
\newblock Rank-finiteness for {$G$}-crossed braided fusion categories.
\newblock {\em Transformation Groups}, 26:915–927, 2021.
\newblock \arXiv{1902.06165}.

\bibitem[JMPP22]{MR4498161}
Corey Jones, Scott Morrison, David Penneys, and Julia Plavnik.
\newblock Extension theory for braided-enriched fusion categories.
\newblock {\em International Mathematics Research Notices. IMRN}, (20):15632--15683, 2022.
\newblock \arxiv{1910.03178}.

\bibitem[JMS14]{MR3166042}
Vaughan F.~R. Jones, Scott Morrison, and Noah Snyder.
\newblock The classification of subfactors of index at most 5.
\newblock {\em Bulletin of the American Mathematical Society (N.S.)}, 51(2):277--327, 2014.
\newblock \arxiv{1304.6141}.

\bibitem[Jon83]{MR0696688}
Vaughan F.~R. Jones.
\newblock Index for subfactors.
\newblock {\em Inventiones Mathematicae}, 72(1):1--25, 1983.
\newblock \doi{10.1007/BF01389127}.

\bibitem[JSW24]{2410.08884}
Corey Jones, Kylan Schatz, and Dominic~J. Williamson.
\newblock Quantum cellular automata and categorical duality of spin chains.
\newblock \arXiv{2410.08884}, 2024.

\bibitem[Kam24]{Kammermeier}
Tessa Kammermeier.
\newblock Higher idempotent completion.
\newblock Master's thesis, University of Hamburg, Department of Mathematics, 2024.

\bibitem[LMGR{\etalchar{+}}24]{soergel}
Yu~Leon Liu, Aaron Mazel-Gee, David Reutter, Catharina Stroppel, and Paul Wedrich.
\newblock A braided $(\infty, 2)$-category of {Soergel} bimodules.
\newblock 2024.
\newblock \arXiv{2401.02956}.

\bibitem[Lur09]{highertopos}
Jacob Lurie.
\newblock {\em Higher Topos Theory}.
\newblock Annals of Mathematics Studies. Princeton Unviersity Press, 2009.

\bibitem[Lur17]{higheralgebra}
Jacob Lurie.
\newblock Higher algebra.
\newblock 2017.
\newblock \href{https://www.math.ias.edu/~lurie/papers/HA.pdf}{https://www.math.ias.edu/~lurie/papers/HA.pdf}.

\bibitem[Mar35]{MartyHypergroup}
Fr\'ed\'eric Marty.
\newblock Sur une g\'en\'eralisation de la notion de groupe.
\newblock In {\em Comptes rendus du huitieme Congres des mathematiciens scandinaves tenu a Stockholm 14-18 aout 1934}, pages 45--49, 1935.

\bibitem[MGS24]{mazel2021universal}
Aaron Mazel-Gee and Reuben Stern.
\newblock A universal characterization of noncommutative motives and cecondary algebraic {$K$}-theory.
\newblock {\em Annals of K-Theory}, 9:369--445, 2024.
\newblock \arXiv{2104.04021}.

\bibitem[ML98]{maclane}
Saunders Mac~Lane.
\newblock {\em Categories for the working mathematician}, volume~5 of {\em Graduate Texts in Mathematics}.
\newblock Springer-Verlag, New York, second edition, 1998.
\newblock \mathscinet{MR1712872}.

\bibitem[Nat18]{natale2018classification}
Sonia Natale.
\newblock On the classification of fusion categories.
\newblock In {\em Proceedings of the {I}nternational {C}ongress of {M}athematicians---{R}io de {J}aneiro 2018. {V}ol. {II}. {I}nvited lectures}, pages 173--200. World Scientific Publishing, Hackensack, NJ, 2018.

\bibitem[Ost03]{ostrik2003module}
Victor Ostrik.
\newblock Module categories, weak {H}opf algebras and modular invariants.
\newblock {\em Transformation groups}, 8:177--206, 2003.
\newblock \arXiv{math/0111139}.

\bibitem[RV16]{MR3415698}
Emily Riehl and Dominic Verity.
\newblock Homotopy coherent adjunctions and the formal theory of monads.
\newblock {\em Advances in Mathematics}, 286:802--888, 2016.
\newblock \arXiv{1310.8279}.

\bibitem[Sch15]{Scheimbauer2015}
Claudia Scheimbauer.
\newblock {\em Factorization Homology as a Fully Extended Topological Field Theory}.
\newblock Phd thesis, ETH Zurich, 2015.

\bibitem[Til98]{Tillmann}
Ulrike Tillmann.
\newblock {$\mathscr{S}$-Structures for $k$-linear categories and the definition of a modular functor}.
\newblock {\em Journal of the London Mathematical Society}, 58:208–228, 1998.
\newblock \arXiv{math/9802089}.

\bibitem[TW24]{MR4737369}
Ryan Thorngren and Yifan Wang.
\newblock Fusion category symmetry. {P}art {I}. {A}nomaly in-flow and gapped phases.
\newblock {\em Journal of High Energy Physics}, (4):Paper No. 132, 41, 2024.
\newblock \arxiv{1912.02817}.

\bibitem[VBW{\etalchar{+}}18]{PhysRevLett.121.177203}
Robijn Vanhove, Matthias Bal, Dominic~J. Williamson, Nick Bultinck, Jutho Haegeman, and Frank Verstraete.
\newblock Mapping topological to conformal field theories through strange correlators.
\newblock {\em Physical Review Letters}, 121:177203, Oct 2018.
\newblock \arxiv{1801.05959}.

\bibitem[Wen88]{MR936086}
Hans Wenzl.
\newblock Hecke algebras of type {$A_n$} and subfactors.
\newblock {\em Inventiones Mathematicae}, 92(2):349--383, 1988.
\newblock \doi{10.1007/BF01404457}.

\bibitem[Wen90]{MR1090432}
Hans Wenzl.
\newblock Quantum groups and subfactors of type {$B$}, {$C$}, and {$D$}.
\newblock {\em Communications in Mathematical Physics}, 133(2):383--432, 1990.
\newblock \doi{10.1007/BF02097374}.

\bibitem[WG18]{Wang:2017moj}
Qing-Rui Wang and Zheng-Cheng Gu.
\newblock Towards a complete classification of symmetry-protected topological phases for interacting fermions in three dimensions and a general group supercohomology theory.
\newblock {\em Physical Review X}, 8(1):011055, 2018.
\newblock \arXiv{1703.10937}.

\bibitem[Xu98]{MR1660937}
Feng Xu.
\newblock Standard {$\lambda$}-lattices from quantum groups.
\newblock {\em Inventiones Mathematicae}, 134(3):455--487, 1998.
\newblock \doi{10.1007/s002220050271}.

\end{thebibliography}
}}
\end{document}